\documentclass[12pt]{amsart}

\usepackage{contour}
\contourlength{1pt}
\usepackage{blindtext} 
\usepackage[linktocpage, pdfborder={0 0 .3}]{hyperref}

\usepackage{verbatim} 

\usepackage[dvipsnames]{xcolor}
\usepackage{fancybox, stmaryrd}
\usepackage{amsmath}
\usepackage{amssymb}
\usepackage{hyperref}
\usepackage{array, color}
\usepackage{overpic, float}
\usepackage{pict2e} 
\usepackage{setspace} 
\usepackage{cleveref}

\usepackage[normalem]{ulem}
\usepackage{xcolor}
\makeatletter
\def\squigred{\bgroup \markoverwith{\textcolor{red}{\lower3.5\p@\hbox{\sixly \char58}}}\ULon}
\makeatother

\makeatletter
\def\squigblue{\bgroup \markoverwith{\textcolor{blue}{\lower3.5\p@\hbox{\sixly \char58}}}\ULon}
\makeatother

\usepackage{xcolor}
\makeatletter
\def\uwave{\bgroup \markoverwith{\lower3.5\p@\hbox{\sixly \textcolor{red}{\char58}}}\ULon}
\font\sixly=lasy6 %
\makeatother

 \definecolor{calpolypomonagreen}{rgb}{0.12, 0.3, 0.17}
 
 \definecolor{darkbyzantium}{rgb}{0.6, 0.2, 0.3}
 \definecolor{azure}{rgb}{0.0, 0.5, 1.0}
 \definecolor{cittingcolor}{cmyk}{60,0,10,0}
  
 \hypersetup{pdfborder=0 0 .3}
	\hypersetup{
    colorlinks=true,
   linkcolor=darkbyzantium,
    filecolor=magenta,      
    urlcolor=NavyBlue,
    citecolor = calpolypomonagreen,
    pdftitle={Overleaf Example},
    pdfpagemode=FullScreen,
    }

\usepackage{mathtools}

\usepackage{tikz-cd}
\usepackage{tikz}
\usetikzlibrary{shapes.geometric, arrows}
\usetikzlibrary{fit,calc,positioning}

\usetikzlibrary{decorations.pathmorphing}

    \usepackage{lipsum} 
\usepackage[inner=2cm,outer=2cm,bottom=2cm]{geometry}

\newtheorem{lemma}{Lemma}[section]
\newtheorem{proposition}[lemma]{Proposition}
\newtheorem{theorem}[lemma]{Theorem}
\newtheorem{claim}[lemma]{Claim}

\newtheorem{question}[lemma]{Question}

\newtheorem{deflemma}[lemma]{Definition-Lemma}

\newtheorem{corollary}[lemma]{Corollary}
\newtheorem{definition}[lemma]{Definition}

\newtheorem{thm}{Theorem}

\newcommand{\R}{\mathbb{R}}
\newcommand{\C}{\mathbb{C}}

\renewcommand{\H}{\mathbb{H}}
\newcommand{\Z}{\mathbb{Z}}
\newcommand{\s}{\mathbb{S}}

\newcommand{\PP}{\mathsf{P}} 
\newcommand{\TT}{\mathsf{T}} 
\newcommand{\BB}{\mathsf{B}} 
 
\newcommand{\QD}{\mathsf{QD}} 
\newcommand{\QF}{\mathsf{QF}} 
\newcommand{\CC}{\mathsf{C}} 
\renewcommand{\AA}{\mathsf{A}}

\newcommand{\hH}{\mathbf{H}} 
 
\newcommand{\tT}{\mathbf{T}}

\newcommand{\cc}{\mathbf{c}}

\newcommand{\BBB}{\mathcal{B}} 
\newcommand{\LLL}{\mathcal{L}} 
\newcommand{\TTT}{\mathcal{T}} 
\newcommand{\RRR}{\mathcal{R}} 
\newcommand{\PPP}{\mathcal{P}} 
\newcommand{\QQQ}{\mathcal{Q}} 
\newcommand{\AAA}{\mathcal{A}} 
\newcommand{\VVV}{\mathcal{V}} 
\newcommand{\CCC}{\mathcal{C}} 
\newcommand{\FFF}{\mathcal{F}} 

\newcommand{\ML}{\mathsf{ML}}

\newcommand{\SL}{\mathrm{SL}}
\newcommand{\PML}{\mathsf{PML}}

\newcommand{\PSL}{{\rm PSL}}
\newcommand{\Isom}{{\rm Isom}}
\newcommand{\CP}{{\mathbb{C}{\rm  P}}}

\newcommand{\Hol}{\operatorname{Hol}}
\newcommand{\Gr}{\operatorname{Gr}}
\newcommand{\Ep}{\operatorname{Ep}}
\newcommand{\Area}{\operatorname{Area}}

\newcommand{\st}{\operatorname{st}}

\newcommand{\infi}{\infty}

\renewcommand{\Im}{{\rm Im}}
\newcommand{\Int}{{\rm int}}
\renewcommand{\Area}{{\rm Area}}

\renewcommand{\Gr}{\operatorname{Gr}}
\newcommand{\dev}{\operatorname{dev}}

\renewcommand{\cc}{{\bf c}}
\renewcommand{\tt}{{\bf t}}

\newcommand{\del}{\delta}
\newcommand{\ep}{\epsilon}

\newcommand{\gam}{\gamma}
\newcommand{\kap}{\kappa}

\newcommand{\Gam}{\Gamma}

\DeclareRobustCommand{\rchi}{{\mathpalette\irchi\relax}}
\newcommand{\irchi}[2]{\raisebox{\depth}{$#1\chi$}}

\newcommand{\minus}{\setminus}

\newcommand{\sub}{\subset}

\newcommand{\col}{\colon}
\newcommand{\bdr}{\partial}

\newcommand{\ti}{\tilde}
\newcommand{\ch}[1]{\check{#1}}

\newcommand{\ol}{\overline}

\newcommand{\length}{\operatorname{length}}

\newcommand{\Supp}{\operatorname{Supp}}

 \newcommand{\Label}[1]{\label{#1}\textcolor{green}{#1} }
 \renewcommand{\Label}[1]{\label{#1}}

\newcommand{\Qed}[1]{\nopagebreak[4]{\tiny \hfill
\fbox{\ref{#1}} \linebreak }\pagebreak[2]}

\usepackage{fancybox}

\definecolor{dblue}{cmyk}{1,.5, 0,.1}

\definecolor{bluep}{rgb}{0.2, 0.2, 0.6}

\newcommand{\tcr}{\textcolor{red}}

\newcommand\restr[2]{{
  \left.\kern-\nulldelimiterspace 
  #1 
  \vphantom{\big|} 
  \right|_{#2} 
  }}

\newcommand{\hide}[1]{}
\newcommand{\tcb}[1]{\textcolor{blue}{#1}}

\newenvironment{boxedlaw}[1]
  {\begin{Sbox}\begin{minipage}{#1}\centering}
  {\end{minipage}\end{Sbox}\begin{center}\fbox{\TheSbox}\end{center}}

\renewenvironment{boxedlaw}[1]{}

\usepackage{fancyhdr}
 \usepackage[many]{tcolorbox}    
\newtcolorbox{mybox}[1]{%
    tikznode boxed title,
    enhanced,
    arc=0mm,
    interior style={white},
    attach boxed title to top center= {yshift=-\tcboxedtitleheight/2},
    fonttitle=\bfseries,
    colbacktitle=white,coltitle=black,
    boxed title style={size=normal,colframe=white,boxrule=0pt},
    title={#1}}

\newcommand\redsout{\bgroup\markoverwith{\textcolor{red}{\rule[0.5ex]{2pt}{0.4pt}}}\ULon}

\begin{document}
 \title[\today]{Bers' simultaneous uniformization  and the intersection of  Poincar\'e holonomy varieties}

\author{Shinpei Baba}
\address{Osaka University }
\setcounter{tocdepth}{1}

\email{baba@math.sci.osaka-u.ac.jp}
\maketitle

\begin{center}
{\it Dedicated to Misha Kapovich on the occasion of his 60th birthday
}
\end{center}
\begin{abstract}
We consider the space of ordered pairs of distinct $\CP^1$-structures on Riemann surfaces (of any orientations) which have identical holonomy,
so that the quasi-Fuchsian space is identified with a connected component of this space.
This space holomorphically maps to the product of the Teichmüller spaces minus its diagonal. 

In this paper, we prove that this mapping is a complete local branched covering map.
As a corollary,  we reprove Bers' simultaneous uniformization theorem without any quasi-conformal deformation theory. 
Our main theorem is that the intersection of arbitrary two Poincaré holonomy varieties ($\operatorname{SL}_2\C$-opers) is a non-empty discrete set,  which is closely related to the mapping.
\end{abstract}

\tableofcontents

\section{Introduction}\Label{sIntro}
In 1960, Bers established a bijection between pairs of Riemann surface structures of opposite orientations and typical discrete and faithful representations of a surface group into $\PSL(2, \C)$ up to conjugacy (\cite{Bers60}).
It is called  {\it Bers' simultaneous uniformization theorem}, and it gave a foundation for the later evolutional development of the hyperbolic three-manifold theory by Thurston (\cite{Thurston-78}) and many others. 
In this paper, we partially generalize Bers' theorem, in a certain sense, to generic surface representations into $\PSL(2, \C)$, which are not necessarily discrete.

Throughout this paper, let $S$ be a closed orientable surface of genus $g > 1$. 
Given a {\it quasi-Fuchsian} representation $\rho\col \pi_1(S) \to \PSL(2, \C)$,  the domain of discontinuity is the union of disjoint topological open disks $\Omega^+, \Omega^-$ in $\CP^1$. 
Then, their quotients $\Omega^+/\Im \rho, \Omega^-/\Im \rho$ have marked Riemann surface structures with opposite orientations. 

Let $S^+, S^-$ be $S$ with opposite orientations.
Then Bers' simultaneous uniformization theorem asserts that this correspondence gives a biholomorphism \begin{eqnarray}
\mathsf{QF} &\to&  \TT \times \TT^\ast (= \R^{6g-6} \times\R^{6g-6} ) \Label{eqSimultaneous}
\end{eqnarray}
where $\QF$ is space of the quasi-Fuchsian representations  $\rho\col \pi_1(S) \to \PSL(2, \C)$ up to conjugation,   $\TT$ is the Teichm\"uller space of $S^+$  and $\TT^\ast$ is the Teichmüller space of $S^-$; see \cite{Hubbard06} \cite{EarleKra06} for the analyticity. 
 (Note that $\TT^\ast$ is indeed anti-holomorphic to $\TT$; see \cite{Wolpert10}. )

 The {\sf $\PSL(2, \C)$-character variety} of $S$ is the space of homomorphisms $\pi_1(S) \to \PSL(2, \C)$, roughly, up to conjugation, and it has two connected components  (\cite{Goldman-88t}). 
 Let $\rchi$ denote the component consisting of representations $\pi_1(S) \to \PSL(2,\C)$ which lift to $\pi_1(S) \to \SL(2, \C)$; then $\rchi$ strictly contains the (Euclidean) closure of $\QF$.

A {\sf $\CP^1$-structure} on $S$ is a locally homogeneous structure modeled on $\CP^1$, and its holonomy is in $\rchi$. 
The quotients $\Omega^+/\Im\rho$ and $\Omega^-/\Im \rho$ discussed above have not only Riemann surfaces structures but also  $\CP^1$-structures on $S^+$ and $S^-$, respectively. 
In fact, almost every representation in $\rchi$ is the holonomy of some $\CP^1$-structure on $S$ \cite{Gallo-Kapovich-Marden}; see \S \ref{sProjectiveStructures} for details. 

In fact, {each $\CP^1$-structure on $S$ corresponds to a holomorphic quadratic differential on a Riemann surface structure on $S$} (\S\ref{sSchwarzianParemeterization}).  
Let $\PP$ be the space all (marked) $\CP^1$-structures on $S^+$ with the fixed orientation, which is identified with the cotangent bundle of $\TT$. 
Similarly, let $\PP^\ast$ be the space of all marked $\CP^1$ on $S^-$, identified with the cotangent bundle of $\TT^\ast$.

By sending each quasi-Fuchsian representation $\rho\col \pi_1(S) \to \PSL(2, \C)$ to the $\CP^1$-structures $\Omega^+/\Im\rho$ and $\Omega^-/\Im \rho$, the quasi-Fuchsian space $\QF$ holomorphically embeds into $\PP \times \PP^\ast$ as a closed half-dimensional submanifold.
 The {\sf holonomy map}  $$\Hol \col \PP \sqcup \PP^\ast \to \rchi$$ takes each $\CP^1$-structure to its holonomy representation. 
Now we introduce the space of  all ordered pairs of distinct $\CP^1$-structures  sharing holonomy 
$$ \big\{(C, D) \in  (\PP \sqcup \PP^\ast )^2 \big\vert \Hol (C) = \Hol (D), C \neq D \big\}.$$  
Let us denote this space by $\BB$ for appreciation of the work of Bers.
Since $\Hol$ is locally biholomorphic,  $\BB$ is also a half-dimensional closed holomorphic submanifold. 
The map switching the order of $C$ and $D$ is a fixed-point-free biholomorphic involution of $\BB$.
Then,  the quasi-Fuchsian space $\QF$ is biholomorphically identified with two connected components of $\BB$, which are identified by this involution (Lemma \ref{OpenAndClosed}).
Every connected component of $ (\PP \sqcup \PP^\ast )^2$ contains at least one component of $\BB$ which does not correspond to $\QF$ (see Lemma \ref{ComponentsOfProjectiveQuasifuchsianSpace}).

Let $\psi\col \PP \sqcup \PP^\ast \to  \TT \sqcup \TT^\ast$ be the projection from the space of all $\CP^1$-structures on $S^+$ and $S^-$ to the space of all Riemann surface structures on $S^+$ and $S^-$. 
Define $\Psi \col \BB \to (\TT \sqcup \TT^\ast)^2 \minus \Delta$ by $\Psi(C, D) = (\psi(C), \psi(D))$, where $\Delta$ is the diagonal $\{ (X, X)  \mid X \in \TT \sqcup \TT^\ast\}$ (which can not intersect $\Psi(B)$).

It is a natural question to ask to what extent connected components of $\BB$ resemble the quasi-Fuchsian space $\QF$.
In this paper, we prove a local and a global property of the holomorphic map $\Psi$: 
\begin{boxedlaw}{13cm}
\begin{thm}\Label{GeneralizedQF}
The map $\Psi$ is a complete local branched covering map. 
\end{thm}
\end{boxedlaw}
(For the definition of complete local branched covering maps, see \S \ref{sComplexGeometry}.)
In particular, $\Psi$ is open, and its fibers are discrete subsets of $\BB$.
Thus its ramification locus is a nowhere-dense analytic subset,  which may possibly be the empty set. 
 (The completeness of Theorem \ref{GeneralizedQF} is given by Theorem \ref{Surjectivity}, and  the local property   by Theorem \ref{Locally a branched covering} below.)

Note that, by the completeness in \Cref{GeneralizedQF}, for every connected component  $Q$ of $\BB$, the restriction $\Psi | Q$ is surjective onto its corresponding component of $(\TT \sqcup \TT^\ast)^2 \minus \Delta$.
We also show that, towards the diagonal $\Delta$, the holonomy of $\CP^1$-structures leaves every compact set in $\rchi$  (see Proposition \ref{NearDiagonal}).  

The deformation theory of hyperbolic cone manifolds is developed, especially, by Hodgson, Kerckhoff and Bromberg  \cite{HodgsonKerckhoff98RigidityOfHyperbolicConeManifolds, HodgsonKerckhoff05UniversalBoundsForDehnSurgery, HodgsonKerckhoff08, Bromberg04HyperbolicConeManifoldsShortGeodesics, Bromberg04RigidityofHyperbolicConeManifolds}).
If cone angles exceed $2\pi$,   their deformation theory is established only under the assumption that the cone singularity is short and, thus, the tube radius is large.
More generally, a conjecture of McMullen  (\cite[Conjecture 8.1]{McMullen-98}) asserts that the deformation space of geometrically-finite hyperbolic cone-manifolds is parametrized by using the cone angles and the conformal structures on the ideal boundary.
Theorem  \ref{GeneralizedQF} provides some additional evidence for the conjecture, when the cone angles are $2\pi$-multiples (c.f. \cite{Bromberg-07}).

 Bers' simultaneous uniformization theorem is a consequence of the measurable Riemann mapping theorem. 
 It thus is important that the domain $\Omega^+ \sqcup \Omega^-$ is a (full measure) subset of $\CP^1$. 
However, in general,  developing maps of $\CP^1$-structures are not embeddings,  and Bers' proof does not apply to the other components of $\BB$. 
In fact, Theorem \ref{GeneralizedQF} implies the simultaneous uniformization theorem  (\S \ref{SimultaneousUniformization}). 
Thus we reprove Bers' theorem genuinely from a complex analytic viewpoint, 
without any quasi-conformal deformation theory. 

Next we describe the local property in Theorem \ref{GeneralizedQF}.
Since $\Hol$ is locally biholomorphic, for every $(C, D) \in \BB$, 
if an open neighborhood $V$ of $(C, D)$ in $\BB$ is sufficiently small, then $\Hol$ embeds $V$ onto a neighborhood $U$ of  $\Hol (C) = \Hol(D)$ in $\rchi$. 
Let $\TT_C$ and $\TT_D$  be $\TT$ or $\TT^\ast$ so that $\psi(C) \in \TT_C$ and $\psi(D) \in \TT_D$, and  define a holomorphic map  $\Psi_{C, D} \col U \to \TT_C \times \TT_D$ by the restriction of $\Psi$ to $V$ and the identification $V \cong U$.
The following gives a finite-to-one  ``parametrization'' of $U$  by pairs of Riemann surface structures associated with $V$. 

\begin{boxedlaw}{13cm}
\begin{thm}\Label{Locally a branched covering}
Let $(C, D) \in \BB$. 
Then, there is a neighborhood $V$ of $(C, D)$ in $\BB$, such that 
 $\Hol$ embeds $V$ into $\rchi$, and 
 the restriction of $\Psi$ to $V$  is a branched covering map onto its image in $T_C \times T_D$ 
 (Theorem  \ref{LocalUniformization}.)
 \end{thm}  
\end{boxedlaw}

By the simultaneous uniformization theorem,
for every $X \in\TT^\ast$ and $Y \in \TT$,  
 the slices $\TT \times \{Y\}$ and $\{X\} \times \TT^\ast$, called the {\sf Bers' slices}, intersect transversally in the point in $\QF$ corresponding to $(X, Y)$ by (\ref{eqSimultaneous}).
The Teichmüller spaces $\TT$ and $\TT^\ast$ are, as complex manifolds, open bounded pseudo-convex domains in $\C^{3g-3}$, where $g$ is the genus of $S$.
In order to prove Theorem \ref{GeneralizedQF} and Theorem \ref{Locally a branched covering}, we consider the analytic extensions of   $\TT \times \{Y\}$ and $\{X\} \times \TT^\ast$  in the character variety $\rchi$ and analyze their intersection.

For each $X \in \TT \sqcup \TT^\ast$, let $\PP_X$ be the space of all $\CP^1$-structures on $X$. 
Then $\PP_X$ is an affine space of holomorphic quadratic differentials on $X$, and thus $\PP_X \cong \C^{3g-3}$. 
Although the restrictions of the holonomy map $\Hol$ to $\PP$ and $\PP^\ast$ are  non-proper and non-injective,  the restriction of $\Hol$ to $\PP_X$ is a  proper embedding  (\cite{Poincare884, Gallo-Kapovich-Marden}, see also \cite{Tanigawa99,  Kapovich-95,Dumas18HolonomyLimit}).
Let $\rchi_X = \Hol(\PP_X)$, which we shall call the {\sf Poincar\'e holonomy variety} of $X$ as its injectivity is due to Poincar\'e.
Note that, if $X \in \TT$, then $\rchi_X$  contains $\{X\} \times \TT^\ast$ as a bounded pseudo-convex subset, and similarly, if $Y \in \TT^\ast$, then $\rchi_Y$ contains $\TT \times \{Y\}$ as a bounded open subset. 

The intersection theory of subvarieties and submanifolds in the character variety $\rchi$ has been important (\cite{Dumas15, Dumas-Wolf08} \cite[Theorem 12]{Faltings83}). 
Since   $\dim \rchi_X$ is half of  $\dim \rchi$,   it is a basic question to ask what the intersection of such smooth subvarieties looks like. 

\begin{boxedlaw}{13cm}
\begin{thm}\Label{DiscreteIntersection}
For all distinct $X, Y$ in $\TT \sqcup \TT^\ast$,
the intersection of  $\rchi_X$  and $\rchi_Y$ is a non-empty discrete set. 
\end{thm}
\end{boxedlaw}
More precisely, we will show that $\rchi_X \cap \rchi_Y$ contains at least one point if the orientations of $X$ and $Y$ are the same, and at least two points if the orientations are opposite (Corollary \ref{Cardinality}). 
Such a  global understanding of $\rchi_X \cap \rchi_Y$ in Theorem \ref{DiscreteIntersection} is completely new. 
In fact, much of this paper is devoted to proving the discreteness of $\rchi_X \cap \rchi_Y$.

    The deformation spaces, $\PP$ and $\PP^\ast$, of $\CP^1$-structures have two distinguished parametrizations: namely,  {\it Schwarzian parametrization} (\S \ref{sSchwarzianParemeterization}) and {\it Thurston parametrization} (\S \ref{sThurstonParameters}). 
In order to understand points in $\rchi_X \cap \rchi_Y$, we give a comparison theorem between those two parametrizations.

Let $C$ be a $\CP^1$-structure on a Riemann surface $X$.
Then the quadratic differential of its Schwarzian parameters gives a vertical measured (singular) foliation $V$ on $X$. 
The Thurston parametrization of $C$ gives the measured geodesic lamination $L$ on the hyperbolic surface. 
Dumas showed that $V$ and $L$ {\it projectively} coincide in the limit as $C$ leaves every compact set in $P_X$ (\cite{Dumas06, Dumas07}), see also  \cite{OttSwobodaWendworthWolf(20)}.)
 
The measured geodesic lamination $L$ of the Thurston parameter is also realized as a  circular measured lamination $\LLL$ on $C$, so that  $\LLL$ and $L$ are the same measured lamination on $S$ (\S \ref{sCollapsing}).
  In this paper, we prove a more explicit asymptotic relation between the Thurston lamination $\LLL$ and the vertical foliation $V$, without projectivization.
For a quadratic differential $q = \phi\, dz^2$ on a Riemann surface $X$,  let  $\| q \| = \int_X |\phi|\, dx\,dy$, the $L^1$-norm. 
Then we have the following.
\begin{boxedlaw}{13cm}
\begin{thm}\Label{FoliationLaminationComparison}
Let $X \in \TT \sqcup \TT^\ast$. 
For every $\ep > 0$, there is $r > 0$, such that, if the holomorphic quadratic differential $q $ on $X$ satisfies $\| q \| > r$, then, letting $C$ be the $\CP^1$-structure on $X$ given by  $q$,  the vertical foliation $V$ of $q$ is $(1 + \ep, \ep)$-quasi-isometric to $\sqrt{2}$ times the Thurston lamination $\LLL$ on $C$, up to an isotopy of $X$ supported on the $\ep$-neighborhood of the zero set of $q$ in the uniformizing hyperbolic metric on $X$. 
(Theorem \ref{ThurstonLaminationAndVerticalFoliationOnDisks}.)
\end{thm}
\end{boxedlaw}
(See \ref{sComoparingMeasuredFoliations} for the definition of being quasi-isometric, and see \S \ref{sCollapsing} for the Thurston lamination on a $\CP^1$-surface.)
\Cref{FoliationLaminationComparison} is reminiscent of the (refined) estimates of high energy harmonic maps between hyperbolic Riemann surfaces by Wolf (\cite{Wolf91}).

  Last we address that,  in our setting, a variation of McMullen's conjecture can be stated in a global manner:
\begin{question}
For every (or even some) non-quasi-Fuchsian component $Q$ of $\BB$,  is the restriction of $\Phi$ to $Q$ a biholomorphic mapping onto its corresponding component of $(\TT \sqcup \TT^\ast)^2$?
\end{question}

\subsection{Outline of this paper}
In \S \ref{sEpsteinSurfaces}, we analyze the geometry of Epstein-Schwarz surfaces corresponding to  $\CP^1$-structures, using \cite{Dumas18HolonomyLimit} and \cite{Baba-15}.
In \S\ref{sComoparingMeasuredFoliations}, we analyze the horizontal foliations of $\CP^1$-structures on $X$ and $Y$ corresponding to the intersection points of $\rchi_X \cap \rchi_Y$ in Theorem \ref{DiscreteIntersection}. 
In fact, we show that such horizontal projectivized measured foliations projectively coincide towards infinity of $\rchi_X \cap  \rchi_Y$ (Theorem \ref{HorizontalFoliationsCoinside}).  

A (fat) train-track is a surface obtained by identifying edges of rectangles in a certain manner. 
In \S \ref{sEuclideanTrainTracks}, we introduce more general train-tracks whose branches are not necessarily rectangles but more general polygons,  cylinders, and even surfaces with staircase boundary ({\it surface train tracks}).  
In \S \ref{sCompatibleTraintrackDecompositions}, given a certain pair of flat surfaces, we decompose them into the surface train tracks in a compatible manner.  

In \S \ref{sLaminationAndFoliations}, we prove Theorem \ref{FoliationLaminationComparison}. 
In \S \ref{sCircularTraintracks}, for every holonomy $\rho$  in  $\rchi_X \cap \rchi_Y$ outside a large compact subset $K$ of $\rchi$,   we construct certain surface train-track decompositions of $\CP^1$-structures on $X$ and $Y$ corresponding to $\rho$ in a compatible manner, using the decomposition of flat surfaces.
 In \S \ref{sGraftingCocycle},  from the compatible decompositions of the $\CP^1$-structures,  we construct an integer-valued cocycle which changes continuously in $\rho  \in \rchi_X \cap \rchi_Y \minus K$. 
In \S\ref{sDiscreteness}, by this cocycle and some complex geometry,  we prove the discreteness in Theorem \ref{DiscreteIntersection}. 
 In \S \ref{sCompleteness},  the completeness of Theorem \ref{DiscreteIntersection} is proven. 
In \S \ref{sOppositeOrientation}, we discuss the case when the orientations of $X$ and $Y$ are opposite. 
In \S \ref{SimultaneousUniformization}, we give a new proof of Bers' theorem.

 \subsection{Acknowledgements}
I thank Ken Bromberg, David Dumas, Misha Kapovich, and Shinnosuke Okawa for their helpful conversations. 
I appreciate the anonymous referee for many valuable comments and for pointing out some important inaccuracies in the original version.

The author was partially supported by the Grant-in-Aid for Scientific Research (20K03610).

\section{Preliminaries}\Label{sPreliminaries}

\subsection{$\CP^1$-structures} (General references are \cite{Dumas-08}, \cite[\S 7]{Kapovich-01}.)\Label{sProjectiveStructures}
Let $F$ be a connected orientable surface.  
A  {\sf $\CP^1$-structure} on $F$ is a $(\CP^1, \PSL(2, \C))$-structure. 
That is, a maximal atlas of charts embedding open sets of $F$ into $\CP^1$ with transition maps in $\PSL(2, \C)$. 
Let $\ti{F}$ be the universal cover of $F$. 
Then, equivalently, a $\CP^1$-structure is a pair of
\begin{itemize}
\item a local homeomorphism $f\col \ti{F} \to \CP^1$ and
\item  a homomorphism  $\rho\col \pi_1(S) \to \PSL(2,\C)$
\end{itemize}
such that $f$ is $\rho$-equivariant (\cite{Thurston-97}).
It is defined up to an isotopy of the surface and an element $\alpha$ of $\PSL(2, \C)$, i.e. $(f, \rho) \sim ( \alpha f,  \alpha^{-1}\rho \alpha )$. 
The local homeomorphism $f$ is called the {\sf developing map} and the homomorphism $\rho$ is called the {\sf holonomy representation} of a $\CP^1$-structure. 
We also write the developing map of $C$ by {\sf $\dev C$}.

\subsubsection{The holonomy map}

The $\PSL(2, \C)$-{\sf character variety} of $S$ is the space of the equivalence classes homomorphisms
$$\{ \pi_1(S) \to \PSL(2, \C))\}\sslash\PSL(2, \C),$$
where the quotient is the GIT-quotient (see \cite{Newstead(06)} for example).
 For the holonomy representations of $\CP^1$-structures on $S$, the quotient is exactly given by the conjugation by $\PSL(2, \C)$. 
Then, the character variety has exactly two connected components, distinguished by the lifting property to $\SL(2, \C)$; see \cite{Goldman-88t}. 
 Let $\rchi$ be the component consisting of representations which lift to $\pi_1(S) \to \SL(2, \C)$, and
let $\PP$ be the space of marked $\CP^1$-structures on $S$. 
Then the {\sf holonomy map} $$\Hol \col \PP  \to \rchi$$
takes each $\CP^1$-structure to its holonomy representation.
Then $\Hol$ is a locally biholomorphic map, but not a covering map onto its image (\cite{Hejhal-75, Hubbard-81, Earle-81}).
By Gallo, Kapovich, and Marden (\cite{Gallo-Kapovich-Marden}), $\rho \in \Im \Hol$ if and only if $\rho$ is non-elementary and $\rho$ has a lift to $\pi_1(S) \to \SL(2,\C)$.
In particular, $\Hol$ is almost onto $\rchi$.

\subsubsection{The Schwarzian parametrization}\Label{sSchwarzianParemeterization} (See \cite{Dumas-08} \cite{Lehto87}.)
Let $X$ be a Riemann surface structure on $S$. 
Then, the hyperbolic structure $\tau_X$ uniformizing $X$ is, in particular, a $\CP^1$-structure on $X$. 
For an arbitrary $\CP^1$-structure $C$ on $X$, the Schwarzian derivative gives a holomorphic quadratic differential on $X$ by comparing with $\tau_X$, so that $\tau_X$ corresponds to the zero differential. 
Then $(X, q)$ is the {\sf Schwarzian parameters} of $C$.
Let $\QD(X)$ be the space of the holomorphic quadratic differentials on $X$, which is a complex vector space of dimension $3g-3$.
 Thus, the space $\PP_X$ of all $\CP^1$ structures on $X$ is identified with $\QD(X)$.

 \begin{theorem}[\cite{Poincare884, Kapovich-95}, see also \cite{Tanigawa99, Dumas18HolonomyLimit}]\Label{HolonomyVariety}
 For every Riemann surface structure $X$ on $S$, 
the set $\PP_X$ of projective structures on $X$ is property embedded in $\rchi$ by $\Hol$.
\end{theorem} 
  For $X \in \TT \sqcup \TT^\ast$, let $\rchi_X$ denote the smooth analytic subvariety $\Hol(\rchi_X)$. 
    Pick any metric $d$ on $\TT$ and $\TT^\ast$ compatible with their topology (for example, the Teichmüller metric or the Weil-Peterson metric). 
\begin{lemma}\Label{LocalProductStructure}
Let $B$ be an arbitrary bounded subset of either $\TT$ or $\TT^\ast$.
For every compact subset $K$ in $\rchi$, there is $\ep > 0$, such that, if distinct $X, Y \in B$ satisfy $d(X, Y) < \ep$, then $\rchi_X \cap \rchi_Y \cap K = \emptyset$.
\end{lemma}
\begin{proof}
For each $X \in  \TT \sqcup \TT^\ast$,   by Theorem \ref{HolonomyVariety}, ~$\PP_X$ is properly embedded in $\rchi$. 
For a neighborhood $U$ of $X$, let $D_r(U)$ denote the set of all holomorphic quadratic differentials $q$ on Riemann surfaces $Y$ in $ U$ such that  the $L^1$-norm $\| q \|$ is  less than $r$. 
Since $\Hol$ is a local biholomorphism,  for every $X \in \TT \sqcup \TT^\ast$ and $r \in \R_{>0}$, there is a neighborhood $U$ of $X$,  $\Hol$ embeds $D_r(U)$ into $\rchi$. 
Let $\PP_U$ be the space of all $\CP^1$-structures whose complex structures are in $U$. 
Then, if $r > 0$ is sufficiently large, we can, in addition, assume that $K \cap \Hol (\PP_U) = K \cap \Hol (D_r(U))$. 
Therefore, for all $Y, W \in U$, we have $\rchi_Y \cap \rchi_W \cap K = \emptyset$.
\end{proof}

  \subsubsection{Singular Euclidean structures}(See \cite{Strebel84QuadraticDifferentials}, \cite{FarbMarglit12}.)
 Let $q = \phi \,dz^2$ be a quadratic differential on a Riemann surface $X$.  
 Then $q$ induces a singular Euclidean structure $E$ on $S$ from the Euclidean structure on $\C$: Namely, for  each non-singular point $z \in X$, we can identify a neighborhood $U_z$ of $z$ with an open subset of $\C \cong \mathbb{E}^2$  by the  integral
$$\eta(w) = \int_z^w \sqrt{\phi}\, dz$$ along a path connecting $z$ and $w$, where  $w \in U_z$ is a  fixed base point
 (for details, see \cite{Strebel84QuadraticDifferentials}). 
 Then the zeros of $q$ correspond to the singular points of $E$. 
Note that, for $r > 0$,  if the differential $q$ is scaled by $r$, then the Euclidean metric $E$ is scaled by $\sqrt{r}$. 
Let $E^1$ denote the normalization $\frac{E}{ \Area E}$ of $E$ by the area. 
   
 The complex plane $\C$ is foliated by horizontal lines and, by the identification $\C = \mathbb{E}^2$, the vertical length $dy$ gives a canonical transversal measure to the foliation.  
 Similarly, $\C$ is also foliated by the vertical lines, and the horizontal length $d x$ gives a canonical transversal measure to the foliation.  
Then, those vertical and horizontal foliations on $\C$ induce vertical and horizontal singular foliations on  $E$ which meet orthogonally. 

In this paper, a {\sf flat surface} is the singular Euclidean structure obtained by a quadratic differential on a Riemann surface, which has vertical and horizontal foliations.

    \subsubsection{Measured laminations} ( See \cite{Thurston-78, Epstein-Marden-87} for details)
 Let $\sigma$ be a hyperbolic structure on the closed surface $S$. 
    A {\it geodesic lamination} on $\sigma$ is a set of disjoint geodesics whose union is a closed subset of $S$. 
    A {\it measured (geodesic) lamination} $L$ on $\sigma$ is a pair of a geodesic lamination and its transversal measure. 
    In this paper, for an arc $\alpha$ on $\sigma$ transversal to $L$, we denote, by $L(\alpha)$, the transversal measure of $\alpha$ given by $L$. 
   If we take a different hyperbolic structure $\sigma'$ on $S$, there is a unique geodesic representative on $L$ on $\sigma'$. 
   We thus can define measured laminations without fixing a specific hyperbolic structure on $S$. 

\subsubsection{Bending a geodesic in the hyperbolic three-space}
The following well-known lemma describes a closeness of a geodesic and a piecewise geodesic in $\H^3$ with a small amount of bending.
\begin{lemma}( \cite[Theorem I.4.2.10]{CanaryEpsteinGreen84} )\Label{ABitBentGeodesic}
Let $c\col [0,\ell] \to \H^3$ be a piecewise geodesic parametrized by arc length. 
Let $s(t)$ be the geodesic segment in $\H^3$ connecting $c(0)$ to $c(t)$.
Let $\theta(t)$ be the angle between the forward tangent vector of $c$ at $t$ and the forward tangent vector of $s(t)$ at $c(t)$. 

For every $\ep > 0$ and $r > 0$, there is $\del > 0$ such that, if each smooth geodesic segment of $c$ has length at least $r$ and the exterior angle of $c$ at every singular point of $c$ is less than $\del$, then $\theta(t) < \ep$ for all $t \in [0, \ell]$.
\end{lemma}

\subsubsection{Thurston's parameterization}\Label{sThurstonParameters}
By the uniformization theorem of Riemann surfaces, 
the space of all marked hyperbolic structures on $S$ is identified with the space  $\TT$ of all marked Riemann surface structures. 
Let $\ML$ be the space of measured laminations on $S$.  
Note that $\CP^1$ is the ideal boundary of $\H^3$, so that $\operatorname{Aut} \CP^1 = \Isom^+ \H^3$.
In fact, Thurston gave a parameterization of $\PP$ using the three-dimensional hyperbolic geometry.
\begin{theorem}[Thurston, see \cite{Kullkani-Pinkall-94, Kamishima-Tan-92}]\Label{ThurstonParametrization}
There is a natural (tangential) homeomorphism
$$\PP \to \TT \times \ML.$$
\end{theorem}
Suppose that, by this homeomorphism,  $C = (f, \rho) \in \PP$ corresponds to  a pair $(\sigma, L) \in \TT \times \ML$.
Let $\ti{L}$ be the $\pi_1(S)$-invariant measured lamination on $\H^2$ obtained by lifting $L$.
Then $(\sigma, L)$ yields a $\rho$-equivariant pleated surface $\beta\col \H^2 \to \H^3$, obtained by bending $\H^2$ along  $\ti{L}$ by the angles given by its transversal measure. 
The map $\beta$ is called a {\bf bending map}, and it is unique up to post-composing with  $\PSL(2, \C)$.

\subsubsection{Collapsing maps}\Label{sCollapsing}(\cite{Kullkani-Pinkall-94}; see also \cite{Baba20ThurstonParameter}.)
Let $C \cong (\tau, L)$ be a $\CP^1$-structure expressed in Thurston parameters. 
Let $\ti{C}$ be the universal cover of $C$. 
Then $\ti{C}$ can be regarded as the domain of $f$, so that $\ti{C}$ is holomorphically immersed in $\CP^1$.
  A {\sf round disk} is a topological open disk whose development is a round disk in $\CP^1$, and a {\sf maximal disk} is a round disk which is {\it not} contained in a strictly bigger round disk. 
 In fact,  for all $z \in \ti{C}$,  there is a unique maximal disk $D_z$ whose core contains $z$.
 Then there is a measured lamination $\LLL$ on $C$ obtained from the cores of maximal disks in the universal cover $\ti{C}$, such that $\LLL$ is equivalent to $L$ in $\ML$. 
This lamination is the {\sf Thurston lamination} on $C$. 
In addition,  there is an associated continuous map 
$\kap\col C \to \tau$ which takes $\LLL$ to $L$, called {\sf the collapsing map}.\

Then,  the bending map and the developing of $C$ are related by the collapsing map $\kap$ and appropriate nearest point projections in $\H^3$: 
Let $\ti\kap\col \ti{C} \to \H^2$ be the lift of $\kap$ to a map between the universal covers.
Let $H_z$ be the hyperbolic plane in $\H^3$ bounded by the boundary circle of $D_z$.
There is a unique nearest point projection from $D_z$ to $H_z$.
Then $\beta \circ \ti\kap(z)$ is the nearest point projection of $f(z)$ to $H_z$.

\subsection{Bers' space}
Recall, from \S \ref{sIntro}, that  $\BB$ is the space of ordered pairs of $\CP^1$-structures on $S$  with identical holonomy, which may have different orientations.

\begin{lemma}\Label{ComponentsOfProjectiveQuasifuchsianSpace} Every component of $(\PP \sqcup \PP^\ast)^2$
 contains, at least, one connected component of $\BB$ which is not identified with the quasi-Fuchsian space.
\end{lemma}
\begin{proof}
By \cite{Gallo-Kapovich-Marden}, 
every non-elementary representation $\rho\col \pi_1(S) \to \SL(2, \C)$ is the holonomy representation of 
infinitely many $\CP^1$-structures on $S^+$ whose developing maps are not embedding, and also of
infinitely many $\CP^1$-structures of $S^-$ whose developing maps are not embedding.  
Therefore,
since a quasi-Fuchsian component of $\BB$ consists of pairs of $\CP^1$-structures whose developing maps are embedding, 
every component of $(\PP \sqcup \PP^\ast)^2$ contains at least one connected component of $\BB$, which is not a quasi-Fuchsian component. 
\end{proof}

\begin{lemma}\Label{ProjectiveQuasiFuchsianManifold}
$\BB$ is a closed analytic submanifold of $\PP \sqcup \PP^\ast$ of complex dimension $6g -6$. 
\end{lemma}
\begin{proof}
It is a holomorphic submanifold, since $\Hol\col \PP \sqcup \PP^\ast \to \rchi$ is a local biholomorphism.
As $\dim_\C \rchi = 6g -6$, the complex dimension of $\BB$ is also $6g -6$. 
Let $(C_i, D_i)$ be a sequence in $\BB$ converging to $(C, D)$ in $(\PP \sqcup \PP^\ast)^2$.
Then, since $\Hol C_i  = \Hol D_i$, by the continuity of $\Hol$, $\Hol (C) = \Hol (D)$. 
Therefore $\BB$ is closed. 
\end{proof}

 \subsection{Angles between laminations}\Label{sAngles}
 Let $F$ be a surface with a hyperbolic or singular Euclidean metric. 
 Let $\ell_1, \ell_2$ be (non-oriented) geodesics on $F$ with non-empty intersection. 
 Then, for $p \in \ell_1 \cap \ell_2$, let $\angle_p(\ell_1, \ell_2) \in [0, \pi/2]$ denote the {\sf angle} between $\ell_1$ and $\ell_2$ at $p$.
 
Let $L_1$ $L_2$ be geodesic laminations or foliations on $F$.
Then 
 $\angle(L_1, L_2)$ be the infimum of  $\angle_p(\ell_1, \ell_2) \in [0, \pi/2]$ over all $p\in L_1 \cap L_2$ where $\ell_1$ and $\ell_2$ are leaves of $L_1$ and $L_2$, respectively, containing $p$.
By convention, if $L_1 \cap L_2 = \emptyset$, then  $\angle(L_1, L_2) = 0$. 
 We say that $L_1$ and $L_2$ are {\sf $\ep$-parallel}, if $\angle(L_1, L_2) < \epsilon$.

\subsection{The Morgan-Shalen compactification}(See  \cite{Culler-Shalen-83, MorganShalen84}, see also \cite[\S 10.3]{Kapovich-01}.)
The Morgan-Shalen compactification  is a compactification of $\PSL(2, \C)$-character variety, introduced in \cite{Culler-Shalen-83, MorganShalen84}.
For our $\rchi$, each boundary point corresponds to a minimal action of  $\pi_1(S)$ on a $\R$-tree, $\pi_1(S) \curvearrowright T$.

Every holonomy $\rho\col \pi_1(S) \to \PSL(2, \C)$ induces a translation length function $\rho^\ast\col\pi_1(S) \to \R_{\geq 0}$, and a minimal action $\pi_1(S)$ on a $\R$-tree also induces a translation length function. 
Then $\rho_i \in \rchi$ converges to a boundary point $\pi_1(S) \curvearrowright T$ if the length function  $\rho^\ast_i$ projectively converges to the projective class of the translation function of $\pi_1(S) \curvearrowright T$ as $i \to \infty$.

\subsection{Complex geometry}\Label{sComplexGeometry}
We recall some basic complex geometry used in this paper. 
Let $U, W$ be complex manifolds of the same dimension. 
A holomorphic map $\phi \col U \to W$ is a {\sf (finite) branched covering map} if 
\begin{itemize}
\item there are closed analytic subsets $U’, W'$ of dimensions strictly smaller than $\dim U = \dim W$, such that the restriction of $\phi$ to $U \minus U'$ is a covering map onto  $W \minus W'$, and
\item its covering degree is finite. (See \cite[p227]{Fritzsche_Hans_02_HolomorphicFunactionsComplexManifolds}.)
\end{itemize}
A holomorphic map $\phi\col U \to W$ is a {\sf local branched covering map} if, for every $z \in U$, there is a neighborhood $V$ of $z$ in $U$ such that the restriction $\phi | V$ is a  branched covering map onto its image. 
A holomorphic map $U \to W$  is {\sf complete} if it has the (not necessarily unique) path lifting property (\cite{AhlforsSario60RiemannSurfaces}).

 Let $U$ be an open subset of $\C^n$. 
Then a subset $V$ of $U$ is {\it analytic} if it is locally an intersection of zeros of finitely many holomorphic functions. 
\begin{proposition}[Proposition 6.1 in \cite{Fritzsche_Hans_02_HolomorphicFunactionsComplexManifolds}]\Label{BoundedAnalyticSet}
Every connected bounded analytic set in $\C^n$ is a  discrete set. 
\end{proposition}

\begin{theorem}[p107 in \cite{Grauert-Remmert_Coherent-analytic-sheaves}, Theorem 7.9  in \cite{Hu_Chung-Chun_Differentiable-and-complex-dynamics}]\Label{DiscreteFiberThenOpen}
Let $U \sub   \C^n$ be a region. 
Suppose that $f\col U  \to \C^n$ is a holomorphic map with discrete fibers. 
Then it is an open map. 
\end{theorem}

\section{Approximations of Epstein-Schwarz surfaces}\Label{sEpsteinSurfaces}
\subsection{Epstein surfaces} (See  Epstein \cite{Epstein84}, and also Dumas \cite{Dumas18HolonomyLimit}.)
Let $C$ be a $\CP^1$-structure on $S$.  
Fix  a developing pair $(f, \rho)$ of $C$, where $f \col \ti{C} \to \CP^1$ is the developing map  and $\rho\col \pi_1(S) \to \PSL(2,\C)$ is the holonomy representation, which is unique up to $\PSL(2,\C)$.
For $z \in \H^3$, by normalizing the ball model of $\H^3$ so that $z$ is the center, we obtain a spherical metric $\nu_{\s^2}(z)$ on $\bdr_\infi \H^3 = \CP^1$.

 Given a conformal metric $\mu$ on $C$, there is a unique map $\Ep\col\ti{C} \to \H^3$ such that, for each $x \in \ti{C}$, the pull back of $\nu_{\s^2} \Ep(z)$ coincides with $\ti\mu$ at $z$.
This map is $\rho$-equivariant, and called the {\sf Epstein surface}.  
\subsection{Approximation}

Let $C = (X, q)$ be a $\CP^1$-structure on $S$ expressed in Schwarzian coordinates, where $q$ is a holomorphic quadratic differential on a Riemann surface $X$. 
Then $q$ yields a flat surface structure $E$ on $S$. 
Moreover $q$  gives a vertical measured foliation $V$ and a horizontal measured foliation $H$ on $E$. 

Let $\Ep\col \ti{S} \to \H^3$ be the Epstein surface of $C$ with the conformal metric given by $E$.  
Then, let $\Ep^\ast\col T\ti{S} \to T \H^3$ be the derivative of $\Ep$, where  $T\ti{S}$ and $T \H^3$ denote the tangent bundles.
Let $d\col \ti{E} \to \R_{\geq 0}$ be the distance function from the singular set $\ti{Z}_q$ with respect to the singular Euclidean metric of $\ti{E}$.

Let $v'(z)$ be the vertical unit tangent vector of $\ti{E}$ at a smooth point $z$.
Similarly, let $h'(z)$ be the horizontal unit tangent vector at a smooth point $z$ of $\ti{E}$. 

\begin{lemma}[\cite{Epstein84},  Lemma 2.6 and Lemma 3.4 in \cite{Dumas18HolonomyLimit}]\Label{Dumas}
~
\begin{enumerate}
\item $\| \Ep^\ast h'(z)\|  <  \frac{6}{d(z)^2}$; \Label{iHoriozontalDerivative}
\item $\sqrt{2} < \| \Ep^\ast v'(z)\| <  \sqrt{2} + \frac{6}{d(z)^2}$; \Label{iVerticalDerivative}
\item $h'(z), v'(z)$ are principal directions of $\Ep$ at $z$;
\item $k_v < \frac{6}{d(z)^2}$, where $k_v$ is the principal curvature of $\Ep$ in the vertical direction. \Label{iVerticalCurvature}
\end{enumerate}
\end{lemma}

Consider the Euclidean metric on  $\C \cong \mathbb{E}^2$.  
By the exponential map  $\exp\col \C \to \C^\ast$, we push forward a complete Euclidean metric to $\C^\ast$, which is invariant under the action of $\C^\ast$. 
If a simply connected region $Q$ in the flat surface $E$ contains no singular points, then $Q$ is immersed into $\C$ locally isometrically preserving horizontal and vertical directions. 
Using Lemma \ref{Dumas}  and the definition of Epstein surfaces, one obtains the following.
\begin{lemma}(\cite[Lemma 12.15]{Baba-15}.)\Label{ApproximationByExponentialMap}
For every $\ep > 0$, there is $r > 0$, such that if $Q$ is a region in $E$ satisfying
\begin{itemize}
\item  $Q$ has $E$-diameter less than $r$, and
\item the distance from the singular set of $E$ is more than $r$. 
\end{itemize}
then $\exp\col \C \to \C^\ast$ and the developing map are $\ep$-close pointwise with respect to the complete Euclidean metrics. 
\end{lemma}

We shall further analyze vertical curves on Epstein surfaces.
Let $v\col [0,\ell] \to \ti{E}$ be a path in a vertical leaf, such that $v$ contains no singular point and has a constant speed $\frac{1}{\sqrt{2}}$ in the Euclidean metric.  
Let $\Ep^\perp(z)$ be the unit normal vector of the Epstein surface $\Ep$ at each smooth point $z \in \ti{E}$. 
Let $s_t$ be the geodesic segment in $\H^3$ connecting $\Ep v(0)$ to $\Ep v(t)$\,; see  \Cref{fAlmostGeodesic}.
\begin{figure}
\begin{overpic}[scale=.15,
] {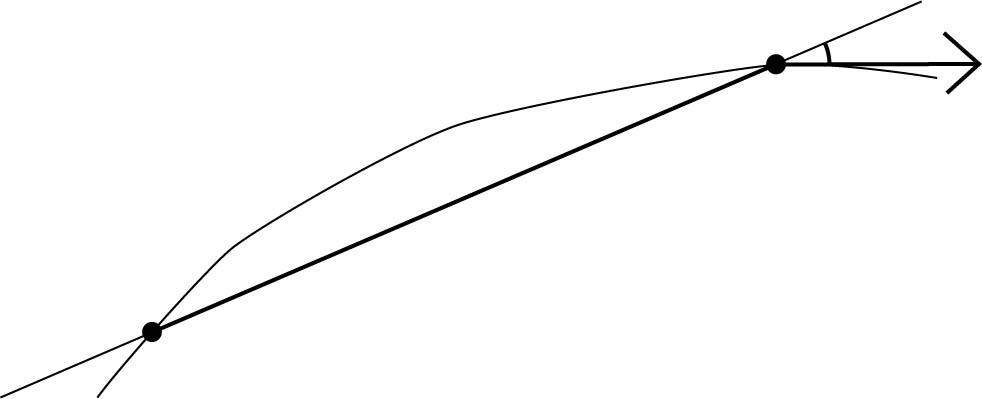} 
  \put(80 ,24 ){\textcolor{black}{\small $\Ep^\ast v'(t)$}}  
  \put(42 ,13 ){\textcolor{black}{$s_t$}}  
  \put(17 , 1 ){\textcolor{black}{\small $\Ep v(0)$}}  
      \end{overpic}
\caption{}\label{fAlmostGeodesic}then the total curvature along 
\end{figure}

The following lemma is an analogue of \Cref{ABitBentGeodesic} regarding piecewise geodesic curves for smooth curves. 
\begin{lemma}\Label{AlmostGeodesic}
For every $\ep > 0$, there is (large) $\omega > 0$ only depending on $\ep$, such that, w.r.t. the $E$-metric, if the distance of the vertical segment $v$ from the zeros $Z_q$ of $q$ is more than $\omega$,
then the angle between $\Ep^\ast v'(t)$ and the geodesic containing $s_t$ is less than $\ep$ for all $t$. (Figure \ref{fAlmostGeodesic}.)
\end{lemma}
\begin{proof}
   In fact, the proof of this lemma is essentially reduced to the analogous lemma (\Cref{ABitBentGeodesic}) for piecewise geodesic curves as follows.
   
Fix a Riemannian metric on the tangent bundle of $\H^3$ which is invariant under the isometries of $\H^3$. 
Then,  by  Lemma \ref{Dumas} (\ref{iVerticalDerivative}) and (\ref{iVerticalCurvature}), for every $\ep_1 > 0$, there is sufficiently large $\omega > 0$ such that, if a vertical segment $v\col [0, \ell] \to  \ti{E}$ of unit speed has length less than $\frac{1}{\ep}$ and distance from $Z_p$ at least $\omega$, then the smooth curve $\Ep \circ \,v$ is $\ep_1$-close to the geodesic segment connecting the  endpoints of $\Ep \circ \,v$ in the $C^1$-topology with respect to the invariant metric. 
Therefore,  the lemma holds true under the additional assumption that the length of $v$ is uniformly bounded from above. 

Now, without any upper bound on the length, let $v\col [0, \ell] \to  \ti{E}$ be a vertical segment of unit speed which has distance at least $\omega$ from $Z_p$.  
Let $\ep_1 > 0$ be a constant. 
Then we decompose $v$ into $n$ segments $v_1, v_2, \dots, v_n$ so that the first $n -1$ segments $v_1, v_2, \dots, v_{n-1}$ have length exactly $\frac{1}{\ep_1}$ and the last segment $v_n$ has length at most $\frac{1}{\ep_1}$ . 
For all $i = 1, 2, \dots, n$, let $u_i$ be the geodesic segment  connecting the endpoints of $\Ep \circ v_i$. 
Then, by the argument above, for every $\ep_2 > 0$, if $\ep_1 > 0$ is sufficiently small, then the piecewise geodesic curve $\cup_{i = 1}^n u_i$ is $\ep_2$-close to $\Ep \circ v$ in $C^1$-topology. 
We can, in addition,  assume that the exterior angle at the common endpoint of $u_i$ and $u_{i + 1}$ is less than $\ep_2$ for all $i = 1, 2, \dots, n_1$. 
Therefore, by \Cref{ABitBentGeodesic}, for every $\ep_2 > 0$, if $\ep_1 > 0$ is sufficiently small, then the piecewise geodesic curve $\cup_{i = 1}^n u_i$  is $\ep_2$-close to the geodesic segment connecting the endpoints of $\Ep \circ v$ in $C^1$-topology. 
(See \Cref{fAlmostGeodesicVerticalSegment}.)

Therefore, for every $\ep >  0$,  if $\ep_1 > 0$ is sufficiently small, then $\Ep \circ v$ is $\ep$-close to the geodesic segment connecting its endpoints.
 Then the lemma immediately follows. 
\begin{figure}[H]
\begin{overpic}[scale=.06
] {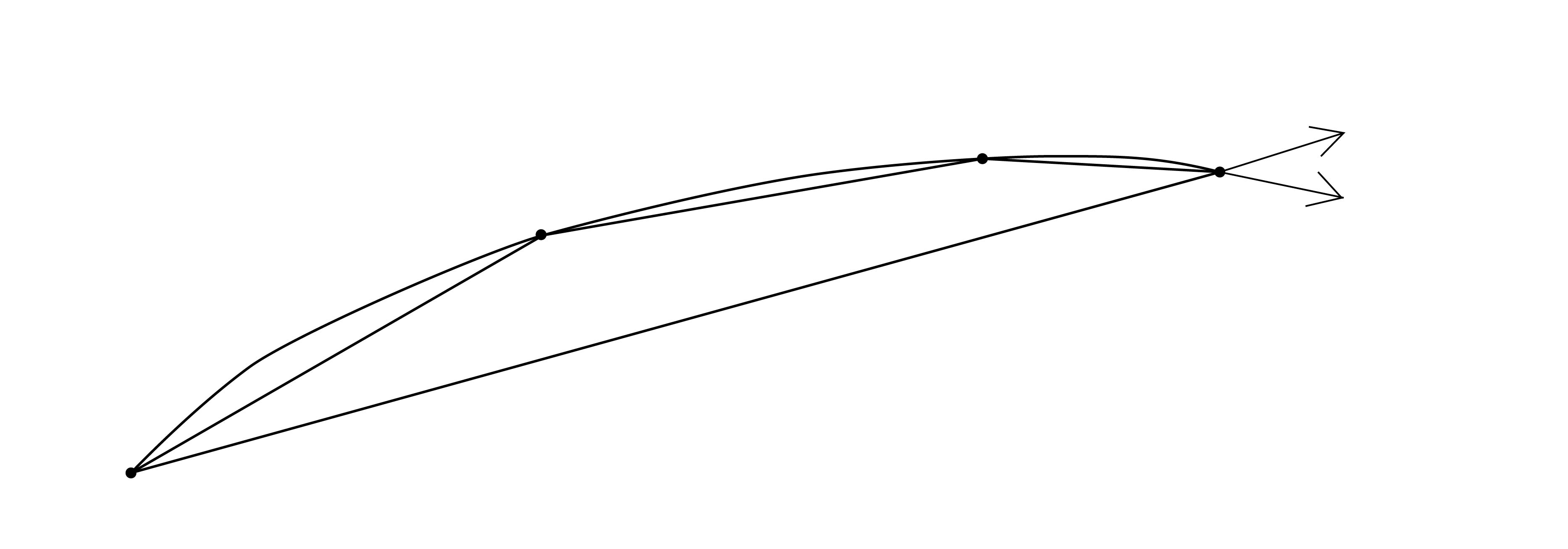} 
 \put(25 , 12 ){\small $u_1$}  
 \put(45 , 19 ){\small $u_2$}  
  \put(68 , 21.5 ){\contour{white}{\small $u_3$}}  
      \end{overpic}
\caption{A piecewise geodesic curve $u_1 \cup u_2 \cup u_3$ which is  $C^1$-close to both the smooth curve $\Ep \circ v$ and the geodesic segment connecting the endpoints of $\Ep \circ v$  (in the case of $n = 3$).}\label{fAlmostGeodesicVerticalSegment}
\end{figure}
\end{proof}

Define $\theta \col [0,\ell] \to T_{\Ep v(0)}$ by the parallel transport of $\Ep^\perp (t)$ along $s_t$ to the starting point $\Ep(v(0))$.
Let $H$ be the (totally geodesic) hyperbolic plane in $\H^3$ orthogonal to the tangent vector $\Ep^\ast v'(0)$, so that $H$ contains $\Ep^\perp v(0)$.   
Then, Lemma \ref{AlmostGeodesic}, implies
\begin{corollary}
For every $\ep > 0$, there is (large) $\omega > 0$ only depending on $\ep$ such that, if 
  the Hausdorff distance between $v$ and the zeros $Z_q$ of $q$ is more than $\omega$  w.r.t. the $E$-metric, 
then
$\angle_{v(0)}(\theta(t), H) < \ep$ for all $t \in [0,\ell]$. 
\end{corollary}

Recall that the $\PSL_2\C$-character variety $\chi$ of the surface $S$ is an affine algebraic variety. 
Then we say a {\sf compact subset} $K$ in the character variety $\chi$ or the holonomy variety $\chi_X$ for $X \in \ol\TT$ is {\sf sufficiently large},  if $K$ contains a sufficiently large ball in the ambient affine space centered at the origin.

\begin{proposition}[Total curvature bound in the vertical direction]\Label{TotalCurvatureBound}
For all $X \in \TT\cup \TT^\ast$ and all $\ep > 0$, there is a bounded subset $K = K(X, \ep)$ in $\rchi_X$, such that, for $\rho \in \rchi_X \minus K$, 
if a vertical segment $v$ has normalized length less than $\frac{1}{\ep}$ and has normalized Euclidean distance from the zeros of $q_{X, \rho}$ at least $\ep$, then 
the total curvature along $v$ is less than $\ep$. 
\end{proposition}
\begin{proof}
For every $r > 0$, if $K$ is sufficiently large, then, if a $\CP^1$-structure $C = (X, q)$ on $X$ has holonomy outside $K$, then the distance from $Z_q$ to $v$ is at least $r$. 
Then the proposition immediately follows from Dumas' estimate in Lemma \ref{Dumas} (\ref{iVerticalCurvature}).
\end{proof}

Consider the projection $\hat\theta(t)$ of $\theta(t) \in T^1_{v(0)} \H^3$ to the unit tangent vector in
 $H$ at $v_0$. 
Let $\eta\col [0, \ell] \to \R$ be the continuous function of the total increase of $\hat\theta(t) \col [0, \ell] \to \R$, so that
 $\eta(0) = 0$ and  $\eta'(t) = \hat\theta'(t)$.

\begin{proposition}\Label{TotalTorsionBound}
Let $X \in \TT \sqcup \TT^\ast$.
For every $\ep > 0$, there is a bounded subset $K = K(X, \ep) > 0$ in $\rchi_X$, such that, if  
\begin{itemize}
\item  $C \in \PP_X$ has holonomy in $\rchi_X \minus K$;
\item a vertical segment $v$ of the normalized flat surface $E_C^1$ has the length less than $\frac{1}{\ep}$;
\item the normalized distance of  $v$ from the singular set $Z_C$ of $E_C^1$ is more than $\ep$, 
\end{itemize}
then, $|\eta'(t)| < \ep$ for $t \in [0,\ell]$ and  $\int_0^\ell |\eta'(t) | < \ep$.
In particular, $|\eta(t)| < \ep$ for all $t \in [0,\ell]$.
\end{proposition}
\begin{proof}
   The absolute value of $\theta'(t)$ is bounded from above by the curvature of $\Ep \circ v\col [0, \ell] \to \H^3$ at $t$. 
Therefore  $|\eta'(t)|$ is bounded from above the curvature.  
 Thus, for every $\ep > 0$, if $K$ is sufficiently large, then by \Cref{Dumas} (\ref{iVerticalCurvature}), then $|\eta'(t)| < \ep$ for all $t \in [0, \ell]$, regardless of the choice of  the vertical segment $v$.
  Therefore, by \Cref{TotalCurvatureBound}, if $K$ is sufficiently large, $\int_0^\ell |\eta'(t) | < \ep$ holds. 
\begin{figure}
\begin{overpic}[scale=.05, 
] {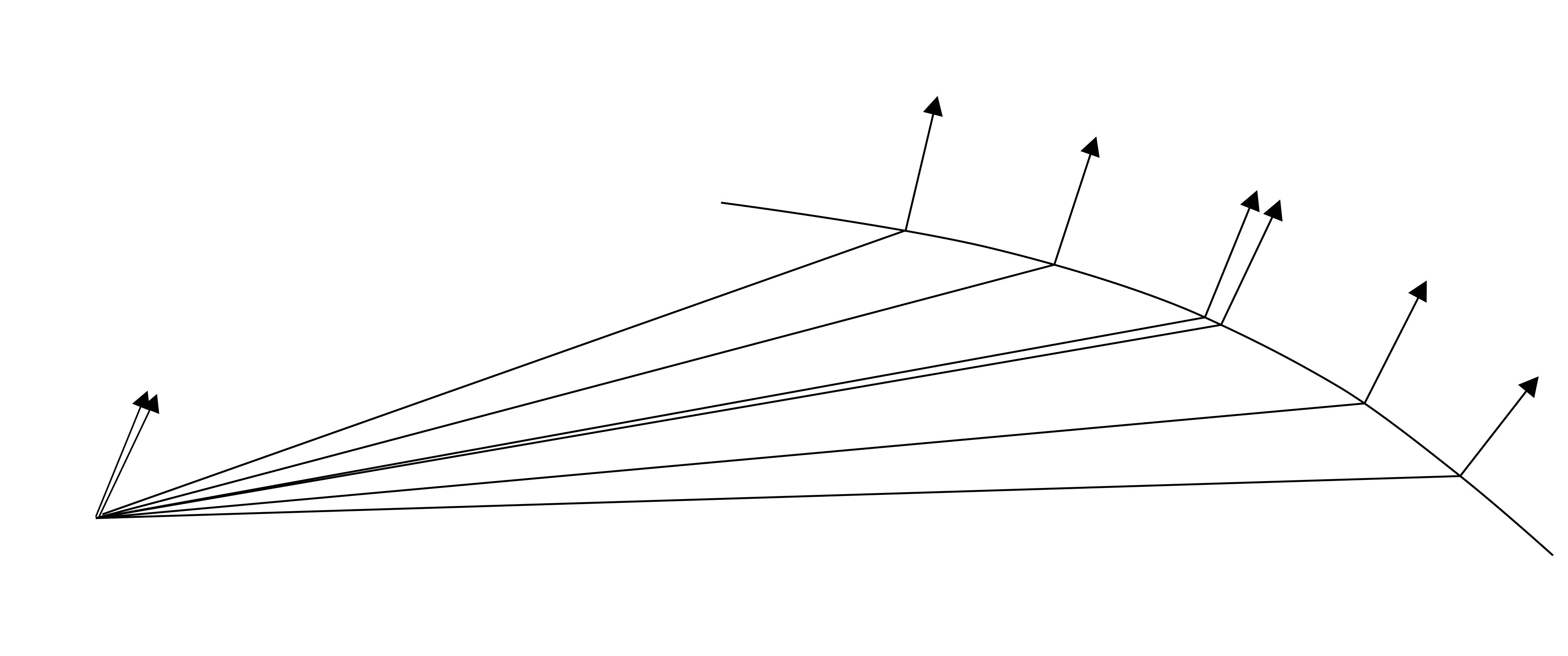} 
 \put(83 , 25){$\Ep^\perp (t)$}  
      \put(50 ,14 ){$s_t$}  
          \put(1 , 19 ){$\theta(t)$}  
      \end{overpic}
\caption{}\label{fTheta}
\end{figure}
\end{proof}

Let $\alpha$ be the bi-infinite geodesic in $\H^3$ through $\Ep(v(0))$ and $\Ep(v(\ell))$. 
Let $p_1, p_2$ denote the endpoints of $\alpha$ in $\CP^1$. 
If a hyperbolic plane in $\H^3$  is orthogonal to $\alpha$, then its ideal boundary is a round circle in $\CP^1 \minus \{p_1, p_2\}$.
Moreover $\CP^1 \minus \{p_1, p_2\}$ is foliated by round circles which bound hyperbolic planes orthogonal to $\alpha$. 

If a hyperbolic plane in $\H^3$ contains the geodesic $\alpha$, then its ideal boundary is a round circle containing $p_1$ and $p_2$.
Then, by considering all such hyperbolic planes,  we obtain another foliation $\mathcal{V}$ of $\CP^1 \minus \{p_1, p_2\}$ by circular arcs connecting $p_1$ and $p_2$.
Then $\mathcal{V}$ is orthogonal to the foliation by round circles. 
Note that $\mathcal{V}$ has a natural transversal measure given by the angles between the circular arcs at $p_1$ (and $p_2$). 
Then the transversal measure is invariant under the rotations of $\H^3$ about $\alpha$, and its total measure is $2\pi$.
Given a smooth curve $c$ on $\CP^1 \minus \{p_1, p_2\}$ such that $c$ decomposes into finitely many segments $c_1, c_2, \dots c_n$ which are transversal to $\mathcal{V}$, possibly, except at their endpoints. 
Let $\mathcal{V} (c)$ denote the ``total'' transversal measure of $c$ given by $\mathcal{V}$, the sum of the transversal measures of $c_1, c_2, \dots, c_n$.
Then, \Cref{TotalTorsionBound} implies the following. 
\begin{corollary}\Label{VerticalSegmentAlmostZeroMeasure}
For every $\ep > 0$, there is a bounded subset $K \sub \rchi_X$, such that, if 
\begin{itemize}
\item  $C \in \PP_X$ has holonomy in $\rchi_X \minus K$,  
\item a vertical segment $v$ of $E_C$ has the normalized length less than $\frac{1}{\ep}$, and
\item the normalized distance of  $v$ from the zeros $Z_C$ is more than $\ep$, 
\end{itemize}
then, 
the curve $f | v\col [0, \ell] \to \CP^1$ intersects $\mathcal{V}$ at angles less than $\ep$, and the total $\mathcal{V}$-transversal measure of the curve is less than $\ep$. 
\end{corollary}

\begin{definition}
Let $v$ be a unit tangent vector of $\H^3$ at $p \in \H^3$. 
Let $H$ be a totally geodesic hyperbolic plane in $\H^3$. 
For $\ep > 0$,  $v$ is $\ep$-almost orthogonal to $H$ if $dist_{\H^3}(H, p) < \ep$ and the angle between the geodesic $g$ tangent to $v$ at $p$ and $H$ is $\ep$-close to $\pi/2$. 
\end{definition}

Fix a metric on the unit tangent space $T^1 \H^3$ invariant under $\PSL(2, \C)$. 
For $\ep > 0$, let $N_\ep Z_{X, \rho}^1$ denote the $\ep$-neighborhood of the singular set $Z_{X, \rho}^1$ of the normalized flat surface $E^1_{X, \rho}$.
\begin{theorem}\Label{HorizontalAndVerticalEstimate} 
Fix arbitrary $X \in \TT \sqcup \TT^\ast$. 
For every $\ep > 0$, if a bounded subset $K_\ep \sub \rchi_X$ is sufficiently large, then, for all $\rho \in \rchi_X \minus K_\ep$, 

\begin{enumerate}
\item  if a vertical segment $v$ in $E_{X, \rho}^1 \minus N_\ep Z_{X, \rho}^1$ has length less than $\frac{1}{\ep}$, then  the total curvature of $\Ep_{X, \rho} | v$ is less than $\ep$, \Label{iTotallCurvature} and
\item  if a horizontal segment $h$ in $E_{X, \rho}^1 \minus N_\ep Z_{X, \rho}^1$ has length less than $\frac{1}{\ep}$, then
 for the vertical tangent vectors $w$ along the horizontal segment $h$, their images $\Ep_{X, \rho}^\ast(w)$ are $\ep$-close in the unit tangent bundle of $\H^3$. \Label{iCloseVerticalTangentVectors}
\end{enumerate}
\end{theorem}
\proof
(\ref{iTotallCurvature}) is already by
Proposition \ref{TotalCurvatureBound}.
 By \cite[Proposition 4.7]{Baba-15}, we have  (\ref{iCloseVerticalTangentVectors}). 
 \Qed{HorizontalAndVerticalEstimate}

\section{Comparing measured foliations}\Label{sComoparingMeasuredFoliations}
\subsection{Thurston laminations and vertical foliations}
Let $L_1, L_2$ be measured laminations or foliations on a surface $F$. 
Then $L_1$ and $L_2$ each define a pseudo-metric almost everywhere on $F$: for all $x, y \in F$  not contained in a leaf of $L_i$ with atomic measure,  consider the minimal transversal measure of all arcs connecting $x$ to $y$.
We say that,  for $\ep > 0$, $L_1$ is  $(1 + \ep, \ep)$-{\sf quasi-isometric} to $L_2$, if for almost all  $x, y \in F$, 
 $$(1 + \ep)^{-1} d_{L_1}(x, y)  - \ep< d_{L_2}(x, y) < (1 + \ep) d_{L_1}(x, y) + \ep.$$

We shall compare a measured lamination of the Thurston parametrization and a measured foliation from the Schwarzian parametrization of a $\CP^1$-surface. 
\begin{theorem}\Label{ThurstonLaminationAndVerticalFoliationOnDisks}
For every $\ep > 0$, there is $r > 0$ with the following property: 
 For every $C \in \PP \sqcup \PP^\ast$, then, letting $(E, V)$ be its associated flat surface, if 
  disk $D$ in $E$ has radius less than  $\frac{1}{\ep}$  and the distance between $D$ and the singular set $Z$ of $E$ is more than $r$, then the vertical foliation $V$ of $C$  is $(1 + \ep, \ep)$-quasi isometric to $\sqrt{2}$ times the Thurston lamination $L$ of $C$ on $D$.
\end{theorem}
\proof[Proof of Theorem \ref{ThurstonLaminationAndVerticalFoliationOnDisks}] 
It suffices to show the assertion when $D$ is a unit disk. 
Since $D$ contains no singular point,  we can regard $D$ as a disk in $\C$ by the natural coordinates given by the quadratic differential. 
The scaled exponential map 
$$ \exp  (\sqrt{2}\, \ast ) \col \C \to \C^\ast\col z \mapsto \exp(\sqrt{2} z).$$ 
 is a good approximation of the developing map sufficiently away from zero (\Cref{ApproximationByExponentialMap}), which was proved using Dumas' work \cite{Dumas18HolonomyLimit}).
Let $C_0$ be the $\CP^1$-structure on $\C$ whose developing map is $\exp  (\sqrt{2}\, \ast)$. 
The next lemma immediately follows from the construction of Thurston coordinates. 
\begin{lemma}
The Thurston lamination on $C_0$ is the vertical foliation of $\C$ with a transversal measure given by the horizontal Euclidean distance. 
\end{lemma}

Let $D_x$ be the maximal disk in $\ti{C}$ centered at $x$.
Let $D_{0,x}$ be the maximal disk in $C_0$ centered at $x$ by the inclusion $D \sub \C$. 
When $\CP^1$ is identified with $\s^2$ so that the center $O$ of the disk $D$ map to the north pole and the maximal disk in $\ti{C}$ centered at $O$ maps to the upper hemisphere. 
If $r > 0$ is sufficiently large, then the $\dev | D$ is close to  $\exp  (\sqrt{2}\, \ast)$. 
Then, for every $x \in D$, its maximal disk $D_x$ in $\ti{C}$ is $\ep$-close to the maximal disk $D_{0,x}$ in $C_0$,  and the ideal point $\bdr_\infi D_x$ is $\ep$-Hausdorff close to the idea boundary $\bdr_\infi D_{0,x}$ on $\s^2$. 

Therefore,  by \cite[Theorem 11.1, Proposition 3.6]{Baba-17}, the convergence of canonical neighborhoods implies the assertion. 
\Qed{ThurstonLaminationAndVerticalFoliationOnDisks}

A staircase polygon is a polygon in a flat surface whose edges are horizontal or vertical (see Definition \ref{dStaircase}). 
\begin{theorem}\Label{ThurstonLaminationAndVerticalFoliationOnPolygons}
For every $X \in \TT \sqcup \TT^\ast$ and every $\ep > 0$,  there is a constant $r > 0$ with the following property: 
Suppose that $C$ is a $\CP^1$-structure on $X$ and $C$ contains a staircase polygon $P$ w.r.t. its flat surface structure $(E, V)$, such that   the ($E$-)distance from $\bdr P$ to the singular set $Z$ of $E$ is more than $r$. 
Then, letting $\LLL$ denote the Thurston lamination  of $C$,
the restriction of  $\LLL$ of $C$  to  $P$ with its transversal measure scaled by $\sqrt{2}$ is
  $(1 + \ep, \ep)$-quasi-isometric to the vertical foliation $V$ on $P$ up to a diffeomorphism supported on the $r/2$-neighborhood of the singular set in $P$. \end{theorem}

\proof
Let $N_{r/2} Z$ denote the $r/2$-neighborhood of $Z$. 
If $r > 0$ is sufficiently large, then, 
for each disk $D$ of radius  $\frac{r}{4}$ centered at a point on $E  \minus N_{r/2}(Z)$, the assertion holds by Theorem \ref{ThurstonLaminationAndVerticalFoliationOnDisks}.

Since $\bdr P \cap N_{r/2} Z =\emptyset$,
there is an upper bound for lengths of edges of such staircase polygons $P$  with respect to the normalized Euclidean metric $E^1$.  \begin{lemma}\Label{VerticalLeafInFoliationAndLamination}
For every $\ep > 0$, if $r > 0$ is sufficiently large, then 
for every vertical segment $v$ of $V | P$ whose distance from the singular set $Z$ is more than $r/2$,  we have $\LLL(v) < \ep$. 
\end{lemma}
\begin{proof}
This follows from Corollary \ref{VerticalSegmentAlmostZeroMeasure}.
\end{proof}
By Theorem \ref{ThurstonLaminationAndVerticalFoliationOnDisks} and Lemma \ref{VerticalLeafInFoliationAndLamination}, $V$ and $\LLL$ are $(1 + \ep, \ep)$-quasi-isometric on $P$ minus $N_{r/2} Z$. 
Note that $V$ and $\LLL$ in $P \cap N_{r/2}$ are determined by $V$ and $\LLL$ in $P \minus N_{r/2}$ up to an isotopy, respectively. 
Therefore, as desired,  $V$ and $\LLL$ are $(1 + \ep, \ep)$-quasi-isometric on $P$, up to a diffeomorphism supported on $N_{r/2} Z$. 
\Qed{ThurstonLaminationAndVerticalFoliationOnPolygons}

\subsection{Horizontal foliations asymptotically coincide}
Let $X, Y \in \TT \sqcup \TT^\ast$ with $X \neq Y$.
Let $\rchi_X = \Hol \PP_X$ and let $\rchi_Y = \Hol \PP_Y$, the holonomy varieties of $X$ and $Y$, respectively.
Suppose that $\rho_i$ is a sequence in $\rchi_X \cap \rchi_Y$ which leaves every compact set in $\rchi$.
Then,  let  $C_{X, i}$ and $C_{Y, i}$ be the $\CP^1$-structures on $X$ and $Y$, respectively, with holonomy $\rho_i$. 
Similarly, let $H_{X, i}$ and $H_{Y, i}$ denote the horizontal measured foliations of $C_{X, i}$ and $C_{Y, i}$.
Then, up to a subsequence, we may assume that 
 $\rho_i$ converges to a $\pi_1(S)$-tree $T$ in the Morgan-Shalen boundary of $\rchi$, and that the projective horizontal foliations $[H_{X, i}]$ and $[H_{Y, i}]$ converge to $[H_X]$ and $[H_Y] \in \PML(S)$, respectively, as $i \to \infi$.
Let $\zeta\col \pi_1(S) \to \Isom\, T$ denote the representation given by the isometric action in the limit, where $\Isom \,T$ is the group of isometries of $T$.

Let $\ti{H}_X$ be the total lift of the horizontal foliation $H_X$ to the universal cover of $X$, which is a $\pi_1(S)$-invariant measured foliation. 
Then, collapsing each leaf of $\ti{H}_X$ to a point, we obtain a $\R$-tree $T_X$, where the metric is induced by the transversal measure (dual tree of $\ti{H}$).  
Let $\phi_X\col \ti{S} \to T_X$ be the quotient collapsing map, which commutes with the $\pi_1(S)$-action. 
By Dumas  (\cite[Theorem A, \S 6]{Dumas18HolonomyLimit}), there is  a unique {\sf straight map}   $\psi_X\col T_X \to T$  such that 
 $\psi_X$ is also $\pi_1(S)$-equivariant, and that every non-singular vertical leaf of $\ti{V}|_X$ maps to a geodesic in $T$.
 
Similarly,
let $\phi_Y\col \ti{S} \to T_Y$ be the map which collapses each leaf of $\ti{H}_Y$ to a point.
Let $\psi_Y\col T_Y \to T$ be the straight map. 
\begin{figure}
\begin{tikzcd}
 & \ti{S} \arrow{dr}{\phi_Y} \arrow{dl}{\phi_X}\\
T_X \arrow{dr}{\psi_X} && T_Y \arrow{dl}{\psi_Y}\\
 & T \end{tikzcd}
\caption{}\Label{fCollapsingAndHoldingDiagram}
\end{figure}
Let $d_T$ be the induced metric on $T$. 

Next we show the horizontal foliations coincide in the limit as projective laminations. 
\begin{theorem}\Label{HorizontalFoliationsCoinside}
$[H_X] = [H_Y]$ in $\PML$. 
\end{theorem}
\proof
Pick a diffeomorphism $X \to Y$ preserving the marking. 
Let  $\xi \col X \to Y$ be a piecewise linear homeomorphism which is a good approximation of $\xi$ with respect to the limit singular Euclidean structures $E_X, E_Y$ on $X$ and $Y$; let $E_X = \cup_{j = 1}^p \sigma_j$ be the piecewise linear decomposition of $E_X$ for $\xi$, where $\sigma_1, \dots, \sigma_p$ are convex polygons in $E_X$ with disjoint interiors.
We diffeomorphically identify $X, Y$ with the base surface $S$, so that the identifications induce $\xi$.
Let $\ti\xi \col \ti{X} \to \ti{Y}$ be the lift of $\xi\col X \to Y$ to a $\pi_1(S)$-equivariant map between the universal covers $\ti{X}, \ti{Y}$. 

Recall that the dual tree $T$ is a geodesic metric space. 
Therefore, the $\zeta$-equivariant maps $\psi_X \circ \phi_X\col \ti{X} \to T$ and $\psi_Y \circ \phi_X\col \ti{Y} \to T$ are  $\zeta$-equivariantly homotopic when identifying their domains by $\ti\xi$.
Namely, for each $x \in \ti{S}$, for $t \in [0,1]$, let $\eta_t(p)$ be the point dividing the geodesic  segment from $\psi_X \circ \phi_X (p)$ to $\psi_Y \circ \phi_Y (p) $ in the ratio $t \colon 1-t$. 
By subdividing the piecewise linear decomposition  $E_X = \cup_{j = 1}^p \sigma_j$ if necessary, we may assume that for each $j = 1, \dots, p$, $\psi_X\circ \phi_X (\ti\sigma_j)$  and $\psi_Y\circ \phi_Y (\ti\sigma_j)$  are the geodesic segments in $T$  contained in a common geodesic in $T$ for all lifts $\ti\sigma_j$ of linear pieces $\sigma_j\, (j =1, \dots, p)$, where $\ti\sigma_j$ is a lift of $\sigma_j$ to the universal cover $\ti{E}_X$. 
Note that $\eta_t(\ti\sigma_j)$ may be a single point in $T$ for $t \in (0,1)$\,;  however this degeneration may happen only at most a single time point $t  \in [0,1]$ for each $j$.
Let $0 < t_1 < t_2 < \dots t_m < 0$ be the time points such that $\eta_{t_i}$ takes some piece $\ti\sigma_j$ to a single point in $T$. 

Suppose $\eta_t (\ti\sigma_j)$ is a segment in $T$ for $t \in [0,1]$.
Then the fibers of $\eta_t$ yield a foliation on $\sigma_j$. 
Moreover the pullback of the distance in $T$ gives the transversal measure on the foliation. 
That is,  if an arc in $\sigma$ is transversal to the foliation, its transversal measure is the distance in $T$ between the images of the endpoints of the arc.
Therefore, if $t \neq t_1, t_2, \dots, t_m$, $\eta_t$ gives a singular measured foliation $H_t$ on $S$, where singular points are contained in the boundary of the linear pieces. 
Then, recalling that we have fixed a metric on $T$ in its projective class,  we have $H_0 = H_X$  and $H_1 = H_Y$ as $\psi_X$ and $\psi_Y$ are straight maps, up to scaling of $H_X$ and $H_Y$. 

At time $t_i$, the  $\eta_{t_i}$-image of $\ti\sigma_j$ is a single point in $T$ for some $j$.
Then, since all points on $\ti\sigma_j$ map to the same point on $T$, the pull-back of the distance on $T$ by $\eta_{t_i}$ can be regarded as the empty lamination on $\sigma_j$. 
Thus,  we obtain a measured lamination $H_{t_i}$ on $S$, pulling back the distance by $\eta_{t_i}$.
Therefore, we obtain a measured lamination $H_t$ on $S$ for all $t \in [0,1]$. 
Moreover,  as  the $\zeta$-equivariant homotopy $\eta_t\col \ti{S} \to T$ changes continuously in $t$, $H_t$ changes continuously on $t \in (0,1)$. 

For each $j = 1, \dots, p$, let $U_j$ be a small piecewise linear neighborhood of $\sigma_j$ homeomorphic to a disk in $E_X$. 
Then, for every $\ep > 0$, we can approximate the homotopy $\eta_t$ ($0 \leq t \leq 1$) by $\xi_t$ such that
\begin{itemize}
\item $\eta_0 = \xi_0$ and $\eta_1 = \xi_1$;
\item $\eta_t$ is piecewise linear;
\item $\eta_t$ is $\ep$-close to $\xi_t$ in  $C^0$-topology; 
\item  there is a sequence $0 = u_0 < u_1 <  u_2 < \dots < u_m =1$, such that, for each $i = 0,1, \dots, m-1$, the homotopy $\xi_t$ is supported on the neighborhood $U_j$ of some $\sigma_j$ for $u_i \leq t \leq u_{i + 1}$.
\end{itemize}

 For each $t \in [0,1]$,   similarly $\xi_t$ induces a measured lamination $W_t$ on $S$ so that, in each linear piece, the fibers of $\xi_t$ yield strata of the lamination and the distance $T$ the transversal measure. 
Then, when $\ep > 0$ is small,  $W_t$ is a good approximation of $W_t$. 
By the continuity of $\xi$, the measured lamination $W_t$ changes continuously in $t \in [0,1]$. 

We shall modify the measured lamination $W_t$ by certain homotopy, removing the ``loose part'' of $W_t$ in order to make $\eta_t$ ``tight". 
By tightening, with respect to the pull-back of the metric of $T$,  the minimal measure of the homotopy class of every closed curve does not increase. 
Thus this tightening operation removes an obviously unnecessary part of the pull-back measure.  
See \Cref{fTightening} for some examples.
\begin{figure}
\begin{overpic}[scale=.1
] {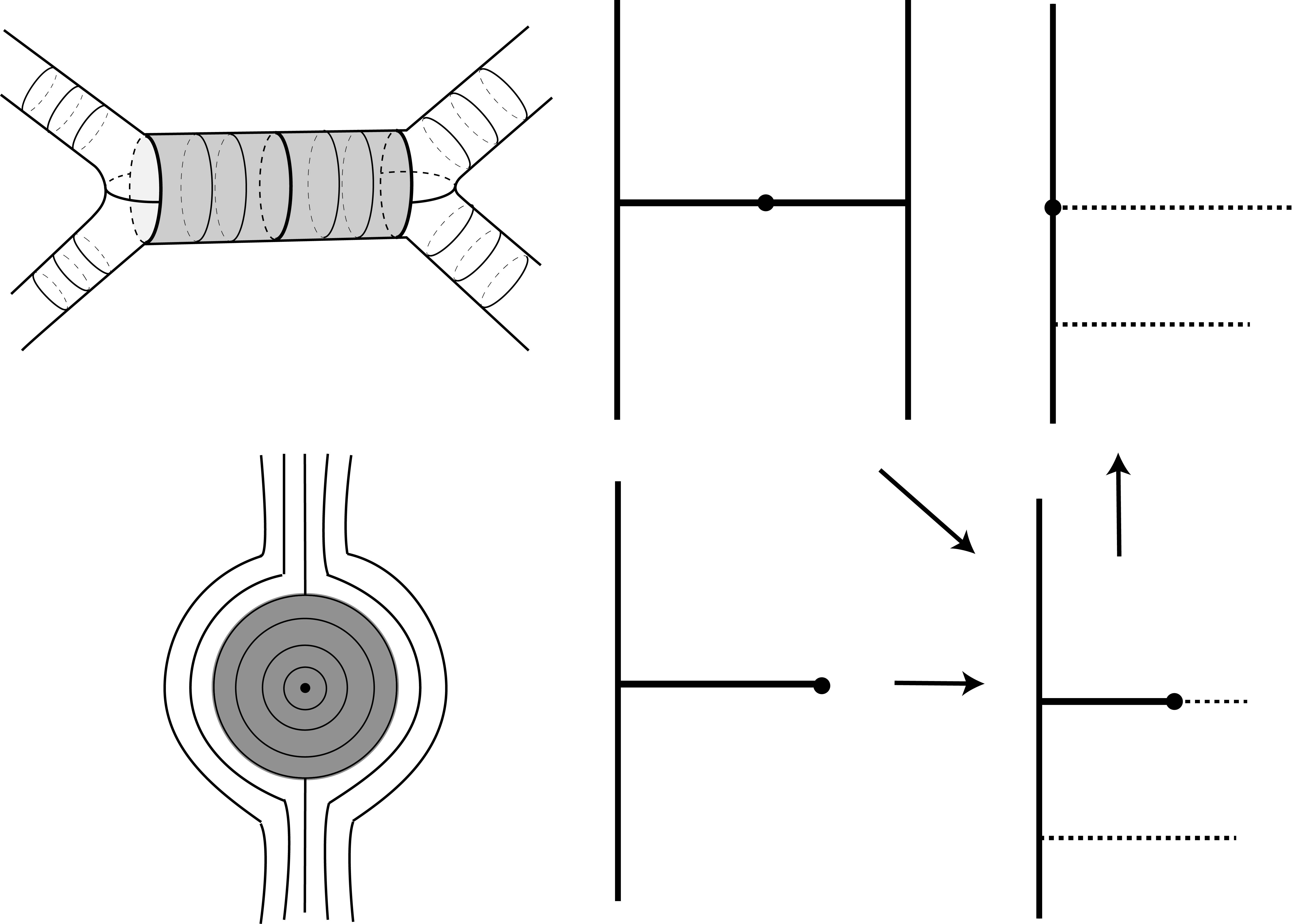} 
  \put(10 , 40){$W_t$ on $S$}  
  \put(60 , 50){$T_t$}  
  \put(55 , 21){$T_t$}  
  \put(90 , 30){tighten}  
  \put(90 , 49){$T$}  
  \put(90 , 10){$T$}  
  \put(82 , 20){$\Im \psi_t$}  
    \put(70 , 22){$\psi_t$}  
   \put(73 , 34){$\psi_t$}  
      \end{overpic}
\caption{Local pictures of basic examples of tightening. The segments of $T_t$ correspond to shaded regions on $S$, by $\psi_t$ map to an edge with a degree one end which collapses to a single point by tightening; thus those shaded regions are strata of $W_t'$, and thus $W_t'$ do not give any measure to arcs in the regions. The dotted lines in $T$ indicated $T \minus \Im \psi_t.$   }\label{fTightening}
\end{figure}

Let $\ti{W}_t$ be the $\pi_1(S)$-invariant measured lamination on $\ti{S}$ obtained by lifting $W_t$.
Let $T_t$ be the dual tree of $\ti{W}_t$. 
Then let $\phi_t \col \ti{S} \to T_t$ denote the collapsing map. 
Let $\psi_t\col T_t \to T$ denote the folding map so that $\eta_t = \psi_t\circ\phi_t$. 
Suppose that there is a bounded connected subtree $\gam$ of $T_t$ such that 
\begin{itemize}
\item $\gam$ is a closed subset of $T_t$; 
\item  The boundary of $\gam$ maps to a single point $z_\gam$ on $T$ by $\psi_t$, and the interior of $\gam$ maps into a single component of $z_\gam \minus z_\gam$;
\item  for every $\alpha \in \pi_1(S)$, ${\rm int}\, (\alpha \gam)$ is  disjoint from   ${\rm int} \,\gam$.  
 \end{itemize}
 We call such a subtree {\sf loose}. 
For a technical reason, we allow $\gam$ to be a single point on $T_t$. 
However, we will later identify a single point subtree of $T$ with the empty set when we consider deformations of such subtrees. 

For $t \in (0,1)$, fix a loose subtree $\gam$ of $T_t$.   
Then let $\psi_t' \col T_t \to T$ be the $\zeta$-equivariant continuous ``collapsing'' map, such that $\psi_t' (\gam)$ is the point $\psi_t(\bdr \gam)$, ~ $\psi_t' (\alpha \gam)$ is the point $\psi_t(\alpha \gam)$ for each $\gam \in \pi_1(S)$, and $\psi_t'(x) = \psi_t(x)$ for all $x \in T$ {\it not} contained in the union of $\pi_1(S)$-orbits of $\gam$.
Notice that $\psi_t(\gam)$ is a subtree of $T$, and $\psi_t(\bdr \gam)$ is an endpoint of the subtree.
Therefore, there is a $\zeta$-equivariant homotopy from $\psi_t$ to $\psi_t'$.
Thus we call $\psi_t'$ a {\sf tightening} of $\psi_t$ w.r.t. $\gam$.
Notice that $\phi_t^{-1}(\gam)$ is a closed simply connected region in $\ti{S}$ bounded by some strata of $\ti{W}_t$ which all map to the same point $z_\gam$ on $T$ by $\psi_t \circ \phi_t$.

More generally, suppose that there are finitely many loose subtrees $\gam_1, \gam_2, \dots, \gam_n$ of $T_t$, such that $\pi_1(S)$-orbits of their interiors ${\rm int} \gam_1, {\rm int} \gam_2, \dots, {\rm int} \gam_n$ are all disjoint. 
Then we can homotopy the holding map $\psi_t\col T_t \to T$, simultaneously tightening all loose subtrees $\gam_1, \gam_2, \dots, \gam_n$.

Pick a maximal collection of such loose subtrees  $\gam_1, \gam_2, \dots, \gam_n$ of $T_t$, so that we can {\it not} enlarge any of those loose subtrees or add another one. 
Then let $\psi'_t\col T_t \to T$ be the tightening of $\psi_t$ w.r.t. $\gam_1, \gam_2, \dots, \gam_n$ (maximal tightening).
  Let $W_t'$ be the (singular) measured lamination on $S$ given by the tightened holding map $\psi_t'\col T_t \to T$, where strata are connected components of fibers and the transversal measure is given by the pull-back metric. 
  In addition, let $R_t$ be the collection $\{\phi_t^{-1}(\gam_i)\}_{i = 1}^{n_t}$ of the closed simply connected regions $\phi_t^{-1}(\gam_i)$ in $\ti{S}$. 

As the homotopy $\xi_t\col \ti{S} \to T$ changes continuously in $t \in [0,1]$, 
we can show that the collection of maximal loose subtrees  $\gam_{t, 1}, \gam_{t, 2}, \dots, \gam_{t, n_t}$ of $T_t$  continuously in $t \in [0,1]$, so that the collection $R_t$ changes continuously in $t$. 
To be precise, by continuity, we mean that the subsets
$\phi_t^{-1} (\gam_1) \cup \dots \cup \phi_t^{-1} (\gam_{n_t})$ and $\phi_t^{-1} (\bdr \gam_1) \cup \dots \cup \phi_t^{-1} (\bdr \gam_{n_t})$ of $\ti{S}$ change continuously in the Gromov-Hausdorff topology on the subsets of $\ti{S}$, except that, if $\gam_{t,i}$ maps to a single point in $T$ for some $t \in [0,1]$, we identify the collection $\gam_{t, 1}, \gam_{t, 2}, \dots, \gam_{t, n_t}$ with  the collection minus  $\gam_{t,i}$. 

Therefore, by the continuity of the maximal loose subtrees, $\phi_t$ changes continuously in $t \in [0,1]$, and thus $W'_t$ changes continuously in $t$. 
Since,  in each interval $[u_i, u_{i + 1}]$,  the homotopy $\xi_t$ is supported on  the topological disk $U_j$,  the change of $W_t$ is also supported in $U_j$. 
Therefore one can show moreover, for all $t \in [u_i, u_{i+ 1}]$:
    \begin{itemize}
\item for every arc $c$ in $U_j$ with endpoints on the boundary of $U_j$, the minimum $W'_t$-measure of the homotopy class of $c$ remains the same when the endpoints of $c$ are fixed; 
\item for every arc $c$ in $S \minus U_j$ with endpoints on the boundary of $U_j$, the $W_t'$-measure of  $c$ remains the same;
\item for every loop $\ell$ in $S \minus U_j$ with endpoints on the boundary of $U_j$, the $W_t'$-measure of $\ell$ remains the same.  
\end{itemize}
Therefore, for every loop $\ell$ on  $S$, the tightened measure $W_t'$ gives a constant measure to the homology class of $\ell$,  for all $t \in [u_i,u_{i + 1}]$.
Hence, $W_t'$ on homotopy classes of loops stays constant for all $t \in [0, 1]$.

Since $\psi_X$ is a straight map, $T_X$ contains no loose subtree. 
Therefore $H_X = W_0$. 
Similarly, $H_Y = W_1$.
Hence $H_X = H_Y$ with respect to the normalization above, and thus $[H_X] = [H_Y]$ in $\PML$.
(As $\ep > 0$ is arbitrary, we can also show that $H_t$ is a constant foliation after collapsing.)
\Qed{HorizontalFoliationsCoinside}

 Recall that the translation lengths of loops given by $\zeta\col \pi_1(S) \to \Isom\,T$ is the scaled limit on the translation lengths of $\rho_i\col\pi_1(S) \to \PSL(2, \C)$ as $i \to \infty$.
Since $\psi_X\circ\phi_X\col \ti{S} \to T$ and $\psi_Y\circ \phi_Y\col\ti{S} \to T$ both $\zeta$-equivariant and the translation lengths of $\rho_i$ in the (asymptotically) same scale when $i$ is very large, \Cref{HorizontalFoliationsCoinside} implies the following.
\begin{corollary}\Label{HorizontalMeasuredLaminations}
There are sequences of positive real numbers $k_i$ and $k_i'$, such that 
$\lim_{i \to \infi}\frac{k_i}{k_i'} = 1$ and
$\lim_{i \to \infi} k_i H_{X, \rho_i} = \lim_{i \to \infi} k_i'H_{Y, \rho_i}$ in $\ML$.  
\end{corollary}

\section{Train tracks}\Label{sEuclideanTrainTracks}

\subsection{Train-track graphs}

A train track graph  is a {\sf $C^1$-smooth} graph $\Gamma$ embedded in a smooth surface in the following sense:
\begin{itemize}
\item Each edge of $\Gamma$ is $C^1$-smoothly embedded in the surface.   
\item  At every vertex $v$ of $\Gamma$, the unit vectors at $v$ tangent to the edges starting from $v$ are unique up to a sign, and the opposite unit tangent vectors are both realized by the edges. 
\end{itemize}
A {\sf weight system} $w$ of a train-track graph is an assignment of a non-negative real number $w(e)$ to each edge $e$ of $\Gamma$,  such that at every vertex $v$ of $\Gamma$, letting $e_1, \dots, e_n$ be the edges from one direction and $e_1', e_2', \dots, e_m'$ the opposite direction, the equation $w(e_1) + w(e_2) + \dots + w(e_n) = w(e_1') + w(e_2') + \dots + w(e_n')$ holds.

\subsection{Singular Euclidean surfaces}
{\it A singular Euclidean structure} on a surface is given by a  Euclidean metric with a discrete set of cone points.
In this paper,   all cone angles of singular Euclidean structures are $\pi$-multiples, as we consider singular Euclidean structures induced by holomorphic quadratic differentials.
In addition, by  a {\sf singular Euclidean polygon}, we mean a polygon with geodesic edges and a discrete set of singular points whose cone angles are $\pi$-multiples.
A polygon is {\sf right-angled} if 
the interior angles are $\pi/2$ or $\pi/3$ at all vertices.
A {\sf Euclidean cylinder} is a non-singular Euclidean structure on a cylinder with geodesic boundary. 
By
{a {\sf flat surface}, we mean a singular Euclidean surface with (singular) vertical and horizontal foliations, which intersect orthogonally. }

\begin{definition}\Label{dStaircase}
Let $E$ be a flat surface. 
A curve $\ell$ on $E$  is a {\sf staircase}, if 
 $\ell$ contains no singular point and 
$\ell$ is piecewise vertical or horizontal.
Then, a staircase curve is {\sf monotone} if the angles at the vertices alternate between $\pi/2$ and $3\pi/2$ along the curve, so that it is a geodesic in the $L^\infi$-metric
 (the infinitesimal length is the maximum of the infinitesimal horizontal length and the infinitesimal vertical length).
A staircase curve is {\sf vertically geodesic}, if for every horizontal segment, the angle at one endpoint is $\pi/2$ and the angle at the other endpoint is $3\pi/2$; see Figure \ref{fVertucalGeodesic}.
\begin{figure}
\centering
\begin{minipage}{.5\textwidth}
  \centering
\begin{overpic}[scale=.05
] {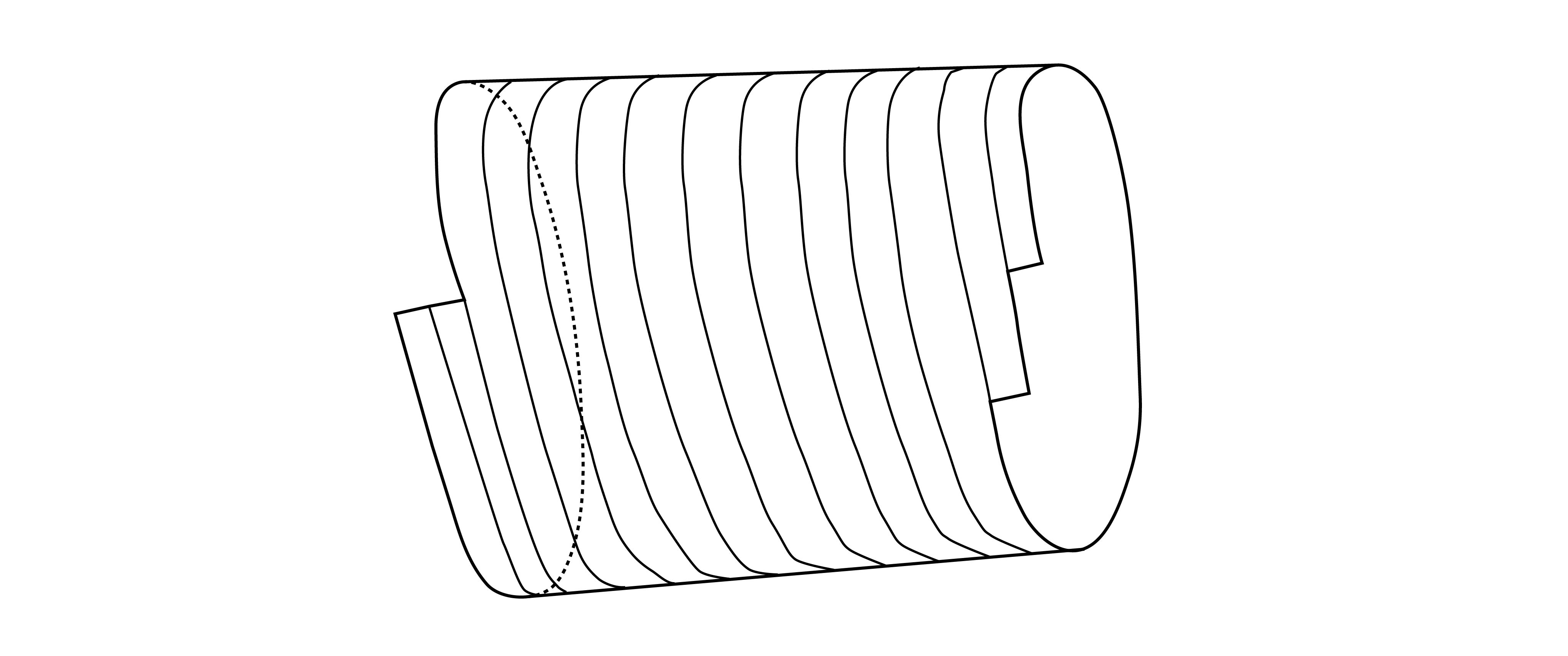} 
      \end{overpic}
\caption{A spiral cylinder  decomposed into rectangles. }\label{fSpiralCylinder}
\end{minipage}%
\begin{minipage}{.5\textwidth}
  \centering
\begin{overpic}[scale=.12
] {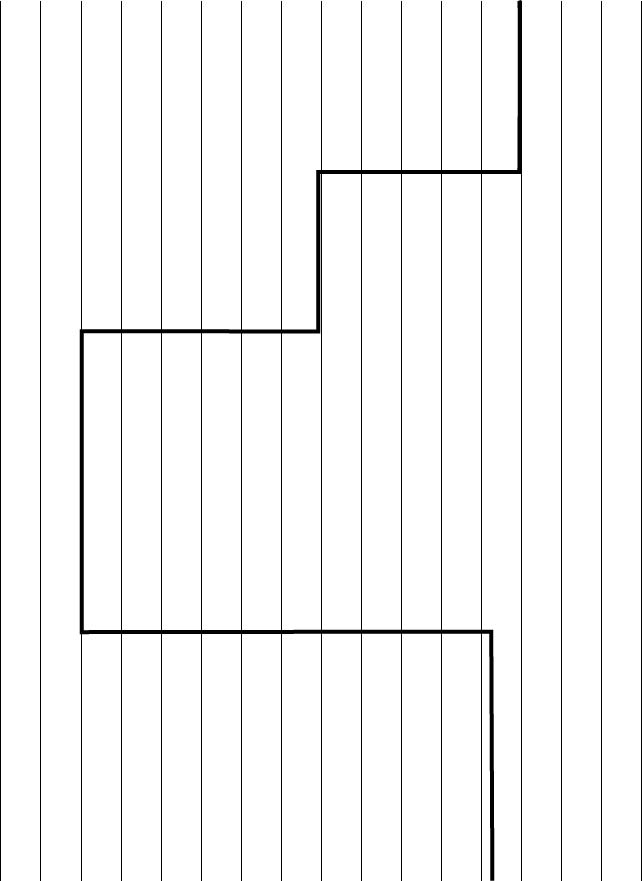} 
      \end{overpic}
\caption{An example of a vertically-geodesic staircase curve on the Euclidean plane with the vertical foliation.}\label{fVertucalGeodesic}
\end{minipage}
\end{figure}

A {\sf staircase surface} is a  flat surface whose boundary components are staircase curves.  
In particular, a staircase polygon is a flat surface homeomorphic to a disk bounded by a staircase curve.
A  ($L^\infty$-){\sf convex staircase polygon} $P$ is a staircase polygon, such that, if $p_1, p_2$  are adjacent vertices of $P$, then at least, one of the interior angles at $p_1$ and $p_2$ is $\pi/2$.
A staircase cylinder $A$ embedded in a flat surface $E$ is a {\sf  spiral cylinder}, if 
 $A$ contains no singular point and
 each boundary component is a monotone staircase loop
(see Figure \ref{fSpiralCylinder}).
\end{definition}

Clearly,  we have the following decomposition. 
\begin{lemma}\Label{SpiralCylinderIntoARectangle}
Every spiral cylinder decomposes into finitely many rectangles when cut along some horizontal segments each starting from a vertex of a boundary component.  (Figure \ref{fSpiralCylinder}.)
\end{lemma}

\subsection{Surface train tracks}\Label{sPolygonalTrainTrack}
Let $F$ be a compact surface with boundary,  such that each boundary component of $F$ is either a smooth loop or a loop with an even number of corner points. 
Then a {\sf (boundary-)marking} of $F$ is an assignment of ``horizontal'' or ``vertical'' to every smooth boundary segment, such that 
 every smooth boundary component is  horizontal and, 
 along every non-smooth boundary component, horizontal edges and vertical edges alternate. 
From the second condition, every boundary component with at least one corner point has an even number of corner points. 

For example,
a marking of a rectangle is an assignment of horizontal edges to one pair of opposite edges and vertical edges to the other pair, and a marking of a $2n$-gon is an assignment of horizontal and vertical edges, such that the horizontal and vertical edges alternate along the boundary. 
Clearly, there are exactly two ways to give a $2n$-gon a marking.
A marking of a flat cylinder is the unique assignment of horizontal components to both boundary components. 

Recall that a {\sf (fat) train track} $T$ is a surface with boundary and corners obtained by gluing marked rectangles $R_i$  along their horizontal edges, in such a way that the identification is given by
 subdividing every horizontal edge into finitely many segments,  pairing up all edge segments, and identifying the paired segments by a diffeomorphism; see for example  \cite[\S 11]{Kapovich-01}.
     
In this paper,  we may allow any marked surfaces as branches. 
\begin{definition}
A {\sf  surface train track} $T$ is a surface having boundary with corners, obtained by gluing marked surfaces $F_i$ in such a way that the identification is given by
(possibly) subdividing each horizontal edge and horizontal boundary circle of $F_i$ into finite segments,  pairing up all segments, and identifying each pair of segments by a diffeomorphism. 

Given a surface train track $T = \cup F_i$, if all branches $F_i$ are cylinders with smooth boundary and polygons, then we call $T$  a {\sf polygonal train track. }
\end{definition}

Suppose that a surface $F$ is decomposed into marked surfaces with disjoint interiors so that the horizontal edges of marked surfaces overlap only with other horizontal edges, and vertical edges overlap with other vertical edges (except at corner points); we call this a {\sf surface train-track decomposition} of $F$. 
Given a train-track decomposition of a surface $F$, the union of the boundaries of its branches is a finite graph on $F$, and we call it {\sf the edge graph}.

Let $F = \cup F_i$ be a train-track decomposition of a surface $F$. 
Clearly the interior of a branch is embedded in $F$, but the boundary of a branch may intersect itself. 
The closure of a branch $F_i$ in $F$ is called the {\sf support} of the branch, and denoted by $\rvert F_i \lvert$, which may not be homotopy equivalent to $F_i$ on $F$.

Next, we consider geometric train-track decompositions of flat surfaces.
Let $E$ be a flat surface, and let  $V$ and $H$ be it vertical and horizontal foliations, respectively. 
Then, when we say that a staircase surface $F$ is on $E$,  we always assume that horizontal edges of $F$ are contained in leaves of $H$  and vertical edges in leaves of $V$. 
Note that a marked rectangle $R$ on $E$ may self-intersect in its horizontal edges, so that it forms a spiral cylinder.
Then
a {\sf staircase train-track decomposition} of a flat surface $E$ is a decomposition of  $E$ into finitely many staircase surfaces on $E$, such that we obtain a surface train-track by gluing those staircase surfaces back only along horizontal edges.
(Note that, in the context of $\CP^1$-structures, the vertical direction is regarded as the stable or stretching direction (see \Cref{Dumas}) and the vertical foliation is carried by this surface train-track.)

More generally,
a {\sf trapezoidal train-track decomposition} of $E$ is a surface train-track decomposition, such that each vertical edge is contained in a vertical leaf and each horizontal edge is a non-vertical line segment disjoint from the singular set of $E$. 

Given a flat surface, we shall construct a canonical staircase train-track decomposition. 
Let $q$ be a holomorphic quadratic differential on a Riemann surface $X$ homeomorphic to $S$. 
Let $E$ be the flat surface given by $q$, which is homeomorphic to $S$.
 As above, let $V, H$ be the vertical and horizontal foliations of $E$. 
Let $E^1$ be the unit-area normalization of $E$, so that $E^1 = \frac{E}{\Area\, E}$. 

Let   $z_1, z_2 \dots z_p$ be the zeros of $q$, which are the singular points of $E$. 
For each $i = 1, \dots, p$, let $\ell_i$ be the singular leaf of $V$ containing $z_i$.
For $r > 0$, let $n_i$ be the closed $r$-neighborhood of $z_i$ in $\ell_i$ with respect to the path metric of $\ell_i$ induced by $E^1$ ({\sf vertical $r$-neighborhood}). 
Let $N_r$ be their union $n_1 \cup \dots \cup n_p$ in $E$, which may not be a disjoint union as a singular leaf may contain multiple singular points.  
If $r > 0$ is sufficiently small, then each (connected) component of $N_r$ is contractible. 
Let $\QD^1(X)$ denote the set of all unit area quadratic differentials on $X$.
Since the set of unit area differentials on $X$ is a sphere, by its compactness, we have the following. 
\begin{lemma}\Label{SmallZeroNbhd}
For every $X$ in $\TT^+ \cup \TT^-$, if $r > 0$ is sufficiently small, then, for all $q \in \QD^1(X)$,  each component of $N_{r}$ is a simplicial tree (i.e. contractible).
\end{lemma}

Fix $X$ in  $\TT^+ \cup \TT^-$, and
let $r > 0$ be the small value given by  Lemma \ref{SmallZeroNbhd}. 
Let $p$ be an endpoint of a component of $N_r$. 
Then $p$ is contained in horizontal geodesic segments, in $E$, of finite length, such that their interiors intersect $N$ only in $p$. 
Let $h_p$ be a maximal horizontal geodesic segment or a horizontal geodesic loop, such that the interior of $h_p$  intersects $N_r$ only in $p$. 
If $h_p$ is a geodesic segment, then the endpoints of $h_p$ are also on $N_r$. 
If $h_p$ is a geodesic loop,  $h_p$ intersects $N_r$ only in $p$. 

Consider the union $\cup_p h_p$ over all endpoints $p$ of $N_r$. 
Then $N_r \cup (\cup_p h_p)$  decomposes $E$ into staircase rectangles and, possibly, flat cylinders.
 Thus we obtain a staircase train track decomposition whose branches are all rectangles.

Next, we construct a polygonal train-track structure of $E$ so that the singular points are contained in the interior of the branches. 
Let $b_i \in \Z_{\geq 3}$ be the {\it balance} of the singular leaf $\ell_i$ at the zero $z_i$, i.e. the number of the segments in $\ell_i$ meeting at the singular point $z_i$. 

We constructed the vertical $r$-neighborhood $n_i$ of the zero $z_i$.
Let $P_i^r$ be the set of points on $E$ whose horizontal distance from $n_i$ is at most  $\sqrt[4]{r}$ ({\sf horizontal neighborhood}).
Then, as $E$ is fixed, if $r > 0$ is sufficiently small, then $P_i^r$ is a convex staircase $2 b_i$-gon whose interior contains $z_i$.  
We say that $P_i^r$ is the {\sf $ (r, \sqrt[4]{r})$-neighborhood} of $z_i$.
Such $(r, \sqrt[4]{r})$-neighborhoods will be used in the proof of Lemma \ref{HomotopyToBeCarried}.

When we vary $q \in \QD^1(X)$, fixing $r$, the convex polygons for different zeros may intersect. 
Nonetheless, by compactness, we have the following.
\begin{lemma}\Label{PolygonalZeroNbhd}
Let $X \in \TT^+ \cup \TT^-$. 
If $r > 0$ is sufficiently small, then,
for every $q \in \QD^1(X)$, 
 each connected component of $P_1^r \cup P_2^r \cup \dots \cup P_{n_q}^r$ is a staircase polygon. 
\end{lemma}
Then, let $r > 0$ and $P^r (= P^r_q)$ be  $P_1^r \cup P_2^r \cup \dots \cup P_{p_q}^r$ as in Lemma \ref{PolygonalZeroNbhd}.
Then, similarly, for each horizontal edge $h$ of $P^r$, let $\hat{h}$ be a maximal horizontal geodesic segment or a horizontal geodesic loop on $E$, such that the interior point of $\hat{h}$ intersects $P^r$ exactly in $h$. 
Then, either 
\begin{itemize} 
\item  $\hat{h}$ is a horizontal geodesic segment whose endpoints are on the boundary of $P^r$, or
\item  $\hat{h}$ is a horizontal geodesic loop intersecting $P^r$ exactly in $h$.
\end{itemize}

Consider the union $\cup_h \hat{h}$ over all horizontal edges $h$ of $P^r$.
Then the union decomposes $E \minus P^r$ into finitely many staircase rectangles and, possibly,  flat cylinders. 
Thus we have a staircase train-track structure, whose branches are polygons and flat cylinders.
Note that the singular points are all contained in the interiors of polygonal branches.

For later use, we modify the train track to eliminate  {\it thin} rectangular branches, i.e. they have short horizontal edges.
Note that each vertical edge of a rectangle is contained in a vertical edge of a polygonal branch. 
Thus, if a rectangular branch $R$ has horizontal length less than $\sqrt[4]{r}$, then naturally glue $R$ with both adjacent polygonal branches along the vertical edges of $R$. 
After applying such gluing for all thin rectangles, we obtain a  train-track structure $t^r$ of $E$. 

\begin{lemma}\Label{PolygonalTraintrackX}
For every $X \in \TT^+ \cup \TT^-$, if $r > 0$ is sufficiently small, then, for every  $q \in \QD^1(X)$,  the branches of the train-track structure $t^r$ on $E$ are staircase polygons and staircase flat cylinder,  and every rectangular branch of $t^r$ has width at least $\sqrt[4]{r}$. 
\end{lemma}

\begin{definition}
Let $E$ be a flat surface. 
A train-track structure $T_1$ is a {\sf refinement} of another train-track structure $T_2$ of $E$, if the $T_1$ is a subdivision of $T_2$ (which includes the case that $T_1 = T_2$).

Let $E_i$ be a sequence of flat surfaces converging to a flat surface $E$.  
Let $T$ be a  train-track structure on a flat surface $E$, and let $T_i$ be a sequence of train-track structures on a flat surface $E_i$ for each $i$.   
Then $T_i$ {\sf converges} to $T$ as $i \to \infi$ if the edge graph of  $T_{k_i}$ converging to the edge graph of $T_\infi$ on $E$ in the Hausdorff topology.
Then $T_i$ {\sf semi-converges} to $T$ as $i \to \infi$ if 
every subsequence $T_{k_i}$ of $T_i$ subconverges to a train-track structure $T'$ on $E$,  such that $T$ is a refinement of  $T'$.
\end{definition}

\begin{lemma}
$t^r_q$ is {\sf semi-continuous} in the Riemann surface $X$ and the quadratic differential $q$ on $X$, and the (small) train-track parameter $r > 0$ given by Lemma \ref{PolygonalTraintrackX}.   
That is, if $r_i \to r$ and $q_i \to q$, then $t^{r_i}_{q_i}$ semi-converges to $t^r_q$ as $i \to \infty$.
\end{lemma}
\begin{proof}
Clearly, the flat surface $E$ changes continuously in $q$. 
Accordingly $P^r$ changes continuously in the Hausdorff topology in $q$ and $r$. 
Then the semi-continuity easily follows from the construction of $t^r_q$.
(Note that $t^r_q$ is{\it not} necessarily continuous since a branch of $t^{r_i}_{q_i}$ may, in the limit,  be decomposed some branches including a rectangular branch of horizontal  length $\sqrt[4]{r}$.)
\end{proof}

\subsection{Straightening foliations on flat surfaces}\Label{sStrighteningFoliations}
Let $E$ be the flat surface homeomorphic to $S$, and let $V$ be its vertical foliation.
Let $V'$ be another measured foliation on $S$. 
 
For each smooth leaf $\ell$ of $V'$, consider its geodesic representative $[\ell]$ in $E$.
If $\ell$ is non-periodic,  the geodesic representative is unique.
Suppose that $\ell$ is periodic.
Then, if $[\ell]$ is not unique, then the set of its geodesic representatives foliates a flat cylinder in $E$. 

Consider all geodesic representatives,  in $E$, of smooth leaves $\ell$ of $V'$, and
 let $[V']$ be the set of such geodesic representatives and the limits of those geodesics. 
 We still call the geodesics of $[V']$ {\sf leaves}.  
 We can regard $[V']$ as a map from a lamination $[V']$ on $S$ to $E$ which is a leaf-wise embedding. 
\section{Compatible surface train track decompositions }\Label{sCompatibleTraintrackDecompositions}

Let $X, Y \in \TT \sqcup \TT^\ast$ with $X \neq Y$.  
Clearly, for each $\rho \in \rchi_X \cap \rchi_Y$,  there are unique $\CP^1$-structures $C_X$ and $C_Y$ on $X$ and $Y$, respectively, with holonomy $\rho$. 
Set  $C_X = (X, q_X)$ and  $C_Y = (Y, q_Y)$, in Schwarzian coordinates, where $q_X \in \QD(X)$ and $q_Y \in \QD(Y)$. 
Then, define $\eta\col \rchi_X \cap \rchi_Y \to \PML \times \PML$ to  be the map taking $ \rho \in \rchi_X \cap \rchi_Y$ to the ordered pair of the projectivized horizontal foliations of  $q_{X, \rho}$ and $q_{Y, \rho}$. 
Let $\Lambda_\infi \sub \PML \times \PML$ be the set of the accumulation points of $\eta$ towards the infinity of $\rchi$ ---
namely, $(H_X, H_Y) \in \Lambda_\infi$ if and only if there is a sequence $\rho_i$ in $\rchi_X \cap \rchi_Y$ which leaves every compact set in $\rchi$ such that $\eta(\rho_i)$ converges to $(H_X, H_Y)$ as $i \to \infty$. 

Let $\Delta \sub \PML \times \PML$ be the diagonal set. 
Then, by Theorem \ref{HorizontalFoliationsCoinside}, $\Lambda_\infi$ is contained in $\Delta$. 
Given a Riemann surface $X$  and a projective measured foliation $H$, by Hubbard and Masur \cite{HubbardMasur79}, there is a unique holomorphic quadratic differential on $X$ such that its horizontal foliation coincides with the measured foliation.  
Let $E_{X, H} = E_{X, H}^1$ denote the unit-area flat surface induced by the differential. 
Given $H_X \in \PML$, let $V_X$ be the vertical measured foliation realized by $(X, H_X)$, and let $V_Y$ be the vertical foliation of $(Y, H_Y)$.

 Noting that a smooth leaf of a (singular) foliation may be contained in a singular leaf of another foliation, 
we let 
$\Delta^\ast$ be 
the set of all $(H_X, H_Y) \in \PML \times \PML$ which satisfies either
\begin{itemize}
\item there is a leaf of $H_X$  contained in a leaf of $V_Y$;
\item there is a leaf of $V_Y$ contained in a  leaf of $H_X$;
\item there is a leaf of $H_Y$  contained in a leaf of $V_X$; or
\item there is a leaf of $V_X$ contained in a leaf of $H_X$.
\end{itemize}
  Then $\Delta^\ast$ is a closed measure-zero subset of $\PML \times \PML$, and disjoint from the diagonal $\Delta$.
 (For the proof of our theorems, we will only consider a sufficiently small neighborhood of $\Delta$, which is disjoint from $\Delta^\ast$.)

\subsection{Straightening maps}\Label{sStraighteningMap}
Fix a transversal pair $(H_X, H_Y) \in (\PML \times \PML) \minus \Delta^\ast$.
Let $p$ be a smooth point in $E_{Y, H_Y}$, and let $\ti{p}$ be a lift of $p$ to the universal cover $\ti{E}_{Y, H_Y}$.

Let $v$ be the leaf of the vertical foliation $\ti{V}_Y$ on $\ti{E}_{Y, H_Y}$ which contains $\ti{p}$, and let $h$ be the leaf of the horizontal foliation $\ti{H}_Y$ on the universal cover which contains $\ti{p}$. 
Then, let $[v]_X$ denote the geodesic representative of $v$ in $\ti{E}_{X, H_X}$, and let $[h]_X$ denote the geodesic representative of $h$ in $\ti{E}_{X, H_X}$.
Since $\ti{E}_{X, H_X}$ is a non-positively curved space, $[v]_X \cap [h]_X$  is a point or a segment of a finite length in $\ti{E}_{X, H_X}$;  let $\st(p)$ be the subset of $E_{X, H_X}$ obtained by projecting the point or a finite segment.  
\subsection{Non-transversal graphs}

Let $E$ be a flat surface with horizontal foliation $H$. 
Let $\ell\col \R \to E$ be a (non-constant) geodesic on $E$ parametrized by arc length. 
A {\sf horizontal segment of $\ell$} is a maximal segment of $\ell$ which is tangent to the horizontal foliation $H$. 
Note that a horizontal segment is, in general, only immersed in $E$. 

Let  $X, Y \in \TT \sqcup \TT^\ast$ with $X \neq Y$, and let  $(H_X, H_Y) \in (\PML \times \PML) \minus \Delta^\ast$. 
For a smooth leaf $\ell_Y$ of $V_Y$,  let $[\ell_Y]_X$ denote the geodesic representative of $\ell_Y$ on the flat surface $E_{X, H_X}$.
The geodesic $[\ell_Y]_X$ is not necessarily embedding and should be regarded as an immersion $\R \to E_{X, H_X}$.
\begin{lemma}\Label{FinitenessOfNon-transversalSegment}
Every horizontal segment $v$ of  $[\ell_Y]_X$ is a segment (i.e. finite length) connecting singular points of $E$.   
 \end{lemma}
\begin{proof}
If $h$ has infinite length, then $\ell_Y$ must be contained in a leaf of $H_Y$.
This contradicts $(H_X, H_Y) \in (\PML \times \PML) \minus \Delta^\ast$.  
\end{proof}
Let $[V_Y]_X$  denote the set of all geodesic representatives of smooth leaves of $V_Y$ on $E_{X, H}$. 
 Let $G_Y \sub E_{X, H_X}$ be the union of (the images of) all horizontal segments of   $[V_Y]_X$.
 Then it follows that $G_Y$ is a finite graph, such that
 \begin{itemize}
\item every connected component of $G_Y$ is contained in a horizontal leaf of $H_X$, and
\item every vertex of $G_Y$ is a singular point of $E_{X, H}$. 
\end{itemize} 
 
\begin{proposition}\Label{DependentSegmentLengthBound}
For all distinct $X, Y \in \TT \sqcup \TT^\ast$ and all $(H_X, H_Y) \in \PML \times \PML \minus \Delta^\ast$, there is $B > 0$, such that,  for all leaves $\ell_Y$ of $V_Y$, every horizontal segment of the geodesic representative $[\ell_Y]_X$ is bounded  by $B$ from above. 
 \end{proposition}
 \begin{proof}
By Lemma \ref{FinitenessOfNon-transversalSegment}, each horizontal segment has a finite length and its endpoints are at singular points of $E_{X, H}$. 
(Although the number of embedded horizontal segments is clearly bounded,  a horizontal segment is in general immersed in $E_{X, H}$.)

    Consider all horizontal segments $s_i\, (i \in I)$ of $[V_Y]_X$. 
    Each $s_i$ is a mapping of a segment of a leaf of $V_Y$ into a leaf of $H_X$.
    Thus, by identifying $s_i$ with the segment of $V_Y$, we regard $\sqcup_{i \in I} s_i$  as a subset of $E_{Y, H}$.
Since $s_i$ is immersed into $G_Y$ for each $i \in I$, we have a mapping from $\bigsqcup_{i \in I} s_i$ to $G_Y$. 
 Therefore, there are a small regular neighborhood $N$ of $G_Y$ and a small homotopy of the mapping $\bigsqcup s_i \to G_Y$, such that  $\bigsqcup s_i$ is, after the homotopy, embedded in $N$ and the endpoints of $s_i$ are on the boundary of $N$ (close to the vertices of $G_Y$ where they map initially). 
Since $G_Y$ is a finite graph and endpoints of $s_i$ map to vertices of $G_Y$, there are only finitely many combinatorial types of horizontal segment  $s_i \to G_Y$.
In particular, the lengths of $s_i$ are bounded from above. 
 \end{proof}

By continuity and the compactness of $\PML$, the uniformness follows: 
\begin{corollary}\Label{UpperBoundForHorizontalSegments}
The upper bound $B$ can be taken uniformly in $H \in \PML$. 
\end{corollary}

 Let $V'_Y$ denote $ [V_Y]_X \minus G_Y$, the set of geodesic representatives $[\ell_Y]_X$ minus their horizontal segments, for all leaves $\ell_Y$ of $V_Y$.
 Then  $V_Y'$ is transversal to $H_X$ at every point, and 
 the angle between them (\S \ref{sAngles}) is uniformly bounded away from zero:
 \begin{lemma}\Label{Semitransversailty}
For every $(H_X, H_Y)  \in \PML \times \PML \minus \Delta^\ast$,
$\angle_{E_{X, H}} ( V'_Y, H_X) > 0$.
\end{lemma} 
\begin{proof}
Suppose, to the contrary, that there is a sequence of distinct points $x_i$ in $V_Y' $ such that $(0 <) \angle_{x_i} (V_Y', H_X) \to 0$ as $i \to \infty$. 
We may, in addition, assume that $x_i$ are smooth points of $E_{X, H_X}$. 
For each $i = 1, 2, \dots$, let $\ell_i$ be a leaf of $[V_Y]_X$ containing $x_i$,  so that $\angle_{x_i}(\ell_i, H_X) \to 0$ as $i \to \infty$.
 Let $s_i$ be a segment of $\ell_i$ containing $x_i$ but disjoint from the singular set of $E_X$.
Clearly,  the angle $\angle_{x_i}(\ell_i, H_X)$ remains the same when $x_i$ moves in $s_i$. 
By the discreteness of the singular set of $E_X$, we may assume that the length of $s_i$ diverges to infinity as $i \to \infty$. 
Then, for sufficiently large $i$,  the segments $s_i$ are all disjoint, and thus $s_i$ must be all parallel as $x_i$ converges to a smooth point.  
This can not happen as $\angle_{x_i}(\ell_i, H_X) \to 0$.  
Therefore $\angle_{E_{X, H}} ( V'_Y, H_X) > 0$.
\end{proof}

\subsection{Train-track decompositions for diagonal horizontal foliations}\Label{sTrainTracksForDiagonal}
\begin{definition}
Let $(E, V)$ be a flat surface. 
Let $T$ be a train track decomposition of $E$. 
A curve $\R \to E$ is {\sf carried by} $T$, if 
 $B$ is a  branch of $T$, then, for every component $s$ of $\ell \cap int B$, 
 both endpoints of $s$ are on different horizontal edges of $B$,
 
 A (topological) lamination on $E$ is {\sf carried by} $T$ if every leaf is carried by $T$. 
\end{definition}

In this paper, a train-track may have ``bigon regions'' which correspond to vertical edges of $T$. 
Thus a measured lamination may be carried by a train track in essentially different ways.  
As a lamination is usually defined up to an isotopy on the entire surface, when a measured lamination is carried by a train-track, we call it a {\sf realization} of the measured lamination.

\begin{definition}
Let $(E, V)$ be a flat surface. 
Let $T$ be a train track decomposition of $E$. 
A geodesic $\ell$ on $E$ is {\sf essentially carried by $T$}, if, for every rectangular branch $B$ of $T$ and every component $s$ of $\ell \cap int B$,

\begin{itemize}
\item both endpoints of $s$ are (different) horizontal edges of $B$, or 
\item the endpoints of $s$ are on adjacent (horizontal and vertical edges) of $B$.
\end{itemize}

The measured foliation $V$ on $E$ is {\sf essentially carried by $T$} if every smooth leaf of $V$ is essentially carried by $T$.
\end{definition}

Because of the horizontal segment, $[V_Y]_X$ is not necessarily carried by $t^r_{X, H}$ even if the train-track parameter $r > 0$ is very small. 
Let $\tt_{X, H}^r$ be the train-track decomposition of $E_{X, H}$ obtained by, for each component of the horizontal graph $G_Y$, taking the union of the branches intersecting the component.
A branch of $\tt_{X, H}^r$ is {\sf transversal} if it is disjoint from $G_Y$, and  {\sf non-transversal} if it contains a component of $G_Y$. 
\begin{lemma}\Label{HomotopyToBeCarried}
For every $(H_X, H_Y) \in (\PML \times \PML) \minus \Delta^\ast$, if $r > 0$ sufficiently small, then 
\begin{enumerate}
\item  $[V_Y]_X$ is essentially carried by $\tt^r_{X, H_X} \eqqcolon \tt_{X, H_X}$, and  \Label{iEssentiallyCarried}
\item  $[V_Y]_X$ can be homotoped along leaves of $H_X$ to a measured lamination $W_Y$ carried by $\tt_{X, H_X}$ so that, by the homotopy, every point of $[V_Y]_X$ either stays in the same branch or moves to the adjacent branch across a vertical edge. \Label{iSmallHomotopy}
\end{enumerate}
Moreover, these properties hold in a small neighborhood of $(H_X, H_Y) \in (\PML \times \PML) \minus \Delta^\ast$.
\end{lemma}

We call $W_Y$ a {\sf realization} of $[V_Y]_X$ on $\tt_{X, H_X}$.

\proof
By the construction of the train track, 
each vertical edge of a rectangular branch has length less than $2 r$, and each horizontal edge has length at least $\sqrt[4]{r}$. 
Then, (\ref{iEssentiallyCarried}) follows from Lemma \ref{Semitransversailty}.

We shall show (\ref{iSmallHomotopy}).
Recall from \S \ref{sPolygonalTrainTrack}, that the construction of $T_{X, H_X}$ started with taking an $r$-neighborhood of the zeros in the vertical direction and then taking points $\sqrt[4]{r}$-close to the neighborhood in the horizontal direction. 
 Therefore, each vertical edge of $\tt_{X,H_X}$ has length at least $\sqrt[4]{r}$ and each horizontal edge has length less than $r$.

Similarly to a Teichm\"uller mapping,
we rescale the Euclidean structure of $E_{X, H_X}$ with area one by scaling the horizontal distance by $\sqrt[4]{r}$ and the vertical distance by $\frac{1}{\sqrt[4]{r}}$, its reciprocal.
Then, by this mapping, the flat surface $E_{X, H}$ is transformed to another flat surface $E'_{X, H_X}$ and the train-track structure $\tt_{X, H_X}$ is transformed to $\tt'_{X, H_X}$. 
Then, the horizontal edges of rectangular branches of  $\tt'_{X, H_X}$ have horizontal length at least $\sqrt{r}$, and the vertical edges have length less than $2 r^{\frac{3}{4}}$. 
Thus, as the train track parameter $r > 0$ is sufficiently small,  the vertical edge is still much shorter than the horizontal edge. 
Note that,  the foliations $V_X$ and $H_X$ persist by the map, except the transversal measures are scaled.

As $r > 0$ is sufficiently small, the geodesic representative $[V_Y]_X'$ of $V_Y$ on $E'_{X, H_X}$ is almost parallel to $V_X$.
Since $N$ is a compact subset of $(\PML \times \PML) \minus \Delta^\ast$,  by Lemma \ref{Semitransversailty},
$\angle_{E_X}(H_X, [V_Y)]_X$ is bounded from below by a positive number uniformly in $\hH  = (H_X, H_Y)\in N$. 
Then, indeed, for every $\upsilon > 0$, if $r > 0$ is sufficiently small, then $\angle_{E_X'}(V_X, [V_Y]'_X) < \upsilon$.

Then, let $\ell$ be a leaf of $V_Y$.  
Let $\ell_X$ be the geodesic representative of $\ell$ in $E_X'$. 
Consider the set $N^v_{\sqrt{r}}$ of points on $E'_X$ whose horizontal distance to the set of the vertical edges of $\tt'_{X, H_X}$ is less than $\sqrt{r}$.
Let $s$ be a maximal segment of $\ell_X$, such that 
 $s$ is contained in $N^h_{\sqrt{r}}$ and
that each endpoint of $s$ is connected to a vertex of $\tt'_{X, H_X}$ by a horizontal segment (which may not be contained in a horizontal edge of $\tt'_{X, H_X}$).
Clearly, if $r > 0$ is sufficiently small,  $s$ does not intersect the same vertical edge twice nor the same branch twice. 
\begin{claim}\Label{ApproximaiteStaircase}
There is a staircase curve $c$ on $E'_{X, H_X}$, such that 
\begin{itemize}
\item $c$ is $r^{\frac{1}{2}}$-close to $s$ in the horizontal direction,
\item each vertical segment of $c$ is a vertical edge of $\tt_{X, H_X}'$, and 
\item each horizontal segment of $c$ contains no vertex of $\tt_{X, H_X}'$ in its interior.
\end{itemize}
(See Figure \ref{fApproximatingVerticalEdges}.)
\end{claim}   

Pick finitely many segments $s_1, \dots, s_n$ in leaves of $[V_Y]_X'$ as above, such that 
if a vertical edge $v$ of $\tt_{X, H_X}'$ intersects $[V_Y]_X'$, then there is exactly one $s_i$ which is $r^{\frac{3}{4}}$-Hausdorff close to $v$. 
Let $c_1, \dots, c_n$ be their corresponding staircase curves on $E'_X$. 

Then, we can homotope $[V_Y]_X$ in a small neighborhood of the region $R_i$ bounded by $s_i$ and $c_i$, such that, while homotoping, the leaves do not intersect $s_i$, and that
the homotopy moves each point horizontally (Figure \ref{fHomotopyToBeCarried}).

Each point on $[V_Y]_X'$ is homotoped at most to an adjacent branch. 
Then, after this homotopy, $[V_Y]_X'$ is carried by $\tt'_{X, H_X}$. 
This homotopy induces a desired homotopy of  $[V_Y]_X$.

\begin{figure}
\centering
\begin{minipage}{.5\textwidth}
  \centering
\begin{overpic}[scale=.2
] {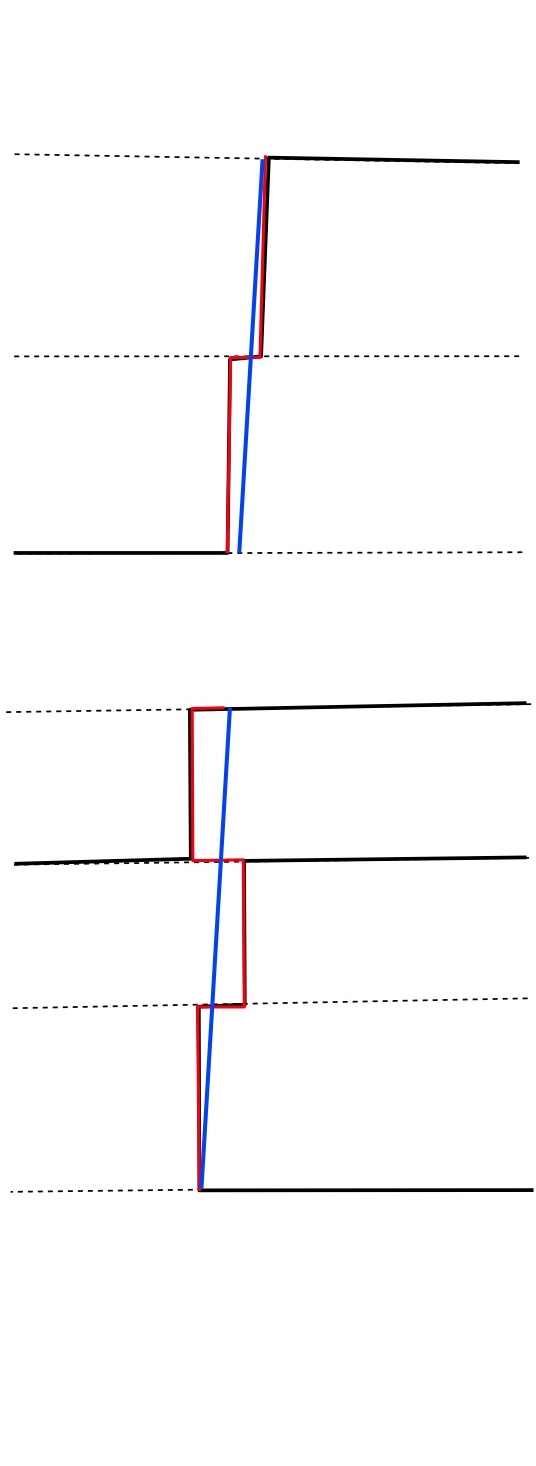} 
      \end{overpic}
\caption{Examples of staircase curves given by Claim \ref{ApproximaiteStaircase}.}\label{fApproximatingVerticalEdges}
\end{minipage}%
\begin{minipage}{.5\textwidth}
  \centering
\begin{overpic}[scale=.05
] {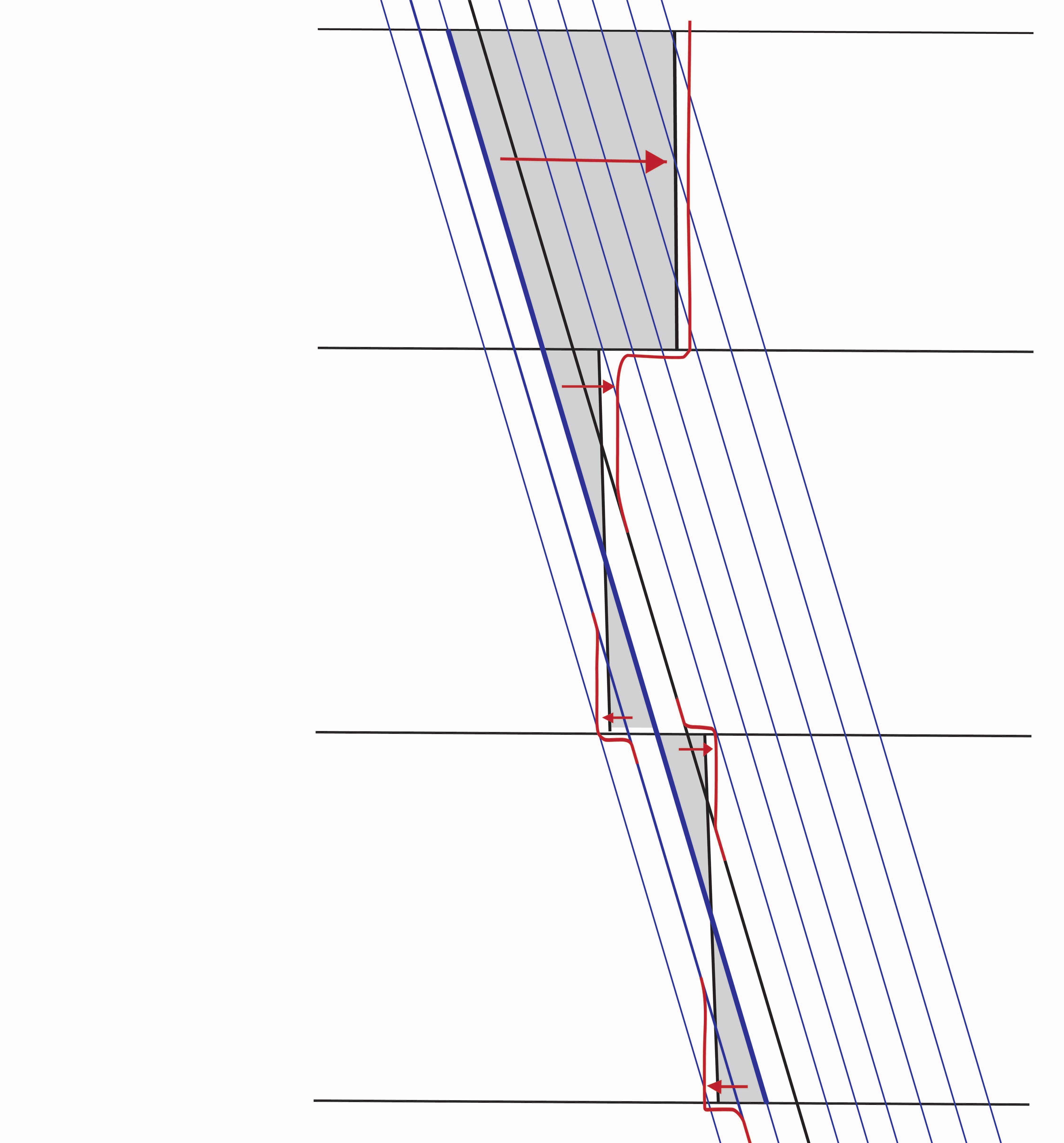} 
      \end{overpic}
\caption{A homotopy to push $[V_Y]_X$ out of the region $R_i$.}\label{fHomotopyToBeCarried}
\end{minipage}
\end{figure}
\Qed{HomotopyToBeCarried}  

\begin{figure}
\begin{minipage}{.5\textwidth}
\centering
\begin{overpic}[scale=.2
] {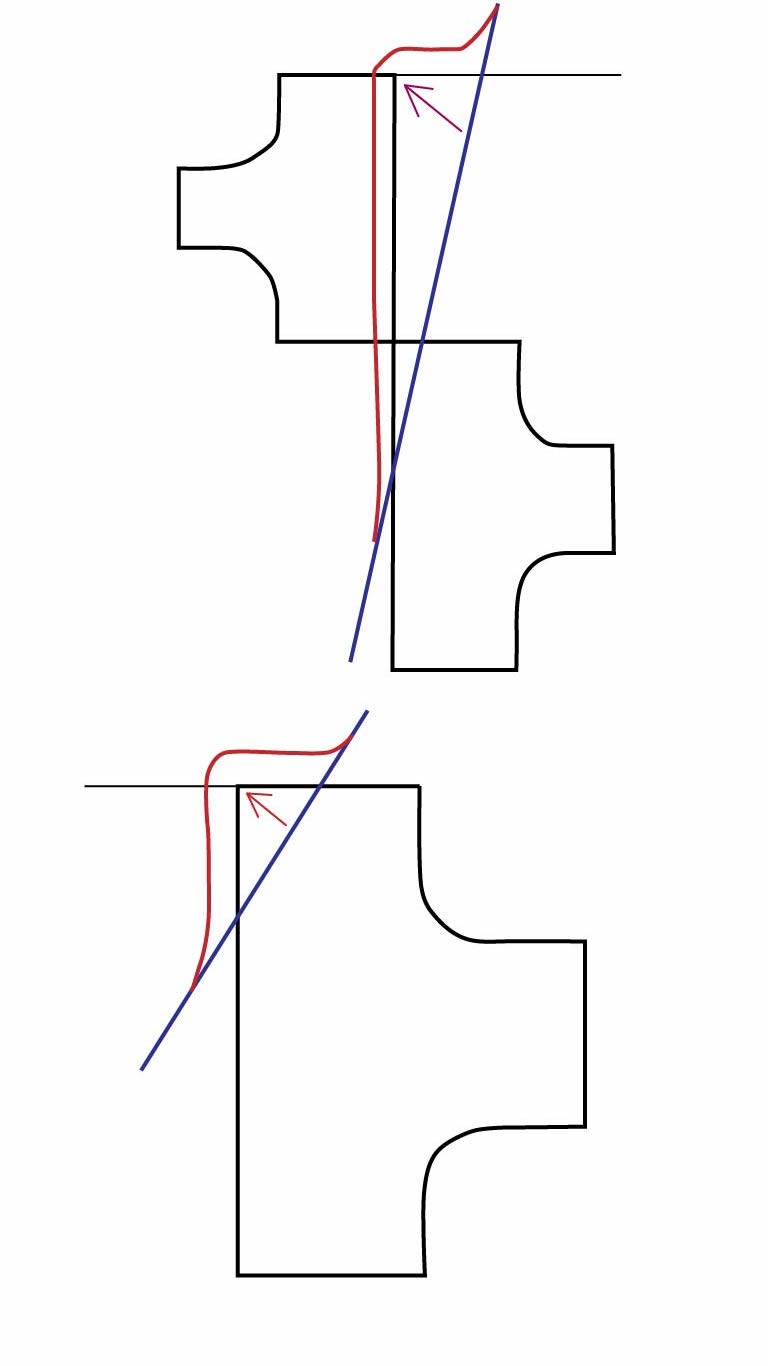} 
      \end{overpic}
\caption{}\label{fHomotopyToBeCarriedTwo}
\end{minipage}%
\begin{minipage}{.5\textwidth}
  \centering
\begin{overpic}[scale=.15
] {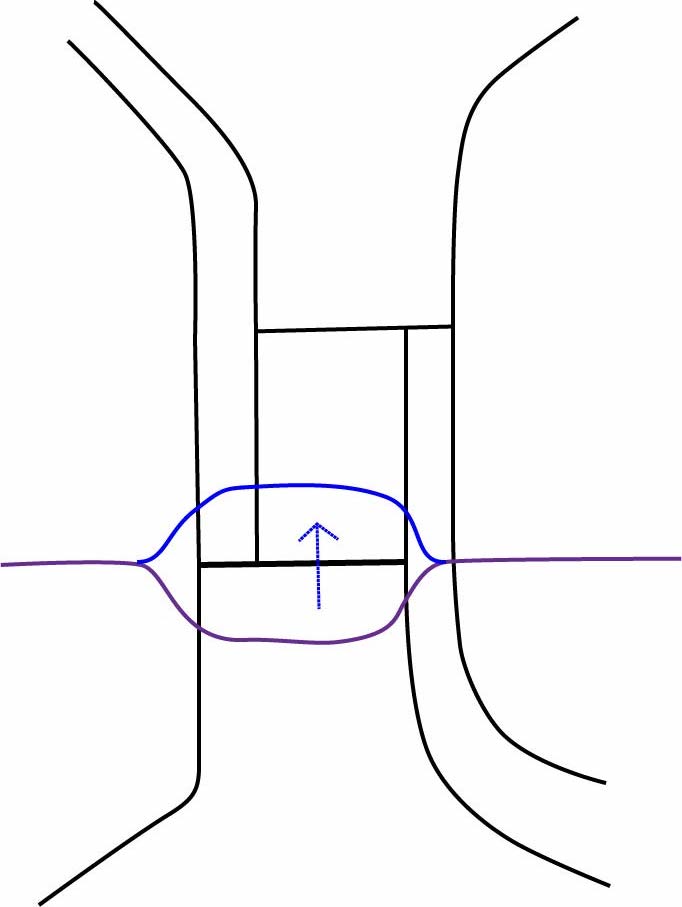} 
      \end{overpic}
\caption{A shifting across a vertical slit.}\label{fShifting}
\end{minipage}
\end{figure}

   Let $\hH = (H_X, H_Y) \in \PML \times \PML \minus \Delta^\ast$ and 
    $W_Y$  denote the realization of $[V_Y]_X$ on $\tt_{X, H_X}$ given by Lemma \ref{HomotopyToBeCarried}. 

A measured lamination in $\PML$ is defined up to an isotopy of the surface. 
The union of the vertical edges of $\tt_{X, H_X}$ consists of disjoint vertical segments. 
Each vertical segment of the union is called a {\sf (vertical) slit}. 
Then, a measured lamination can be carried by a train track in many different ways by homotopy across slits: 
\begin{definition}[Shifting]
Suppose that $T$ is a train-track structure of a flat surface $E$, and let $L_1$ be a realization of $L \in \ML$ on $T$. 
For a vertical slit $v$ of $T$, consider the branches on $T$ whose boundary intersects $v$ in a segment. 
A {\sf shifting} of $L_1$ across $v$ is a homotopy of   $L_1$ on $E$ to another realization $L_2$ of $L$ which reduces the weights of the branches on one side of $v$ by some amount and increases the weights of the branches on the other side of $v$ by the same amount (Figure \ref{fShifting}). 
Two realizations of $L$ on $T$ are {\sf related by shifting} if they are related by simultaneous shifts across some vertical slits of $T$. 
\end{definition}

    The homotopy of $[V_Y]_X$ in \Cref{HomotopyToBeCarried} moves points at most to adjacent branches in the horizontal direction.
Thus we have the following.
    \begin{lemma}\Label{UniqueUpToShifting}
In Lemma \ref{HomotopyToBeCarried}, the realizations given by different choices of $s_i$  are related by shifting. 
\end{lemma}

\begin{proposition}\Label{UniqueRealizationUpToShifting}
 Let $\hH_i = (H_{X, i}, H_{Y, i})$ be a  sequence in $\PML \times \PML \minus \Delta^\ast$ converging to  $\hH = (H_X, H_Y)$ in  $\PML \times \PML \minus \Delta^\ast$.  
 Let $W_i$ be a realization of $[V_{Y_i}]_{X_i}$ on $\tt_{X, H_{X, i}}$,  and let $W$ be a realization of $[V_Y]_X$ on $\tt_{X, H_X}$ given by Lemma \ref{HomotopyToBeCarried}.
 Then, 
 a limit of the realization $W_i$ and the realization $W$ are related by shifting across vertical slits. 
\end{proposition}

\begin{proof}
By the semi-continuity of $\tt_{X, H_X}$ in $H_X$ (\Cref{EuclideanTraintarck}), the limit of the train tracks $\tt_{X, H_{X, i}}$ is a subdivision of $\tt_{X, H_X}$.   
Let $s_{i, 1}, \dots, s_{i, k_i}$ be the segments from the proof of  Lemma \ref{HomotopyToBeCarried} which determine the realization $W_i$.
The segment $s_{j, i}$ converges up to a subsequence. 
Then, the assertion follows from Lemma \ref{UniqueUpToShifting}. 
\end{proof}

In summary, we have obtained the following.
\begin{boxedlaw}{13cm}

\begin{proposition}[Staircase train tracks]\Label{EuclideanTraintarck}
For all distinct $X, Y \in \TT \sqcup \TT^\ast$ and a compact neighborhood ${\bf N}_\infi$ of $\Lambda_\infi$ in $(\PML \times \PML) \minus \Delta^\ast$, 
if the train-track parameter $r  > 0$ is sufficiently small, then, for every $\mathbf{H}= (H_X, H_Y)$  of ${\bf N}_\infi$,
the staircase train track $\tt^r_{X, V_X}$  satisfies the following: 
 \begin{enumerate}
\item  $\tt^r_{X, H_X}$ changes semi-continuously in $\hH \in N_\infi$.   \Label{iContinuityOfVYinTX}
\item  $V_Y$ is essentially carried by $\tt^r_{X, H_X}$, and its realization on $\tt^r_{X, H_X}$ changes continuously $\hH \in N_\infi$, up to shifting across vertical slits. \Label{iRealzationUpToShifting} 
\end{enumerate}
\end{proposition}
\end{boxedlaw}

\subsection{An induced train-track structure for diagonal horizontal foliations}\Label{sInducedTraintrack}
We first consider the diagonal case when $H_X = H_Y \eqqcolon H \in \PML$.  
We have constructed a staircase train track decomposition $\tt_{X, H}$ of $E_{X, H}$.
Moreover, the geodesic representative $[V_Y]_X$ is essentially carried by $\tt_{X, H}$. 
Thus, we homotope $[V_Y]_X$ along leaves of $H_X$, so that it is carried by the  train track $\tt_{X, H}$ (\Cref{HomotopyToBeCarried}).
Let $W_Y$ denote this topological lamination being carried on $\tt_{X, H}$ which is homotopic to $[V_Y]_X$.
 
From the realization $W_Y$ on $\tt_{X, H}$, we shall construct a polygonal train-track structure on $E_{Y, H_Y}$.  
The flat surfaces $E_{X, H}$ and $E_{Y, H}$  have the same horizontal foliation, and the homotopy of $[V_Y]_X$ to $W_Y$ is along the horizontal foliation. 
Therefore,  for each rectangular branch $R_X$ of $\tt_{X, H}$, if the weight of $W_Y$ is positive, by taking the inverse-image of the straightening map $\st \col E_{Y, H} \to E_{X, H}$ in \S \ref{sStraighteningMap}, we obtain a corresponding rectangle $R_Y$ on $E_{Y, H}$ whose vertical length is the same as $R_X$ and horizontal length is the weight. 
 Note that an edge of $R_Y$ may contain a singular point of $E_{Y, H_Y}$. 

Next let $P_X$ be a polygonal branch of $\tt_{X, H}$. 
Similarly, let $P_Y$ be the inverse-image of $P_X$ by the straighten map.
Note that $P_Y$ is not necessarily homeomorphic to $P_X$. 
In particular, $P_Y$ can be the empty set, a staircase polygon which may have a smaller number of vertices than  $P_X$.
Moreover,  $P_Y$  may be disconnected (\Cref{fBranchWithDisjointInterior}). 
Then, we have a (staircase) polygonal train-track decomposition $\tt_{Y, H}$ of $E_{Y, H}$.
By convention, non-empty $P_Y$, as above, is called a {\sf branch} of $\tt_{Y, H}$ corresponding to $P_X$ (which may be disconnected). 
In comparison to $\tt_{X, H}$, the one-skeleton of $\tt_{Y, H}$ may contain some singular points of $E_{Y, H_Y}$. 
Since $\tt_{Y, H}$ changes continuously in the realization  $W_Y$ of $[V_Y]_X$ on $\tt_{X, H}$, 
the semi-continuity of $\tt_{X, H}$ (\Cref{EuclideanTraintarck} (\ref{iContinuityOfVYinTX})) gives a semi-continuity of $\tt_{Y, H}$.
\begin{lemma}
$\tt_{Y, H}$ changes semi-continuously in the horizontal foliation $H$ in $\PML$ and the realization $W_Y$ of $[V_Y]_X$ on $\tt_{X, H}$.
\end{lemma}

    \subsection{Filling properties}
\begin{lemma}\Label{Filling}
Let $X \neq Y \in \TT \sqcup \TT^\ast$.
For every diagonal $H_X = H_Y$, every component of $H_{X, H} \minus [V_Y]_X$ is contractible, i.e. a tree.  
\end{lemma}
 \begin{proof}
 Recall that $H_Y$ and $V_Y$ are the horizontal and vertical foliations of the flat surface $E_{Y, H_Y}$.
  Then, since $H_X = H_Y$, the lemma follows. 
\end{proof}  

A {\sf horizontal graph} is a connected graph embedded in a horizontal leaf (whose endpoints may not be at singular points). 
  Then, Lemma \ref{Filling} implies the following.
\begin{corollary}\Label{SmallTransversalMeasureThenContractible}
Let $X \neq Y \in \TT \sqcup \TT^\ast$.
For every diagonal pair $H_X = H_Y$, let $r> 0$ be the train-track parameter given by Lemma \ref{HomotopyToBeCarried}. 
Then, for sufficiently small $\ep > 0$, if a horizontal graph $h$ of  $H_X$ has total transversal measure less than $\ep$ induced by the realization $W_Y$, then $h$ is contractible. 
\end{corollary}

By continuity, 
\begin{proposition}\Label{ContractibleHorizontalArcs}
  There is a neighborhood $N$ of the diagonal $\Delta$ in $\PML \times \PML$ and $\ep > 0$ such that, if the train-track parameter $r > 0$ is sufficiently small, then for every $(H_X, H_Y) \in N$, if a horizontal graph $h$  of $H_X$ has total transversal measure less than $\del$ induced by $W_Y$, then $h$ is contractible. 
\end{proposition}

\subsection{Semi-diffeomorphic surface train-track decompositions}\Label{sCompatibleEuclideanTraintracks}
\subsubsection{Semi-diffeomorphic train tracks for diagonal foliation pairs}\Label{sOnTheDiagonal}
\begin{definition}
Let $F_1$ and $F_2$ be surfaces with staircase boundary. 
Then $F_1$ is {\sf semi-diffeomorphic} to $F_2$, if there is a homotopy equivalence $\phi\col F_1 \to F_2$ which collapses some horizontal edges of $F_1$ to points: 
To be more precise, 
\begin{itemize}
\item the restriction of $\phi$ to the interior $\Int F_1$ is a diffeomorphism onto the interior $\Int F_2$; 
\item $\phi$ takes  $\bdr F_1$ to $\bdr F_2$, and  $\Int F_1$ to $\Int F_2$;
\item for every vertical edge $v$ of $F_1$, the map $\phi$ takes $v$ diffeomorphically onto a vertical edge or a segment of a vertical edge in $F_2$;
\item for every horizontal edge $h$ of $F_1$, the map $\phi$ takes $h$ diffeomorphically onto a horizontal edge of $F_2$ or collapses $h$ to a single point on a vertical edge of $F_2$.
\end{itemize}

Let $T$ and $T'$ be train-track structures of flat surfaces $E$ and $E'$, respectively, on $S$. 
Then $T$ is {\sf semi-diffeomorphic} to $T'$, if there is a marking preserving continuous map $\phi \col E \to E'$, such that, 
\begin{itemize}
\item $T$ and $T'$ are homotopy equivalent by $\phi$ (i.e. their 1-skeletons are homotopy equivalent), and 
\item for each branch $B$ of $T$, there is a corresponding branch $B'$ of $T'$ such that $\phi | B$ is a semi-diffeomorphism onto $B'$.
\end{itemize}
\end{definition}
In \S \ref{sInducedTraintrack}, for every $H \in \PML$, we constructed a staircase train-track structure $\tt_{Y, H}$ of the flat surface $E_{Y, H}$ with staircase boundary from a realization $W_Y$ of $[V_Y]_X$ on the train-track  structure  $\tt_{X, H}$ of $E_{X, H}$. 
However, when a branch $B_X$ of $\tt_{X, H}$ corresponds to a branch $B_Y$ of $\tt_{Y, H}$, in fact, $B_Y$ might not be connected, and in particular not semi-diffeomorphic to $B_X$ (\Cref{fBranchWithDisjointInterior}, Left).
In this section, we modify $\tt_{X, H}$ and $\tt_{Y, H}$ by gluing some branches in a corresponding manner, so that corresponding branches are semi-diffeomorphic after a small perturbation.

\begin{figure}
\begin{overpic}[scale=.059,
] {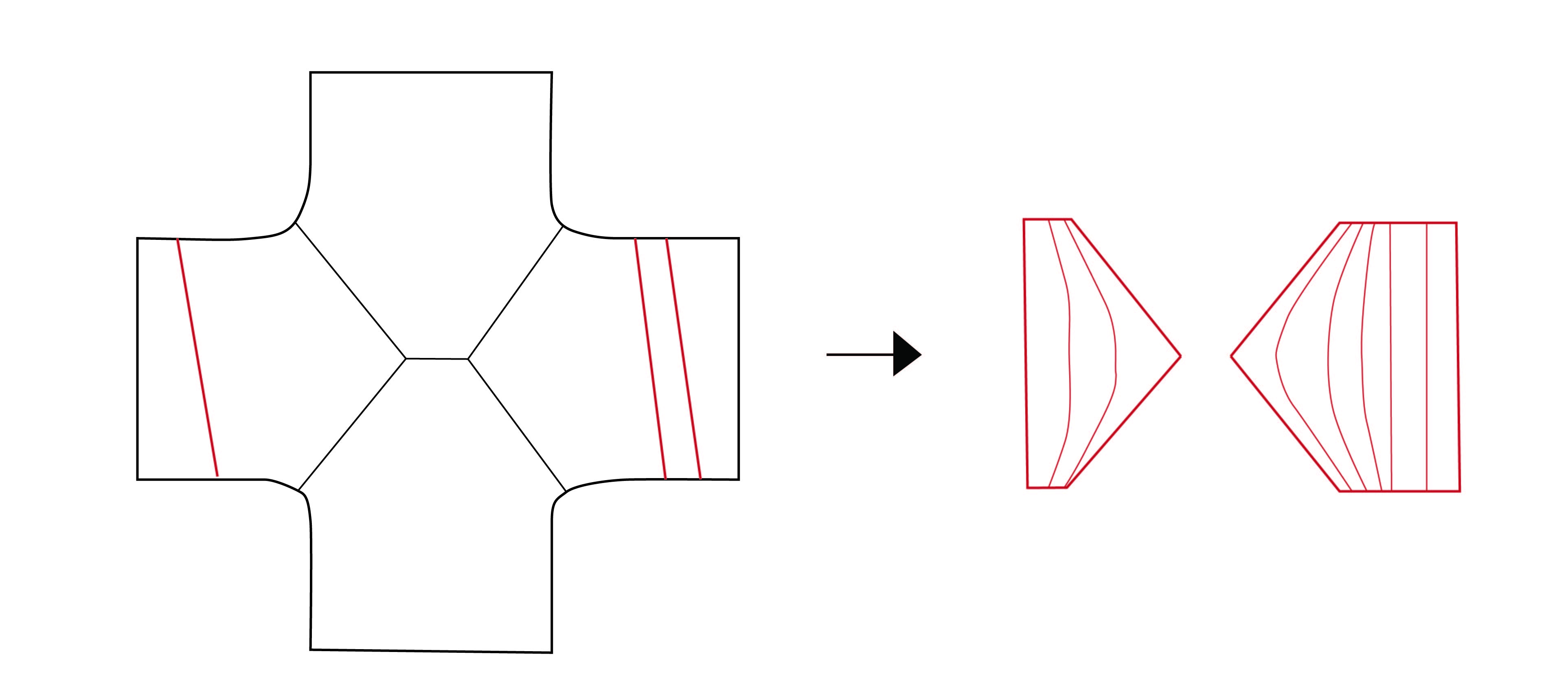} 
  \put(25 ,28 ){\textcolor{black}{$P_X$}}  
 \put(73, 27 ){\textcolor{RedOrange}{$P_Y$}}  
      \end{overpic}
\begin{overpic}[scale=.059
] {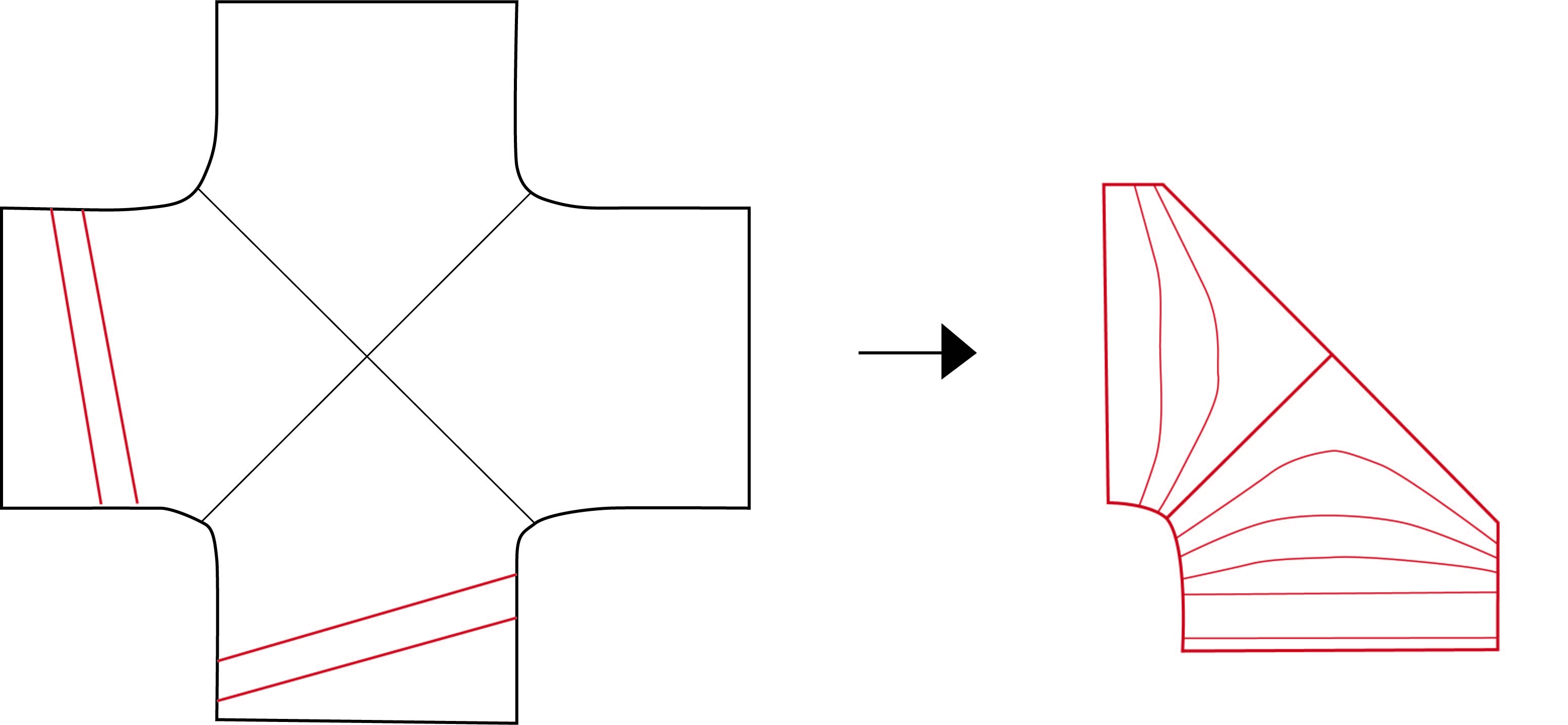} 
  \put(21 ,32 ){\textcolor{black}{$P_X$}}  
 \put(82, 28 ){\textcolor{RedOrange}{$P_Y$}}  
      \end{overpic}
\caption{Some non-diffeomorphic correspondences of branches.}\Label{fBranchWithDisjointInterior} 
\end{figure}
  Let $v$ be a {\sf (minimal) vertical edge} of $\tt_{X, H}$, i.e. a vertical edge not containing a vertex in its interior.  
 Let $B_X$ be a branch of $\tt_{X, H}$ whose boundary contains $v$. 
 Suppose that $\alpha$ is an arc in $B_X$ connecting different horizontal edges of $B_X$. 
 Then, we say that $v$ and $\alpha$ are {\sf vertically parallel 
} in $B_X$ if
\begin{itemize}
\item $\alpha$ is homotopic in $B_X$ to an arc $\alpha'$ transversal to the horizontal foliation $H | B_X$, keeping its endpoints on the horizontal edges,  and
\item $v$ diffeomorphically projects into $\alpha'$ along the horizontal leaves $H_X | B_X$ (see Figure \ref{fHorizontallyParallel}).  
\end{itemize} 
   The {\sf $W_Y$-weight} of $v$ in $B_X$ is the total weight of the leaves of $W_Y | B_X$ which are vertically parallel 
 to $v$. 
   \begin{figure}
\begin{overpic}[scale=.15
] {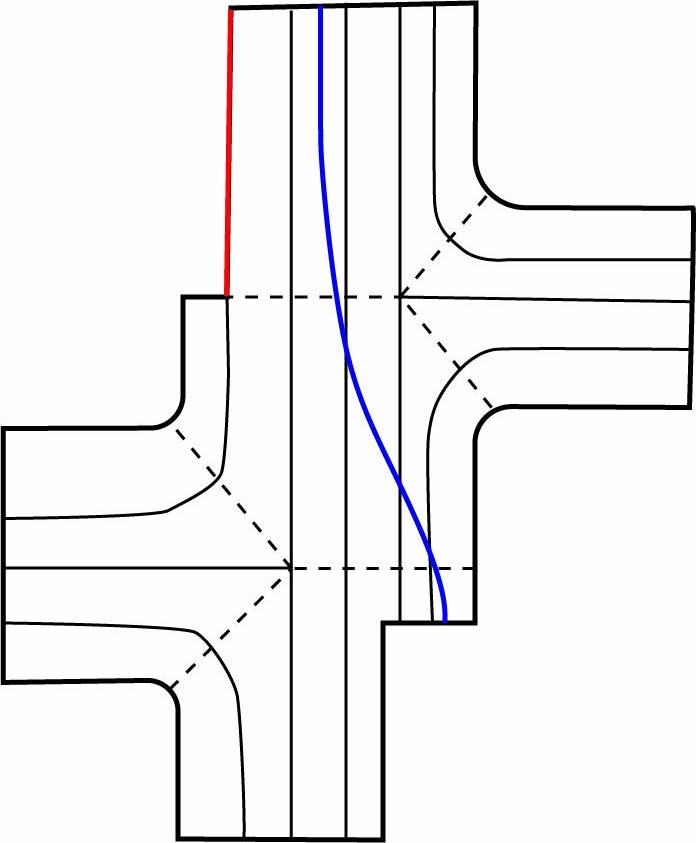} 
\put(20, 80 ){$\tcr{v}$}  
\put(35 ,40 ){\contour{white}{$\tcb{\alpha'}$}} 
      \end{overpic}
\caption{The curve $\alpha'$ is vertically parallel 
 to $v$.}\label{fHorizontallyParallel}
\end{figure}

Let $w$ be the $W_Y$-weight of $v$ in $B_X$.
Then, there is a staircase rectangle in $B_Y$ such that a vertical edge corresponds to $v$ and the horizontal length is $w$.

  Consider a horizontal arc $\alpha_h$ in $B$ connecting a point on $v$ to a point on another vertical edge of $B$;  clearly, the transversal measure of $W_Y$ of $\alpha_h$ is a non-negative number. 
  Then, the {\sf $W_Y$-weight} of  $v$ in $B$ is the minimum of the $W_Y$-transversal measures of all such horizontal arcs $\alpha_h$ starting from $v$. 

Fix $0 < \del < r$ to be a sufficiently small positive number.  
 We now consider both branches  $B_1, B_2$  of $\tt_{X, H}$ whose boundary contains $v$. 
 Suppose that, the $W_Y$-weight of $v$ is less than $\del$ in $B_i$ for both  $i = 1,2$; then, glue $B_1$ and $B_2$ along $v$, so that $B_1$ and $B_2$ form a single branch. 
Let $T_{X, H}^{r, \del}$, or simply $T_{X, H}$, denote the train-track structure of $E_{X, H}$ obtained by applying such gluing, simultaneously,  branches of $\tt_{X, H}$ along all minimal vertical edges satisfying the condition. 
Then, since $\tt_{X, H}$ is a refinement of $T_{X, H}$, the realization $W_Y$ of $[V_Y]_X$ on $\tt_{X, H}$ is also a realization on  $T_{X, H}$. 
Similarly, let $T_{Y, H}^{r ,\del}$, or simply $T_{Y, H}$,  be the train-track structure of $E_{Y, H}$ obtained by the realization $W_Y$ on $T_{X, H}$; then $\tt_{Y, H}$ is a refinement of $T_{Y, H}$. 
   \Cref{Filling} implies the following. 
\begin{lemma}\Label{TransversalBranches}
Every transversal branch of $T_{X, H}$ has a non-negative Euler characteristic. 
\end{lemma}

Let $B$ be a branch of $T _{Y, H}$, and let $v$ be a minimal vertical edge of $T_{Y, H}$ contained in the boundary of $B$. 
 Let $B'$ be the branch of $T_{Y, H}$ adjacent to $B$ across $v$. 
Suppose that the $W_Y$-weight of $v$ is less than $\del$ in $B$. 
 Then, it follows from the construction of $T_{X, H}$, that there is a staircase rectangle $R_v$ in $B'$, such that the horizontal length of $R_v$ is $\del/3$ and that $v$ is a vertical edge of $R_v$. 
Let $v$ be a vertical edge of $B$. 
Then we enlarge $B$ by gluing the rectangle $R_v$ along $v$, and we remove $R_v$ from $B'$ (\Cref{fCutAndPasteRectangle})---  this cut-and-paste operation transforms $T_{Y, H}$ by pushing the vertical edge $v$ by $\del/3$ into $B'$ in the horizontal direction.
\begin{figure}
\begin{overpic}[scale=.2
]{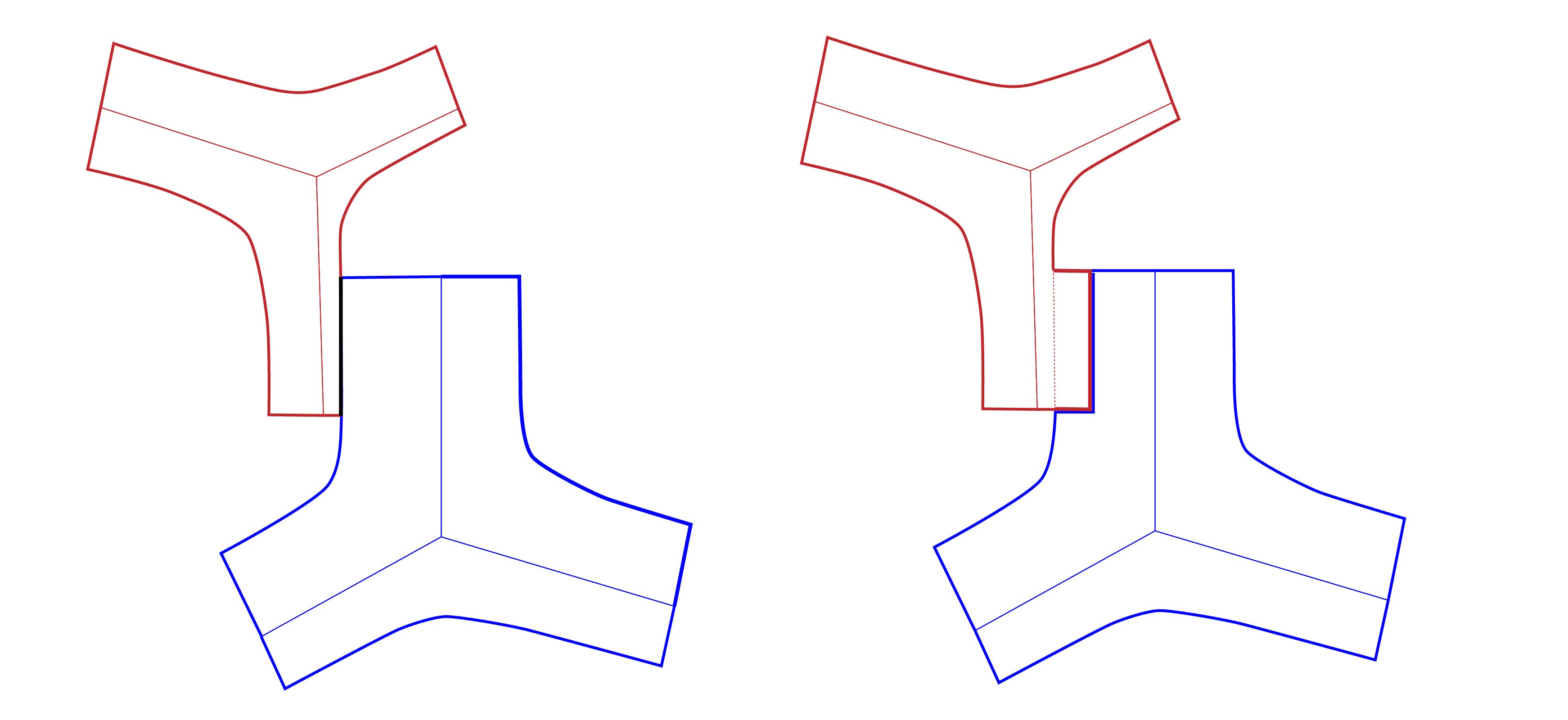} 
\linethickness{.3pt}
\put(70,30){\color{Red}\vector(-1,-1.8){1.7}}
\put(69 ,30.5 ){$\textcolor{red}{R_v}$}  
\put(23 ,22 ){$v$}  
\linethickness{1pt}
\put(45,20){\color{Black}\vector(1,0){10}}
\put(15 , 35){\contour{white}{$\textcolor{red}{B}$}}  
\put(26 , 15 ){\contour{white}{$\textcolor{blue}{B'}$}}  
      \end{overpic}
      \caption{A rectangle exchange across a vertical edge $v$.}\label{fCutAndPasteRectangle}
\end{figure}      
For all minimal vertical edges $v$ of $T_{Y, H}$ whose $W$-weights are less than $\del$ as above, we apply such modifications simultaneously and obtain 
 a train-track structure $T'_{Y, H}$  of $E_{Y, H}$ homotopic to $T_{Y, H}$. 
 (We push weight only $\del/ 3$ across a vertical edge, since, if another  $\del/3$ is pushed out across the opposite vertical edge,  at least $\del/ 3$-weight remains left.)  

\begin{lemma}\Label{TforY}~
\begin{itemize}
\item  The edge graph of  $T_{Y, H}'$ is, at least, $\frac{\del}{ 3}$ away from the singular set of $E^1_{Y, H}$;
\item $T_{Y, H}'$ is $\del$-Hausdorff close to $T_{Y, H}$ in  $E^1_{Y, H}$;
\item $T_{X, H}$ is semi-diffeomorphic to $T_{Y, H}'$;
\item $T_{X, H}$ changes semi-continuously in $H$;
\item $T_{Y, H}$ changes semi-continuously in $H$, and the realization of $W$ on $T_{X,  H}$.
\end{itemize}
\end{lemma}
\begin{proof}

The first three assertions follow from the construction of $T_{X, H}$ and $T_{Y, H}$.
The semi-continuity of $T_{X, H}$ is given by its construction and the semi-continuity of $\tt_{X, H}$ (\Cref{EuclideanTraintarck}).
Similarly, the semi-continuity of $T_{Y, H}$ follows from its construction and the semi-continuity of $\tt_{Y, H}$.
\end{proof}

\begin{figure}
\begin{overpic}[scale=.06, 
] {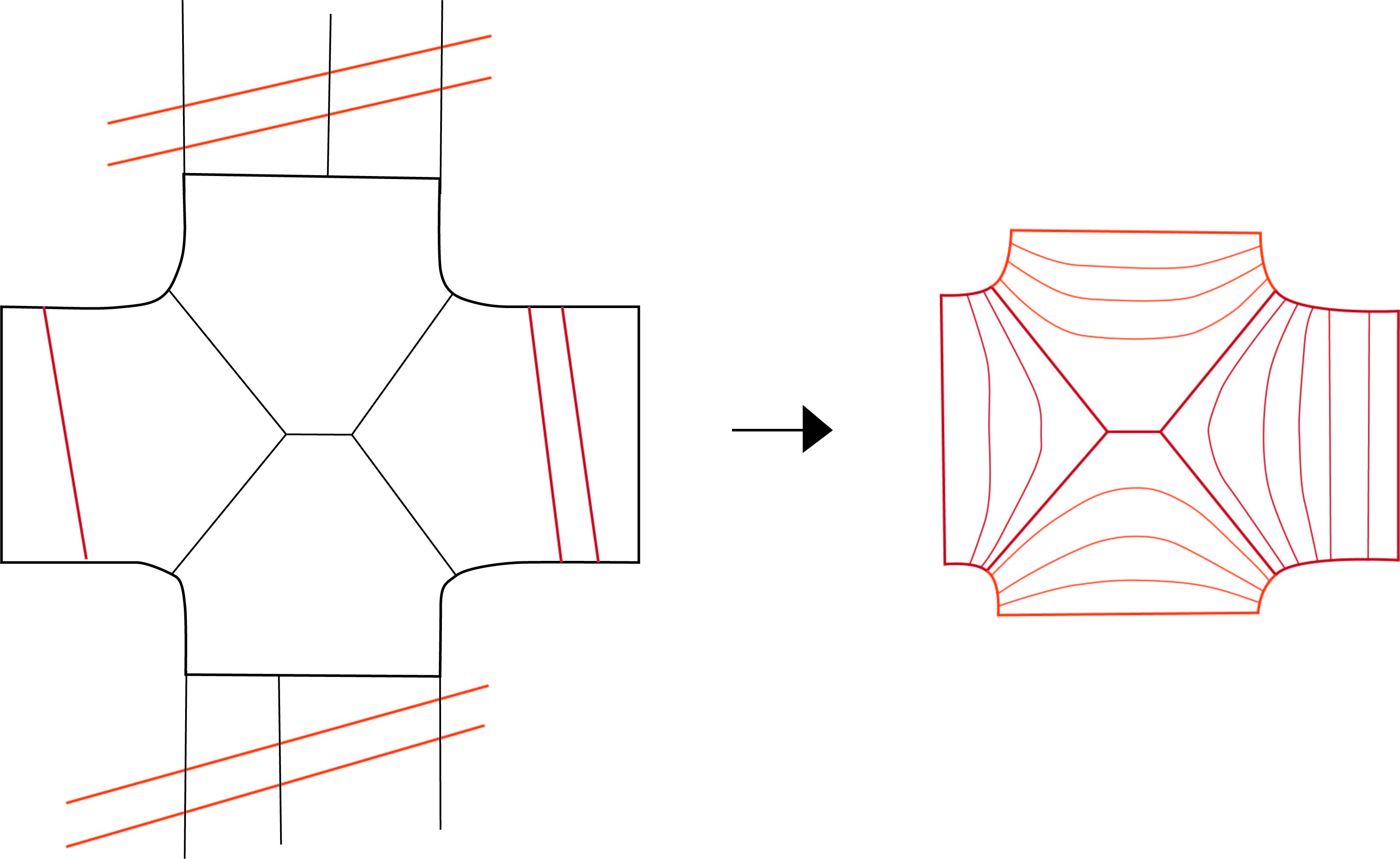} 
    \put(77 , 11){\textcolor{orange}{$B'$}}  
      \end{overpic}
\caption{}\label{fEnlargedYBranch}
\end{figure}

\subsubsection{Semi-diffeomorphic train-tracks for almost  diagonal horizontal foliations}\Label{sSeimidiffeoTraintrackNearDiagonal}
In this section, we extend the construction form \S\ref{sOnTheDiagonal} to the neighborhood of the diagonal $(\PML \times \PML) \minus \Delta^\ast$.
By Lemma \ref{Semitransversailty},
for every compact neighborhood $N$ of the diagonal $\Delta$ in $(\PML \times \PML) \minus \Delta^\ast$,
there is $\del > 0$, such that 
$$\angle_{E_{X, H_X}} ([V_Y]_X, H_X) > \del$$ 
for all $(H_X, H_Y) \in N$, where $V_Y$ is the vertical measured foliation of the flat surface structure on $Y$ with the horizontal foliation $H_Y$. 
Let $\tt_{X, H_X}^r  (= \tt_{X, H_X})$ be the train-track decomposition of $E_{X, H_X}$ obtained in \S \ref{sTrainTracksForDiagonal}. 
\Cref{HomotopyToBeCarried} clearly implies the following.
\begin{proposition}\Label{EssentiallyCarriedAlmostDiagonal}
For a compact subset $N$ in $(\PML \times \PML) \minus \Delta^\ast$, if the train-track parameter $r > 0$ is sufficiently small, then 
 for all $(H_X, H_Y) \in N$, 
$\tt_{X, H_X}^r$ essentially carries $[V_Y]_X$. 
\end{proposition}

Let $W_Y$ be a realization of $[V_Y]_X$ on $\tt_{X, H_X}$ by a homotopy along horizontal leaf $H_X$ (\S\ref{sTrainTracksForDiagonal}).  
For every branch $B_X$ of $\tt_{X, H_X}$, consider the subset of $E_{Y, H_Y}$ which maps to $W_Y | B_X$ by the straightening map $\st\col E_{Y, H_Y} \to E_{X, H_X}$ (\S \ref{sStraighteningMap}) and the horizontal homotopy. 
Then, the boundary of the subset consists of straight segments in the vertical foliation  $V_Y$ and curves topologically transversal to $V_Y$ (Figure \ref{fStraightenBranches} for the case when $B_X$ is a rectangle).
\begin{figure}
\begin{overpic}[scale=.2
] {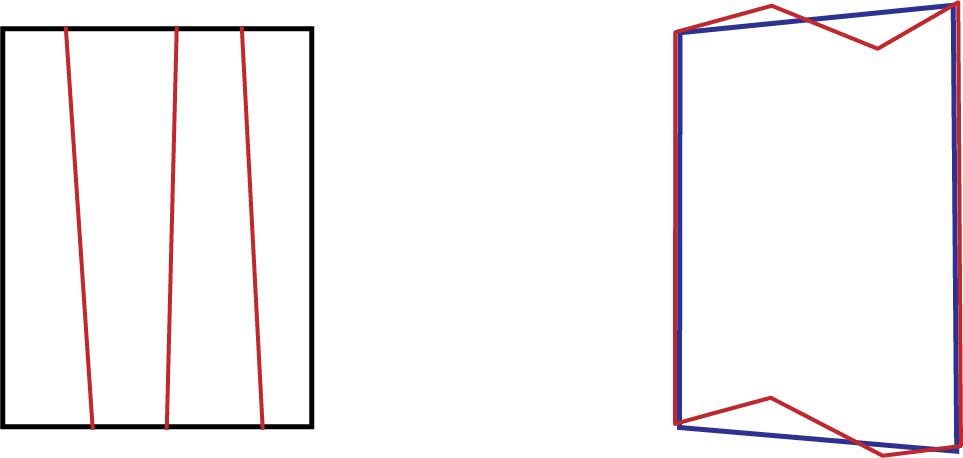} 
\put(-10 ,33 ){ \contour{white}{$B_X$}}  
 \put(27 ,21 ){ \contour{white}{\textcolor{Red}{{$W_Y | B_Y$}}}}  
  \put(72, 18){\contour{white}{\textcolor{Red}{\small $\st^{-1} (W_Y | B_Y)$}}}  
   \put(60 , 32){\textcolor{Blue}{$B_Y$}}  
      \end{overpic}
\caption{}\label{fStraightenBranches}
\end{figure}
We straighten each non-vertical boundary curve of the subset keeping its endpoints.
let $B_Y$ be the region in $E_{Y, H_Y}$ after straightening all non-vertical curves, so that the boundary of  $B_Y$  consists of segments parallel to $V_Y$ and segments transversal to $V_Y$. 
Then, for different branches $B_X$ of $\tt_{X, H_X}$, corresponding regions $B_Y$ have disjoint interiors; thus the regions $B_Y$ yield a trapezoidal surface train-track decomposition of $E_{Y, H_Y}$.
   
  Let $E$ be a flat surface, and let $H$ be its horizontal foliation. 
Then, for $\ep > 0$, a piecewise-smooth curve $c$ on  $E$ is {\sf $\ep$-almost horizontal}, if $\angle_E(H, c) < \ep$, i.e. the angles between the tangent vectors along $c$ and the foliation $H$ are less than $\ep$. 
More generally, $c$ is {\sf $\ep$-quasi horizontal} if $c$ is $\ep$-Hausdorff close to a geodesic segment which is $\ep$-almost horizontal to the horizontal foliation $H$. 
(In particular,  the length of $c$ is very short, then it is $\ep$-quasi horizontal.)
\begin{definition}
Let $E$ be a flat surface. 
For $\ep > 0$,  an {\sf $\ep$-quasi-staircase} train-track structure of $E$ is a trapezoidal train-track structure of $E$ such that its horizontal edges are all $\ep$-quasi horizontal straight segments. 
\end{definition}

If $H_X = H_Y$, then $\tt_{Y, H}$ is a staircase train-track, by continuity, we have the following.
\begin{lemma}
Let $r >0$ be a train-track parameter given by \Cref{EssentiallyCarriedAlmostDiagonal}.
Then, for every $\ep > 0$, if  the neighborhood $N$ of the diagonal in $\PML \times \PML$ is sufficiently small, then, for all $(H_X, H_Y) \in N$, the trapezoidal train-track decomposition $\tt_{Y, H_Y}^r$ of $E_{Y, H_Y}$ is $\ep$-quasi staircase.  
\end{lemma}

Next, similarly to \S \ref{sOnTheDiagonal}, we modify $\tt_{X, H_X}$ and $\tt_{Y, H_Y}$ by gluing some branches, so that corresponding branches have small diffeomorphic neighborhoods. 
Let $W_Y$ be a realization of $[V_Y]_X$ in $\tt_{X, H_X}$. 
Fix small $\del > 0$.
Let $v$ be a vertical edge $v$ of $\tt_{X, H_X}$, and let $B_1, B_2$ be the branches of $\tt_{X, H_X}$ whose boundary contains $v$. 
We glue $B_1$ and $B_2$ along $v$, if  the $W_Y$-measure of $v$ in $B_i$ is less than $\del$ for both $i = 1, 2$.
By applying such gluing for all vertical edges satisfying the condition, we obtain a staircase train-track $T_{X, H_X}^{r, \del} = T_{X, H_X}$, so that $\tt_{X, H_X}$ is a refinement of $T_{X, H_X}$. 

Then, $W_Y$ is still carried by $T_{X, H_X}$.
Therefore,  let $T_{Y, H_X}$ be the trapezoidal train-track decomposition of $E_{Y, H_Y}$ obtained by this realization, so that $\tt_{Y, H_Y}$ is its refinement.

Let $v$ be a vertical edge of $T_{X, H_X}$. 
Let $B_X$ be a branch of $T_{X, H_X}$ whose boundary contains $v$. 
Let $B_X'$ be the branch of $T_{X, H_X}$ adjacent to $B_X$ across $v$.
Let $B_Y$ and $B_Y'$ be the branches of $T_{Y, H_Y}$ corresponding to $B_X$ and $B_X'$, respectively. 
Then, there is a vertical edge  $w$ of $T_{Y, H_Y}$corresponding to $v$, contained in the boundary of both $B_Y$ and $B_Y'$.

If the $W_Y$-weight of $v$ in $B_X$ is less than $\del$, then the $W_Y$-weight of $V$ in $B_X'$ is at least $\del$, by the construction of $T_{X, H_X}$. 
Therefore, $B_Y'$ contains an $\ep$-quasi-staircase trapezoid $R_Y$, such that $w$ is a vertical edge of $R_Y$ and the horizontal length between the vertical edges is $\del/3$. 
(c.f. Figure \ref{fCutAndPasteRectangle}.)

Then, we can modify the train track $T_{Y, H_Y}$ by removing $R_Y$ from $B_Y'$ and gluing $R_Y$ with $B_Y$ along $w$ --- this modified $T_{Y, H_Y}$ by a homotopy. 
By simultaneously applying this modification for all vertical edges $v$ of $T_{X, H_X}$ satisfying the condition, we obtain a trapezoidal train-track decomposition $T_{Y, H_Y}'$. 
\
  
\begin{proposition}\Label{ModifiedTTForY}
For an arbitrary compact neighborhood $N$ of the diagonal $\Delta$ in $(\PML \times \PML) \minus \Delta^\ast$, fix
a sufficiently small train-track parameter $r > 0$ obtained by \Cref{EssentiallyCarriedAlmostDiagonal}.
 Then, if the parameter $\del > 0$ is sufficiently small, then
 for every $(H_X, H_Y) \in N$, \begin{itemize}
\item $T_{X, H_X}$ is semi-diffeomorphic to $T_{Y, H_Y}'$; 
\item $T'_{Y, H_Y}$ is $\del$-Hausdorff close to $T_{Y, H_Y}$ in the normalized metric $E^1_{Y, H_Y}$;
\item  the open $\del/4$-neighborhood of the singular set is disjoint from the one-skeleton of $T'_{Y, H_Y}$.
\end{itemize}
\end{proposition} 

\begin{figure}
\centering
\begin{minipage}{.5\textwidth}
  \centering
\begin{tikzcd}
T_{X, H_X}   \arrow[d, hook, "subdivide"] \arrow{r}{\scriptsize\parbox{1.2cm}{semi-\\ diffeo}}&  T'_{Y, H_Y}   \arrow[r, "\frac{\delta}{3}-close",leftrightsquigarrow] & T_{Y, H_Y}   \arrow[dl, hook, "subdivide"]  \\
\tt_{X, H_X}   \arrow[r, rightsquigarrow, "W_Y"]     &  \tt_{Y, H_Y} 
 \end{tikzcd}
 \caption{Relations between constructed train tracks. }\label{}
\end{minipage}%
\begin{minipage}{.5\textwidth}
  \centering
\begin{overpic}[scale=.15
] {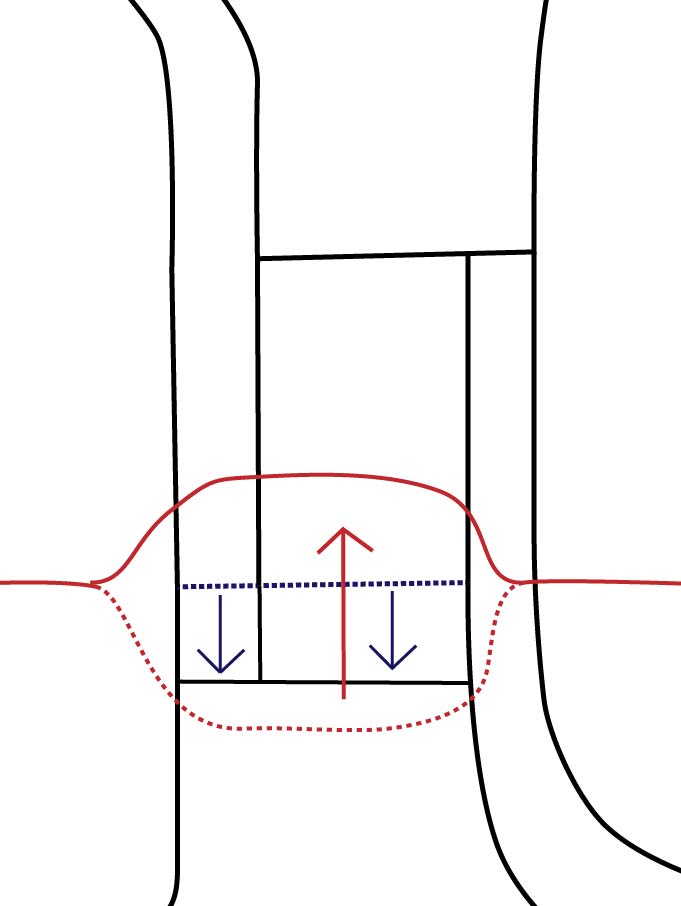} 
      \end{overpic}
\caption{A slide corresponding a shift in \Cref{fShifting}. }\label{fShiftingAndSliding}

\end{minipage}
\end{figure}


{\sf A sliding} is an operation of a train-track moving some vertical edges in the horizontal direction without changing the homotopy type of the train-track structure.
If we change the realization $W_Y$ on $\tt_{X, H}$ by shifting across a vertical slit, the induced train-track $\tt_{Y, H}$ changes by sliding its corresponding vertical segment (Figure \ref{fShiftingAndSliding}).

As before, 
a branch of $T_{X, H_X}$ disjoint from the non-transversal graph $G_Y$ is called a {\sf transversal branch.} 
A branch of $T_{X, H_X}$ containing a component of  $G_Y$ is called the {\sf non-transversal branch}.
   By the semi-continuity of $T_{Y, H_Y}$ in \Cref{TforY} and the construction of $T'_{Y, H_Y}$, we obtain a semi-continuity of $T_{Y, H_Y}$ up to sliding. 
    \begin{lemma}
Let $\hH_i = (H_{X, i}, H_{Y, i})$ be a sequence converging to $\hH = (H_X, H_Y)$.
Then, up to a subsequence, $T_{Y, H_{Y, i}}$ semi-converges to a train track structure  $T_{Y, H_Y}''$ of $E_{Y, H_Y}$,  such that either $T_{Y, H_Y}'' = T_{Y, H_Y}'$  or $T_{Y, H_Y}''$ can be transformed to a refinement of $T_{Y, H_Y}'$ by sliding some vertical edges by $\del /3$. 
\end{lemma}

\subsection{Bounded polygonal train tracks  for  the Riemann surface $X$}\Label{sBoundedEuclideanTraintrackForX}

The train tracks we constructed may so far have rectangular branches with very long horizontal edges. 
In this section, we further modify the train-track structures $T_{X, H_X}$ and $T_{Y, H_Y}$ from \S \ref{sCompatibleEuclideanTraintracks} by reshaping those long rectangles into spiral cylinders. 

Given a rectangular branch of a train track,  although its interior is embedded in a flat surface, its boundary may intersect itself.
  Let $T$ be a train-track structure of a flat surface $E$. 
The {\sf diameter} of a branch $B$ of $T$ is the diameter of the interior of $B$ with the path metric in $B$. 
The {\sf diameter} of a train track $T$ is the maximum of the diameters of the branches of $T$.

Recall that we have fixed a compact neighborhood $N \sub (\PML \times \PML) \minus \Delta^\ast$ of the diagonal.
Recall that, for $(H_X, H_Y) \in N$,  $E_{X, H_X}^1$ and $E_{Y, H_Y}^1$ are the unit-area flat structures realizing $(X, H_X)$ and $(Y, H_Y)$, respectively.
 Pick a small $r > 0$ given by Proposition \ref{EuclideanTraintarck}, so that, for every $(H_X, H_Y) \in N$, there are train-track structures $T_{X, H_X}$ of $E_{X, H_X}^1$ and $T_{Y, H_Y}$ of $E_{Y, H_Y}^1$ from \S \ref{sSeimidiffeoTraintrackNearDiagonal}.
\begin{lemma}\Label{LimitOfLongBranches}
\begin{enumerate}
\item \Label{iRectangleConveringToFlatyCylinder}
Let $\hH_i = (H_{X,  i},H_{Y, i}) \in N$ be a sequence converging to $\hH = (H_X,  H_Y) \in N$.
Suppose that $T_{X, H_{X, i}} \eqqcolon T_{X, i}$ contains a rectangular branch $R_i$ for every $i$, such that the horizontal length of $R_i$ diverges to infinity as $i \to \infi$. 
Then, up to a subsequence, the support $|R_i| \sub E_{X, H_{X, i}}^1 \eqqcolon E_{X, i}$ converges to either  
\begin{itemize}
\item a flat cylinder which is a branch of  $T_{X, H_X}$  or 
\item a closed leaf of $H_X$ which is contained in the union of the horizontal edges of $T_{X, H_X}$.
\end{itemize}

\item \Label{iSpiralCylinderCloseToFlatCylinder}
Let $A$ be the limit flat cylinder or a loop in (\ref{iRectangleConveringToFlatyCylinder}).
For sufficiently large $i > 0$, let $R_{i, 1}, \dots, R_{i, n_i}$ be the set of all rectangular branches of $T_{X, i}$ which converge to $A$ as $i \to \infi$ in the Hausdorff metric.
Then, the union  $R_{i, 1} \cup \dots \cup R_{i, n_i} \sub E_{X, H_i}^1$ is a spiral cylinder for all sufficiently large $i$. (See \Cref{fSpiralCylinder}.) 
\end{enumerate}

\end{lemma}

\proof
(\ref{iRectangleConveringToFlatyCylinder})
Let $R_i$ be a rectangular branch of $T_{X, i}$ such that the horizontal length of $R_i$ diverges to infinity as $i \to \infi$. 
Then, as  $\Area\, E_i = 1$, the vertical length of $R_i$ must limit to zero.
Then, in the universal cover $\ti{E}_i$ of $E_i$, we can pick a lift $\ti{R}_i$ of $R_i$ which converges, uniformly on compact,  to a smooth horizontal leaf of $\ti{H}_X$ or a copy of $\R$ contained in a singular leaf of $\ti{H}_X$. 
Let $\ti\ell$ denote the limit, and let $\ell$ be its projection into a leaf of $H_X$.
\begin{claim}\Label{ClosedLeaf}
$\ell$ is a closed leaf of $H_X$.
\end{claim}
\begin{proof}
Suppose, to the contrary, that $\ell$ is not periodic. 
Then $\ell$ is either a leaf of an irrational sublamination or a line embedded in a singular leaf of $H_{X, H_X}$.
Then, the distance from $\ell$ to the singular set of $E_{X, H_X}$ is zero. 

Recall that  the $(r, \sqrt[4]{r})$-neighborhood of the singular set of $E_{X, H_{X, i}}$ is contained in the (non-rectangular) branches of $T_{X, H_{X, i}}$.
Thus, the distance from $R_i$ to the singular set of $E_{X, H_{X,i}}$ is at least $r > 0$ for all $i$. 
This yields a contradiction.  
\end{proof}

By Claim \ref{ClosedLeaf},  as a subset of $E_i$,
the rectangular branch $R_i$ converges to the union of closed leaves $\{\ell_j\}_{j \in J}$ of $H_{X, L_X}$.
Thus the Hausdorff limit $A$ of $R_i$ in $E_{X, H}$ must be a connected subset foliated by closed horizontal leaves.
Therefore, $A$ is either a flat cylinder or a single closed leaf. 

First, suppose that the limit $A$ is a flat cylinder.
Then, the vertical edges of $R_i$ are contained in the vertical edges of non-rectangular branches. 
The limit of the vertical edges of $R_i$ are points on the different boundary components of $A$.
Therefore, each boundary component of $A$ must intersect a non-rectangular branch in its horizontal edge. 
Therefore, the cylinder is a branch of $T_{X, H}$ by the construction of $\tt_{X, H}$.  

If the limit $A$ is a single leaf, similarly, one can show that the vertical one-skeleton of $T_{X, H_X}$, since a loop can be regarded as a degeneration of a flat cylinder.

(\ref{iSpiralCylinderCloseToFlatCylinder})
First assume that the limit $A$ is a flat cylinder. 
Since the $(r, \sqrt[4]{r})$-neighborhood of the singular set is disjoint from the interior of $A$,  we can enlarge $A$  to a maximal flat cylinder $\hat{A}$ in $E_{X, H}$ whose interior contains (the closure of) $A$.
Then, each boundary component of $\hat{A}$ contains at least one singular point. 
Since $A$ is a cylindrical branch, each boundary component $\ell$ of $A$ contains a horizontal edge of a non-rectangular branch $P_\ell$ of $T_{X, H_X}$ which contains a singular point in the boundary of $\hat{A}$.

Let $P_{\ell, 1}, \dots, P_{\ell, n}$ be the non-rectangular branches of $T_{X, H_X}$ whose boundary intersects $\ell$. 
Recall that  $P_{\ell, i}$ is the union of some branches of $\tt_{X, H}$. 
Although $P_{\ell_1, i}$ itself may not be convex,  a small neighborhood of the intersection $P_{\ell, i} \cap \ell$ in $P_{\ell, i}$ is convex. 
Let $P_{i, 1}, \dots, P_{i, k_i}$ be all non-rectangular branches of $T_{X, i}$, such that their union $P_{i, 1} \cup \dots \cup P_{i, k_i}$ converges to the union of all non-rectangular branches of $T_{X, H_X}$ which have horizontal edges contained in the boundary of $A$.
Then, for sufficiently, large $i$, the vertical edges of $R_{i,1}. \dots, R_{i, n_i}$ are contained in  vertical edges of polygonal branches $P_{i, 1}, \dots, P_{i, k_i}$. 
  Then, by the convexity above,  if the union of  $R_{i,1}, \dots, R_{i, n_i}$ intersects $P_{i, j}$, then its intersection is a monotone staircase curve. 
Therefore the union of  $R_{i,1}. \dots, R_{i, n_i}$ is a spiral cylinder.
See Figure \ref{fSprialCylinderCloseToFlatCylinder}.
\begin{figure}[H]
\begin{overpic}[scale=.06
] {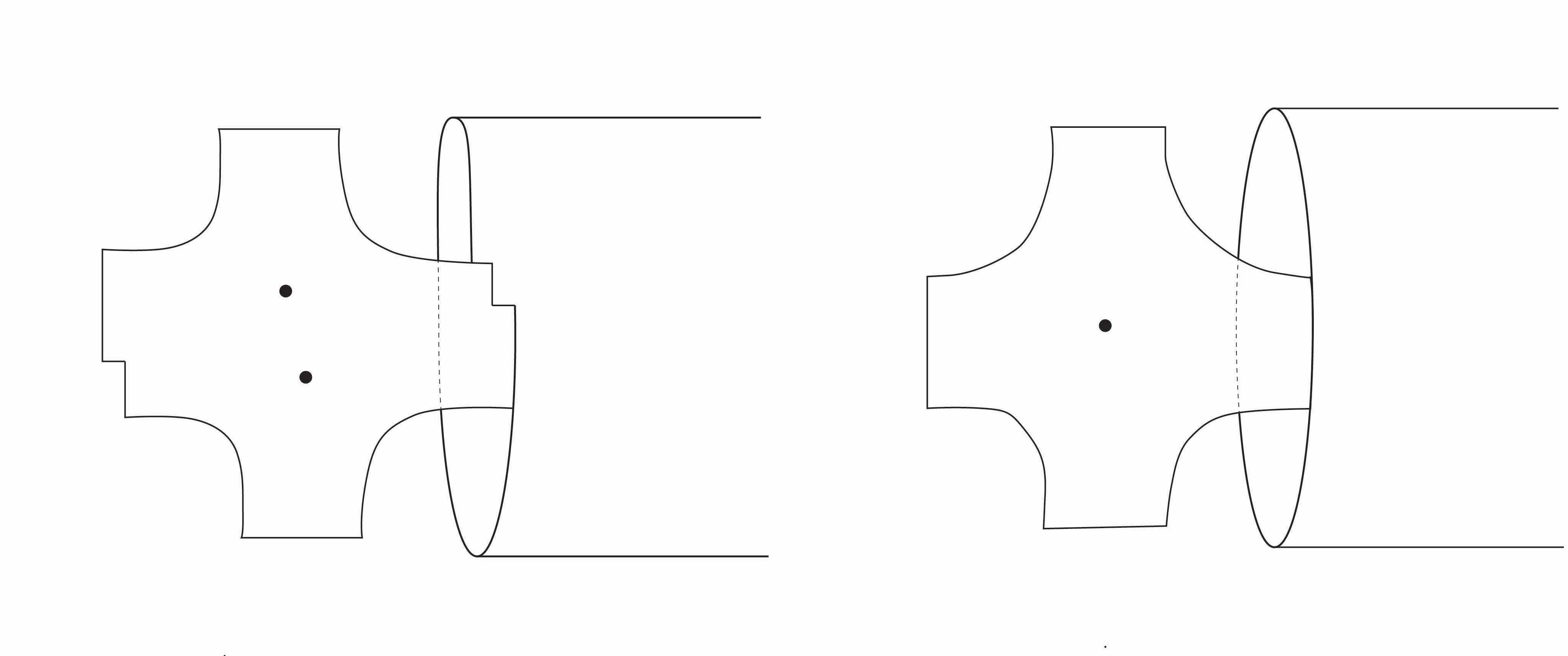} 
  \put(68 , 14 ){$P_\ell$}  
  \put(90 , 18 ){$A$}  
    \put(34 , 18 ){$\cup_{j = 1}^{n_i} R_{i, j}$}  
      \end{overpic}
\caption{}\label{fSprialCylinderCloseToFlatCylinder}
\end{figure}
A similar argument holds in the case when the limit is a closed loop in a singular leaf. 
 \Qed{LimitOfLongBranches}
By Lemma  \ref{LimitOfLongBranches} (\ref{iRectangleConveringToFlatyCylinder}), (\ref{iSpiralCylinderCloseToFlatCylinder}), there is a constant $c > 0$, such that, for $H_X \in N$,
if a rectangular branch $R$ of $T_{X, H_X}$ has a horizontal edge of length more than $c$, then $R$ is contained in a unique spiral cylinder, which may contain other rectangular branches. 
The diameter of such spiral cylinders is uniformly bounded from above by a constant depending only on $X$. 
Thus, replace all rectangular branches $R$ of $T_{X, H_X}$ with corresponding spiral cylinders, and we obtain a staircase train track $\tT_{X, H_X}$:

\begin{corollary}\Label{UniformlyBoundedTrainttrackX}
There is $c > 0$, such that, for all $\hH = (H_X, H_Y) \in N$, 
the diameters of the branches of the staircase train track $\tT_{X, H_X}$ are bounded by $c$. 
\end{corollary} 
 
\subsection{Semi-diffeomorphic bounded almost polygonal train-track structures for $Y$ }\Label{sCompatibleBoundedPolygonalTraintrack}

For  $\ep > 0$,  we have constructed,  for  all $\hH = (H_X, H_Y)$ in  some compact neighborhood $N$ of the diagonal  $\Lambda_\infi$ in $(\PML \times \PML) \minus \Delta^\ast$, 
a staircase train-track structure $T_{X, H_X}$ of $E_{X, H_X}$ and an $\ep$-quasi-staircase train-track structure $T_{Y, H_Y}$ of $E_{Y, H}$ , such that $T_{X, H_X}$ is semi-diffeomorphic to $T_{Y, H_Y}$.  
In \S \ref{sBoundedEuclideanTraintrackForX}, we modify $T_{X, H_X}$ and obtain a uniformly bounded train-track $\tT_{X, H_X}$ creating spiral cylinders. 
In this section, we accordingly modify $T_{Y,  H_Y}$ to a bounded $\ep$-quasi-staircase train-track structure. 

\begin{lemma}\Label{CorrespondingSpiralCyliner}
\begin{enumerate}
\item  \Label{iCorrespondingSpiralCylindersSemiDiff}
For every spiral cylinder $A$ of $\tT_{X, H_X}$, letting  $R_{X, 1}, R_{X, 2}, \dots, R_{X, n}$ be the rectangular branches of $T_{X, \hH}$ whose union is $A$, 
 there are  corresponding branches $R_{Y, 1}, R_{Y, 2}, \linebreak[3] \dots, R_{Y, n}$ of $T_{Y, \hH}$, such that
 \begin{itemize}
 \item their union  $R_{Y, 1} \cup R_{Y, 2} \cup \dots \cup R_{Y, n}$ is  a spiral cylinder in $E_{Y, H_Y}$, and
\item  $A$ is semi-diffeomorphic to  $R_{Y, 1} \cup R_{Y, 2} \cup \dots \cup R_{Y, n}$. 
\end{itemize}
\item
 Moreover, there is a constant $c' > 0$, such that, if a rectangular branch of $T_{Y, H_Y}$ has horizontal length more than $c'$, then it is contained in a spiral cylinder as above. \Label{iCorrespondingRectanglesAreQI}
\end{enumerate}
\end{lemma}
\begin{proof}
As $(H_X, H_Y) \cap \Delta^\ast = \emptyset$, the geodesic representative $[V_Y]_X$ essentially intersects $A$.
Thus, the realization $W_Y$ has positive weights on $R_{X, 1}, R_{X, 2}, \dots, R_{X, n}$.
Thus $R_{X, j}$ corresponds to a rectangular branch $R_{Y, j}$ of $T_{Y, H_Y}$, and their union  $\cup_j R_{Y, j}$ is a spiral cylinder in $T_{Y, H_Y}$ ((\ref{iCorrespondingSpiralCylindersSemiDiff})).

Let $R_X$ and $R_Y$ be corresponding rectangular branches of $T_{X, H_X}$ and $T_{Y, H_Y}$, respectively. 
As $(H_X, H_Y)$ varies only in a fixed compact neighborhood of the diagonal,
the horizontal length of $R_X$ is bilipschitz close to the length of the horizontal length of $R_Y$ with a uniform bilipschitz constant for such all $R_X$ and $R_Y$.  
 Therefore, there is $c' > 0$ such that if $R_Y$ is more than $c'$, then the corresponding branch $R_X$ has length more than the constant $c$ (right before \Cref{UniformlyBoundedTrainttrackX}), then $R_X$ is contained in a unique spiral cylinder (\ref{iCorrespondingRectanglesAreQI}).
\end{proof}

For every spiral cylinder $A$ of $\tT_{X, H_X}$, by applying  Lemma \ref{CorrespondingSpiralCyliner}, we replace the branches  $R_{Y, 1}, R_{Y, 2}, \dots, R_{Y, n}$ of $T_{Y, H_Y}$ with the spiral cylinder $R_{X, 1} \cup R_{X, 2} \cup \dots \cup R_{X, n}$ of $T_{Y, H_Y}$.
Then, 
we obtain an $\ep$-quasi-staircase train-track decomposition $\tT_{Y, H_Y}$ without long rectangles:
\begin{proposition}\Label{PolygonalTaintrackY}
For every $\ep > 0$, there are $c > 0$ and a neighborhood $N$ of the diagonal in $(\PML \times \PML) \minus \Delta^\ast$, such that,
for every $\hH = (H_X, H_Y) \in N \sub \PML \times \PML$,
there is an $\ep$-quasi-staircase train-track decomposition $\tT_{Y, H_Y}$ of $E_{Y, H_Y}$, such that 
\begin{enumerate}
\item $\tT_{Y, H_Y}'$ is $\del$-Hausdorff close to  $\tT_{Y, H_Y}$ in $E_{Y, H}^1$; \Label{iHausdorffCloseTY}
\item the diameters of $\tT_{Y, H_Y}$ and $\tT_{Y, H_Y}'$ are less than $c$;  \Label{iDiameterTY}
\item $\tT_{X, H_X}$ is semi-diffeomorphic with $\tT_{Y, H_Y}'$; \Label{iSemiDiffeomoprhictT}
\item  $\tT_{Y, H_Y}$ changes semi-continuously in $(H_X, H_Y)$ and the realization of $[V_Y]_X$ on $\tT_{X, H_X}$. \Label{iSemiContinuoustT} 
\end{enumerate}
\end{proposition}
\begin{proof}
Assertion  (\ref{iDiameterTY}) follows from \Cref{CorrespondingSpiralCyliner} (\ref{iCorrespondingRectanglesAreQI}).
Assertion (\ref{iHausdorffCloseTY}) follows from \Cref{ModifiedTTForY}.
Assertion (\ref{iSemiDiffeomoprhictT}) follows from  \Cref{ModifiedTTForY} and \Cref{CorrespondingSpiralCyliner} (\ref{iCorrespondingSpiralCylindersSemiDiff}).
Assertion (\ref{iSemiContinuoustT}) holds, by \Cref{TforY}, since $T_{Y, H_Y}$ changes semi-continuously in $(H_X, H_Y)$ and the realization $W_Y$ of $[V_Y]_X$ on $T_{X, H_X}$.
\end{proof}
   \section{Thurston laminations and vertical foliations}\Label{sLaminationAndFoliations}
  
 \subsubsection{Model Euclidean Polygons and projective circular polygons}\Label{sModelCircularPolygons}
A {\sf polygon with circular boundary} is a projective structure on a polygon such that the development of each edge is contained in a round circle in $\CP^1$.
 Let $\sigma$ be an ideal hyperbolic $n$-gon  $(n \geq 3)$ . 
 Let $L$ be a measured lamination on $\sigma$ except that each boundary geodesic of $\sigma$ is a leaf of weight $\infty$.
 From a view point of the Thurston parameterization, it is natural to add such weight-infinity leaves.
 In fact, there is a unique $\CP^1$-structure $\CCC = \CCC(\sigma, L)$ on the complex plane $\C$ whose Thurston's parametrization is the pair $(\sigma, L)$; see  \cite{GuptaMj21}. 
  Let $\LLL$ be the Thurston lamination on $\CCC$.
Denote, by $\kap\col \CCC \to \sigma$, the collapsing map (\S \ref{sCollapsing}). 
  
For each boundary edge $l$ of $\sigma$, pick a leaf $\ell$ of $\LLL$ which is sent  diffeomorphically onto $l$ by $\kappa$. 
Then,  those circular leaves bound a circular projective $n$-gon $\PPP$ in $\CCC$, called an {\sf ideal projective polygon}.

For each $i = 1, 2, \dots, n$, 
let $v_i$ be an ideal vertex of $\sigma$, and let $l_i$ and $l_{i + 1}$ be the edges of $\sigma$ starting from $v_i$. 
Consider the geodesic $g$ starting from $v_i$ in the middle of $l_i$ and $l_{i+1}$, so that the reflection about $g$ exchanges $l_i$ and $l_{i + 1}$.
Embed $\sigma$ isometrically into a totally geodesic plane in $\H^3$. 
Accordingly  $\PPP$ is embedding in $\CP^1$ so that the restriction of $\kap$ to $\PPP$ is the nearest point projection to $\sigma$ in $\H^3$.

Then, pick a round circle $c_i$ on $\CP^1$ such that the hyperbolic plane, ${\rm Conv}\, c_i$, bounded by $c_i$ is orthogonal to $g$, so that $l_i$ and $l_{i + 1}$ are transversal to  ${\rm Conv}\, c_i$. 

Let $\ell_i$ and $\ell_i$ be the edges of $\PPP$ corresponding to $l_i$ and $l_{i+1}$, respectively.  
If $c_i$ is close to $v_i$ enough, then there is a unique arc $a_i$ in $\PPP$ connecting $\ell_i$ to $\ell_{i + 1}$ which is immersed into $c_i$ by the developing map. 
Then, the region in $\PPP$ bounded by $a_1 \dots a_n$ is called the {\sf truncated ideal projective polygon}. 

 \begin{definition}
Let $C$ be a $\CP^1$-structure on $S$. 
Let $E^1$ be the normalized flat surface of the Schwarzian parametrization of $C$. 
Let $P$ be a staircase polygon in $E^1$. 
Then  $P$ is {\sf $\ep$-close} to 
a truncated ideal projective polygon $\PPP$, if $\PPP$ isomorphically embeds onto a polygon in $C$ which is $\ep$-Hausdorff close to $P$ in the normalized Euclidean metric. 
\end{definition}

For $X \in \TT \sqcup \TT^\ast$, recall that $\rchi_X$ be the holonomy variety of the $\CP^1$-structures on $X$.
For $\rho \in \rchi_X$, let $C_{X, \rho}$ be the $\CP^1$-structure on $X$ with holonomy $\rho$, and let $E_{X, \rho}$ be the flat surface given by the holomorphic quadratic differential of $C_{X, \rho}$. 
Similarly, for $\ep > 0$,  let $N_\ep^1 Z_{X, \rho}$ be the $\ep$-neighborhood of the singular set in the normalized flat surface $E^1_{X, \rho}$.
Let $\LLL_{X, \rho}$ be the Thurston lamination of $C_{X, \rho}$.

Then, by  combining what we have proved,  we obtain the following.
\begin{theorem}\Label{EuclideanAndProjectivePolygons}
Let $X \in \TT \sqcup \TT^\ast$.
Then, for every $\ep > 0$, there is a bounded subset $K = K(X, \ep)$ of $\rchi_X$ satisfying the following:
 Suppose that $\rho$ is in $\rchi_X \minus K$, and that the flat surface  $E_{X, \rho}$ contains
 a staircase polygon $P$ such that
\begin{itemize}
\item $\bdr P$ disjoint from $N^1_\ep Z_{X, \rho}$ and
\item  the diameter of $P$ is less than $\frac{1}{\ep}$. 
\end{itemize}
Then
\begin{enumerate}
\item \Label{iThurstonLaminationAndVerticalFoliation}$\LLL_{X, \rho} | P$ is $(1 + \ep, \ep)$-quasi-isometric to $V_{X, \rho} | P$ up to an isotopy supported on  $N_\ep^1 Z_{X, \rho} \cap P$, such that, in the normalized Euclidean metric $E^1_{X, \rho}$,  
\begin{enumerate}
\item  on $P$, each leaf of $V$ is $\ep$-Hausdorff-close to a leaf of $\LLL$, and \Label{iLeavesAreClose}
\item the transversal measure of $V$ is $\ep$-close to the transversal measure of $\LLL$ for all transversal arcs whose lengths are less than one.  \Label{iTransversalMeasureForShortArcs}
\end{enumerate}
\item  In the  (unnormalized) Euclidean metric, $P$ is $\ep$-close to a truncated circular polygon of the hyperbolic surface in the Thurston parameters. \Label{CloseToModelCircularPolygon}
\end{enumerate}
\end{theorem}
\proof
The assertion 
(\ref{iLeavesAreClose}) follows from Lemma \ref{VerticalLeafInFoliationAndLamination}.
The assertion (\ref{iTransversalMeasureForShortArcs}) is given by Proposition \Cref{ThurstonLaminationAndVerticalFoliationOnPolygons}.

We shall prove (\ref{CloseToModelCircularPolygon}).
Set $C_{X, \rho} = (\tau, L) \in \TT \times \ML$ be the $\CP^1$ structure on $X$ with holonomy $\rho \in \rchi \minus K$ in Thurston coordinates, and let $\kap\col C_{X, \rho} \to \tau$ be the collapsing map.
Since sufficiently away from the zero, the developing map is well-approximated by the exponential map (\Cref{ApproximationByExponentialMap}). 

If $K$ is sufficiently large, then
for every vertical edge $v$ of $P$, the restriction $\Ep_{X, \rho} |  v$  is
 a $(1+ \ep)$-bilipschitz embedding on the Epstein surface.
Therefore,  by the closeness of $\Ep_{X, \rho}$ and $\hat\beta_{X, \rho}$, $\kap (v)$ is $\ep$-close to a geodesic segment $s_v$ of length $\sqrt{2} \length v$.
By (\ref{iTransversalMeasureForShortArcs}), if $K$ is large enough,  $L (s_v) < \ep$.  

Every horizontal edge $h$ of $P$ is very short on the Epstein surface (\Cref{Dumas}). 
As the developing map is approximated by the Exponential map and $\operatorname{Area} \tau = 2\pi \rvert \rchi(S) \lvert$,  
 it follows that, if $\kap(h)$ has length less than $\ep$ on $\tau$.
Therefore, the image of $P$ on the hyperbolic surface is $\ep$-close to a truncated ideal polygon. 

\begin{figure}
\begin{overpic}[scale=.06
] {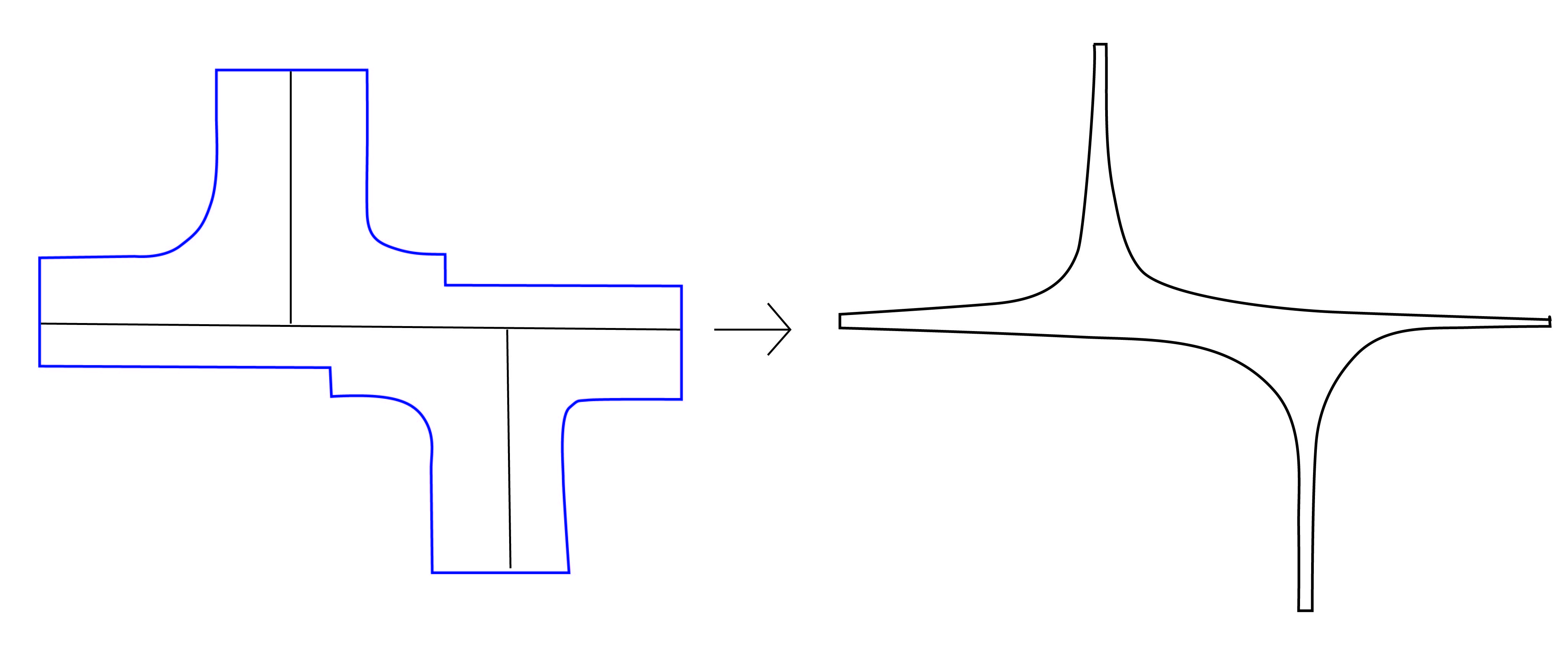} 
      \end{overpic}
\caption{}\label{fEuclideanPolygonToHyperbolicPolygon}
\end{figure}

\Qed{EuclideanAndProjectivePolygons}

\subsection{Equivariant  circle systems}
For $\rho \in \rchi_X$, we shall pick a system of a $\rho$-equivalent round circles on $\CP^1$, which will be used to construct a circular train-track structure of $C_{X, \rho}$. 
Let $\ti\tT_{X, \rho}$ be the $\pi_1(S)$-invariant train-track structure on $\ti{E}_{X, \rho}$ obtained by lifting the train-track structure $\tT_{X, \rho}$ on $E_{X, \rho}$.
Let  $\Ep^\ast_{X, \rho}\col T \ti{E}_{X, \rho} \to T \H^3$ be the differential of $\Ep_{X, \rho}\col \ti{E}_{X, \rho} \to \H^3$.
\begin{lemma}\Label{EquivariantCircles}
For every $\ep > 0$, there is a bounded subset  $K_\ep$ of $\rchi$ such that, if $\rho\col \pi_1(S) \to \PSL(2, \C)$  belongs to  $\rchi_X \minus K_\ep$,  then, we can assign a round circle $c_h$ to every minimal horizontal edge $h$ of $\ti{\tT}_{X, \rho}$ with the following properties:  

\begin{enumerate}
\item The assignment $h \mapsto c_h$ is $\rho$-equivariant. 
\item The hyperbolic plane bounded by $c_h$ is $\ep$-almost orthogonal to the $\Ep^\ast_{X, \rho}$-images of  the vertical tangent vectors along $h$. \Label{iAlmostOrthogonal}
\item If $h_1, h_2$ are horizontal edges of $\ti\tT_{X, \rho}$ connected by a vertical edge $v$ of length at least $\ep$, then the round circles $c_{h_1}$ and $c_{h_2}$ are disjoint.   \Label{iDisjointCircles}
\item If $h_1, h_2, h_3$ are ``vertically consecutive" horizontal edges, such that
\begin{itemize}
\item $h_1$ and $h_2$ are connected by a vertical edge $v_1$ of $E_{X, \rho}$-length at least $\ep$; 
\item $h_2$ and $h_3$ are connected by a vertical edge $v_3$ of length at least $\ep$;
\item $h_1$ and $h_3$ are on the different sides of $h_2$, i.e. the normal vectors of $h_2$ in the direction of $v_1$ and $v_3$ are opposite,
\end{itemize}
  then $c_{h_1}$ and $c_{h_3}$ are disjoint, and they bound a round cylinder whose interior contains $c_{h_2}$. \Label{iPreservingOrder}
\end{enumerate}
\end{lemma}
\begin{proof}
Without loss of generality, we can assume that $\ep > 0$ is sufficiently small. 
With respect to the normalized Euclidean metric $E_{X, \rho}^1$, the lengths of minimal horizontal edges of $T_{X, \rho}$ are uniformly bounded from above by Corollary \ref{UniformlyBoundedTrainttrackX}, and the distances of the horizontal edges from the singular set of $E_{X, \rho}^1$ are uniformly bounded from below. 
 Then, by Lemma \ref{Dumas}, for every $\ep > 0$, if a bounded subset $K_\ep$ in $\chi$ is sufficiently large, then for every minimal horizontal edge $h$ of $T_{X, \rho}$, the vertical tangent vectors along $h$ on $E_{X, \rho}$ of unit length map to  $\ep$-close tangent vectors of  $\H^3$.

 Therefore,  if $K_\ep$ is large enough,  for each minimal horizontal edge $h$ of $\ti\tT_{Y, \rho}$, we pick a round circle $c_h$, such that 
 the assignment of $c_h$ is holonomy equivariant and that the images of vertical tangent vectors along $h$ are $\ep^2$-orthogonal to the hyperbolic plane bounded by $c_h$.

Then, if $v$ is a vertical edge sharing an endpoint with $h$, then  $\Ep_{X, \rho} (v)$ is $\ep^2$-almost orthogonal to the hyperbolic plane bounded by $c_h$. 
For every sufficiently small $\ep > 0$, if $K > 0$ is sufficiently large,  then the geodesic segment of length, at least, $\ep$ connects the hyperbolic planes bounded by $c_{h_1}$ and $c_{h_2}$, and the geodesic segment is $\ep^2$-almost orthogonal to both hyperbolic planes. 
Therefore, if $\ep > 0$ is sufficiently small, then, by elementary hyperbolic geometry,  the hyperbolic planes are disjoint, and (3) holds. 
By a similar argument, (\ref{iPreservingOrder}) also holds.
\end{proof}

The circle system in Lemma \ref{EquivariantCircles} is not unique, but unique up to  an appropriate isotopy:
\begin{proposition}\Label{IsotopeCircleSystem}
For every $\ep_1 > 0$, there is $\ep_2 > 0$, such that, for every $\rho \in \rchi_X \minus K_{\ep_2}$
given two systems of round circles $\{c_h\}$ and $\{c_h'\}$ realizing Lemma \ref{EquivariantCircles}  for $\ep_2 > 0$, there is a one-parameter family of equivalent circles systems $\{c_{t, h}\} ~(t \in [0, 1])$  realizing Lemma \ref{EquivariantCircles}  for $\ep_1 > 0$ which continuously connects $\{c_h\}$ to $\{c_h'\}$. 
\end{proposition}
\begin{proof}
The proof is left for the reader.
\end{proof}
\subsection{Pleated surfaces are close}
The following gives a measure-theoretic notion of almost parallel measured laminations.
\begin{definition}[Quasi-parallel]\Label{almostParallel}
Let $L_1, L_2$ be two measured geodesic laminations on a hyperbolic surface $\tau$. 
Then, $L_1$ and $L_2$ are  $\ep$-quasi parallel, if  a leaf $\ell_1$ of $L_1$ and a leaf $\ell_2$ of $L_2$ intersect at a point $p$ and $\angle_p (\ell_1, \ell_2) > \ep$, then 
letting $s_1$ and $s_2$ be the unit length segments in $\ell_1$ and $\ell_2$ centered at $p$, 
$$\min( L_1(s_2), L_2(s_1) ) < \ep.$$
\end{definition}

\begin{proposition}\Label{AlmostParalellLaiminationAndFoliation}
For every $\ep  > 0$, if a bounded subset $K  \sub \rchi_X \cap \rchi_Y$ is sufficiently large, then 
$L_{Y, \rho}$ is $\ep$-quasi-parallel to $L_{X, \rho}$ on $\tau_{X, \rho}$ away from the non-transversal graph $G_Y$.
\end{proposition}

\begin{proof}
If $K$ is sufficiently large, 
$\angle_{E_{X, \rho}}(H_{X, \rho}, V_{Y, \rho}')$ is uniformly bounded from below by a positive number by Lemma \ref{Semitransversailty}. 
Then, the assertion follows from Lemma \ref{Dumas} and 
Theorem 
\ref{EuclideanAndProjectivePolygons}.
\end{proof}
In this section, we show that the pleated surfaces for $C_{X, \rho}$ and $C_{Y, \rho}$ are close away from the non-transversal graph.
Recall that $\hat\beta_{X, \rho}\col \ti{C}_{X, \rho} \to \H^3$  denotes the composition of the collapsing map and the bending map for $C_{X, \rho}$, and similarly $\hat\beta_{Y, \rho}\col \ti{C}_{Y, \rho} \to \H^3$  denotes the composition of the collapsing map and the bending map for $C_{Y, \rho}$.
\begin{theorem}\Label{PleatedSurfacesAreClose}
Let $X, Y \in \TT \sqcup \TT^\ast$ with $X \neq Y$. 
For every $\ep > 0$, there is a bounded subset $K_\ep$ in $\rchi_Y \cap \rchi_X$ such that, for every $\rho \in \rchi_X \cap \rchi_Y \minus K_\ep$, 
there are a homotopy equivalence map $\phi\col E_{Y, \rho} \to E_{Y, \rho}$ and a semi-diffeomorphism $\psi \col \tT_{X, \rho} \to \tT_{Y, \rho}'$ given by \Cref{PolygonalTaintrackY} (\ref{iSemiDiffeomoprhictT}) satisfying the following:
\begin{enumerate}
\item $d_{E^1_{Y, \rho}}(\phi(z), z) < \ep$;
\item  the restriction of $\phi$ to  $E_{Y, \rho} \minus N^1_\ep Z_{Y, \rho}$ can be transformed to the identity by a homotopy along vertical leaves of $E_{X, \rho}$;
\item 
$\hat\beta_{X, \rho}(z)$ is $\ep$-close to  $\hat\beta_{Y, \rho} \circ \ti\phi \circ \ti{\psi}(z)$ in $\H^3$ for every point $z \in \ti{E}_{X, \rho}$ which are not in the interior of the non-transversal branches of $\tT_{X, \rho}$.\Label{iPleatedSurfacesPartiallyClose}
\end{enumerate}

\end{theorem}

Using \Cref{Dumas}, one can prove the following. 
\begin{lemma}\Label{StaircaseCurveAndQuasiGeodesic}
Let $\ep > 0$ and let $X \in \TT \sqcup \TT^\ast$. 
Then, there is a compact subset $K$ of $\rchi_X$ such that,  for every $\rho \in \rchi_X \minus K$,  
if $\alpha$ is a monotone staircase closed curve in $E_{X, \rho}$, such that 
\begin{itemize}
\item the total vertical length of $\alpha$ is more than   $\ep$ times  the total horizontal length of $\alpha$, and 
\item $\alpha$ is disjoint from the $\ep$-neighborhood of the singular set in the normalized metric $E^1_{X, \rho}$, 
\end{itemize}
then $\Ep_{X, \rho} \ti\alpha$ is a $(1 +\ep, \ep)$-quasi-geodesic with respect to the vertical length. 
\end{lemma}
\begin{lemma}
Let $\alpha_X$ be a staircase curve carried by $t_{X, \rho}$ satisfying the conditions in Lemma \ref{StaircaseCurveAndQuasiGeodesic}.
Then, there is a staircase geodesic closed curve $\alpha_Y$ carried by $\tT_{Y, \rho}'$  satisfying the conditions in Lemma \ref{StaircaseCurveAndQuasiGeodesic}, such that 
the image of $\alpha$ by the semi-diffeomorphism $\tT_{X, \rho} \to \tT_{Y, \rho}'$ is homotopic to $\alpha_Y$ in the train-track $\tT_{Y, \rho}'$.
\end{lemma}
\begin{proof}
The proof is left for the reader. 
\end{proof}
Let $W_Y$ be  a realization of $[V_Y]_X$ on $\tT_{X, \rho}$ (\S \ref{sTrainTracksForDiagonal}).
Let $x$ be a point of the intersection of the realization $W_Y$ and a horizontal edge of $h_X$ of $\tT_{X, \rho}$. 
Let $y$ be a corresponding point of $V_{Y, \rho}$ (on $E_{Y, \rho}$).
Recall that $r$ is the train-track parameter, so that, in particular, horizontal edges are distance, at least, $r$ away from the singular set in the normalized Euclidean metric.
Let $v_x$ be a vertical segment of length $r/2$  on $E^1_{X, \rho}$ such that $x$ is the middle point of $v_x$.
Similarly, let $v_y$ be the vertical segment of length $r/2$ on $E^1_{Y, \rho}$ such that $y$ is the middle point of $v_y$. 
We normalize the Epstein surfaces for $C_{X, \rho}$ and $C_{Y, \rho}$ so that they are $\rho$-equivariant for a fixed representation $\rho\col \pi_1(S) \to \PSL(2, \C)$ (not a conjugacy class).
\begin{proposition}[Corresponding vertical edges are close in $\H^3$]\Label{CorrespondingVerticalSegmentsNearHorizontalEdges}
For every $\ep > 0$, there is a compact subset $K$ in $\rchi$, such that, for every $\rho \in \rchi_X \cap \rchi_Y \minus K$, 
if $v_X$ and $v_Y$ are vertical segments of $E_{X, \rho}$ and  of $E_{Y, \rho}$, respectively, as above, 
then there is a (bi-infinite) geodesic $\ell$ in $\H^3$ satisfying the following:
\begin{itemize}
\item  
$\Ep_{X, \rho} v_X$ is $\ep$-close to a geodesic segment  $\alpha_X$  of $\ell$ in $C^1$-metric;
\item 
$\Ep_{Y, \rho} v_Y$ is $\ep$-close to a geodesic segment  $\alpha_Y$  of $\ell$ in $C^1$-metric; 
\item 
if $p_X$ and $p_Y$ are corresponding endpoints of $\alpha_X$ and $\alpha_Y$, then the distance between $\Ep_{X, \rho} p_X$ and  $\Ep_{Y, \rho} p_Y$  is at most  $\ep$ times the diameters of $E_{X, \rho}$ and $E_{Y, \rho}$. 
\end{itemize}
\end{proposition}

\begin{proof}
Then, pick a $L^\infty$-geodesic staircase closed curves $\ell_{X,1}, \ell_{X, 2}$ on $E_{X, \rho}$ containing $v_x$
such that, for $i = 1,2$, by taking appropriate lift $\ti\ell_{X, 1}$ and $\ti\ell_{X, 2}$ to $\ti{E}_{X, \rho}$,

\begin{enumerate}
\item  $\ell_{X, i}$ is carried by $\tT_{X, \rho}$;
\item  $\ti\ell_{X, 1} \cap \ti\ell_{X, 2}$ is a single staircase curve connecting singular points of $\ti{E}_{X, \rho}$, and the projection of  $\ti\ell_{X, 1} \cap \ti\ell_{X, 2}$ to $E_{X, \rho}$ does not meet a branch of $\tT_{X, \rho}$ more than  twice; \Label{iCommonstaircaseCurve}
\item if a branch $B$ of $\ti{T}_{X, H_X}$ intersects both  $\ti\ell_{X, 1}$ and $\ti\ell_{X, 2}$, then $B$ intersects   $\ti\ell_{X, 1} \cap \ti\ell_{X, 2}$;
\item  $\ti\ell_{X, 1}$ and $\ti\ell_{X, 2}$ intersect, in the normalized metric of $\ti{E}^1_{X, \rho}$, the $\ep$-neighborhood of the singular set only in the near the endpoints of  $\ti\ell_{X, 1} \cap \ti\ell_{X, 2}.$ 
\end{enumerate}
Then, there are homotopies of $\ell_{X, 1}, \ell_{X, 2}$ to staircase vertically-geodesic closed curves $\ell'_{X,1}, \ell'_{X, 2}$ carried by $\tT_{X, \rho}$, such that the homotopies are supported on the $2\ep$-neighborhood of the singular set of $E^1_{X, \rho}$ and that $\ell'_{X,1}, \ell'_{X, 2}$ are  disjoint from the $\ep$-neighborhood of the zero set. 
Then $\Ep_{X, \rho} \ti\ell_{X, 1}'$ and $\Ep_{X, \rho} \ti\ell_{X,2 }'$ are $(1 + \ep, \ep)$-quasi-geodesics which are close only near the segment corresponding to $\ti\ell_{X, 1} \cap \ti\ell_{X, 2}$. 

Pick closed geodesic staircase-curves $\ell_{Y, 1} , \ell_{Y, 2}$ on $E_{Y, \rho}$, such that
\begin{itemize}
\item $\ell_{Y, i}$ contains $v_y$;
\item the semi-diffeomorphism $\tT_{X, \rho} \to \tT_{Y, \rho}$ takes $\ell_{X, i}'$ to a  curve homotopic to $\ell_{Y, i}$ on $\tT_{Y, \rho}$;
\item  $\ell_{Y, i}$  is carried by $\tT_{Y, \rho}$;
\item $\ell_{Y, i}$ is disjoint from $N_\ep^1 Z_{X, \rho}$.
\end{itemize}
Let $\alpha$ be the geodesic such that a bounded neighborhood of $\alpha$ contains the quasi-geodesic $\Ep_{X, \rho} \ti\ell_{X, i}'$. 
Let $\ti\ell_{Y, i}$ be a lift of $\ell_{Y, i}$ to $\ti{E}_{Y, \rho}$ corresponding to $\ti\ell_{X, i}'$ (connecting the same pair of points in the ideal boundary of $\ti{S}$).

\begin{lemma}
For every $\ep > 0$, if a compact subset $K$ of $\rchi$ is sufficiently large and $\upsilon > 0$ is sufficiently small, then, for all $\rho \in \rchi_X \cap \rchi_Y \minus K$, 
$\Ep_{Y, \rho} \ti\ell_{Y, i}$ is $(1 + \ep, \ep)$-quasi-isometric with respect to the vertical length for both $i = 1, 2$. 
\end{lemma}

Then  $\Ep_{X, \rho} \ti\ell_{X, 1}' \cup \ti\ell_{X, 2}'$ and  $\Ep_{Y, \rho} \ti\ell_{Y, 1} \cup \ti\ell_{Y, 2}$ are both $\ep$-close in the Hausdorff metric of $\H^3$. 
Therefore, corresponding endpoints  of  $\Ep_{Y, \rho}  \ti\ell_{X, 1}' \cap \ti\ell_{X, 2}'$ and  $\Ep_{Y, \rho} \ti\ell_{Y, 1}' \cap \ti\ell_{Y, 2}'$ have distance, at most, $\ep$ times the diameters of $E_{X, \rho}$ and $E_{Y, \rho}$.
By (\ref{iCommonstaircaseCurve}), the length of  $\ti\ell_{X, 1}' \cap \ti\ell_{X, 2}'$  can not be too long relative to the diameter of $E_{X, \rho}$.
Letting $\ell$ be the geodesic in $\H^3$ fellow-traveling with $\Ep_{X, \rho} \ti\ell_{X, 1}'$ (or $\Ep_{X, \rho} \ti\ell_{X, 2}'$), 
the vertical segment $v_x$ and $v_y$ have the desired property.
\end{proof}

Finally \Cref{PleatedSurfacesAreClose} follows from the next proposition.
\begin{proposition}
Suppose that a branch  $B_Y'$ of  the train track $\tT_{Y, \rho}'$ corresponds transversally to a branch $B_X$ of  $\tT_{X, \rho}$. 

Then, there is an $\ep$-small isotopy of  $B_Y'$ in the normalized surface $E_{Y, \rho}^1$ such that 
\begin{itemize}
\item in the complement of the $\frac{r}{2}$-neighborhood of the zero set, every point of $B_Y'$ moves along the vertical foliation $V_{Y, \rho}$, and 
\item after the isotopy $\hat\beta_{X, \rho} | B_X $ and $\hat\beta_{Y, \rho} \vert B_Y'$ are $\ep$-close pointwise  by a diffeomorphism $\psi\col B_Y' \to B_X$.
\end{itemize}
\end{proposition}

\begin{proof}
By Proposition \ref{CorrespondingVerticalSegmentsNearHorizontalEdges},  there is an $\ep$-small isotopy of the boundary of $B_Y'$ satisfying the conditions on the boundaries of the branches. 
Since the branches are transversal, by Theorem \ref{ThurstonLaminationAndVerticalFoliationOnPolygons}, if $K$ is sufficiently large, then the restriction of $L_{X, \rho}$ to $B_X$ and $L_{Y, \rho}$ on $B_Y$ are $\ep$-quasi parallel on the hyperbolic surface $\tau_{X, \rho}$ (\Cref{AlmostParalellLaiminationAndFoliation}). 
Therefore
we can extend to the interior of the branch by taking an appropriate diffeomorphism $\psi\col B'_Y \to B_X$.
\end{proof}

\section{Compatible circular train-tracks}\Label{sCircularTraintracks}
In \S \ref{sCompatibleTraintrackDecompositions}, for every $\rho$ in $\rchi_X \cap \rchi_Y$ outside a large compact $K$, we constructed semi-diffeomorphic train-track structures $\tT_{X, \rho}$ and $\tT_{Y, \rho}'$ of the flat surfaces $E_{X,\rho}$ and $E_{Y, \rho}$, respectively. 
In this section, as $E_{X,\rho}$ and $E_{Y, \rho}$ are  the flat structures on $C_{X, \rho}$ and $C_{Y, \rho}$, using Theorem \ref{PleatedSurfacesAreClose},
we homotope $\tT_{X, \rho}$ and $\tT_{Y, \rho}'$ to make them circular in a compatible manner. 

\subsection{Circular rectangles}
A {\sf round cylinder}  is a cylinder on $\CP^1$ bounded by two disjoint round circles. 
Given a round cylinder $A$,  the boundary components of $A$ bound unique (totally geodesic) hyperbolic planes in $\H^3$, and 
 there is a unique geodesic $\ell$ orthogonal to both hyperbolic planes. 
Moreover $A$ is foliated by round circles which, in $\H^3$, bound hyperbolic planes orthogonal to $\ell$\, --- we call this foliation the {\sf horizontal foliation}. 
In addition, $A$ is also foliated by circular arcs which are contained in round circles bounding hyperbolic planes, in $\H^3$, containing $\ell$\,
--- we call this foliation the {\sf vertical foliation}.
Clearly, the horizontal foliation is orthogonal to  the vertical foliations of $A$.

\begin{definition}\Label{dCircularRectangle}
Let $\RRR$ be a $\CP^1$-structure on a marked rectangle $R$, and let $f\col R \to \CP^1$ be its developing map.
Then $\RRR$ is {\sf circular} if there is a round cylinder $A$ on $\CP^1$ such that
\begin{itemize}
\item the image of $f$ is contained in $A$;
\item  the horizontal edges of $R$ are immersed into different boundary circles of $A$;
\item for each vertical edge $v$ of $R$, its development  $f(v)$ is a simple arc on $A$  transverse to the horizontal foliation.
\end{itemize}
\end{definition}

Given a circular rectangle $\RRR$, the {\sf support} of $\RRR$ consists of the round cylinder $A$ and the simple arcs on $A$ which are the developments of the vertical edges of $\RRR$ in \Cref{dCircularRectangle}. 
We denote the support by $\Supp \RRR$.
We can pull-back the horizontal foliation on $A$ to a foliation on $\RRR$ by the developing map, and call it the {\sf horizontal foliation} of $\RRR$.

Given projective structures $\RRR$ and $\QQQ$ on a  marked rectangle $R$,  we say that $\PPP$ and $\QQQ$ are {\sf compatible} if $\Supp \RRR = \Supp \QQQ$. 
Let $\RRR$ be a circular rectangle, such that the both vertical edges are supported on the same arc $\alpha$ on a circular cylinder. 
Then, we say that $\RRR$ is  {\sf semi-compatible} with $\alpha$.

\subsubsection{Grafting a circular rectangle}(See \cite{Baba-10}.)
Let $\RRR$ be a  circular $\CP^1$-structure on a marked rectangle $R$. 
Let $A$ be the round cylinder in $\CP^1$ which supports $\RRR$.  
Pick an arc $\alpha$ on $\RRR$, such that  $\alpha$ connects the horizontal edges and it is transversal to the horizontal foliation of $\RRR$. 
Then $\alpha$ is embedded into $A$ by $\dev \RRR$ --- we call such an arc $\alpha$ an {\sf admissible} arc.
By cutting and gluing $A$ and $\RRR$ along $\alpha$ in an alternating manner, we obtain a  new circular $\CP^1$-structure on $R$ whose support still is  $\Supp \RRR$.
This operation is the {\sf grafting} of $\RRR$ along $\alpha$, and the resulting structure on $R$ is denoted by $\Gr_\alpha \RRR$.

One can easily show that $\Gr_\alpha \RRR$ is independent of the choice of the admissible arc $\alpha$, since an isotopy of $\alpha$ preserving its initial conditions does not change $\Gr_\alpha \RRR$. 
    
\subsection{Circular staircase loops}
Let $C = (f, \rho)$ be a $\CP^1$-structure on $S$. 
 A {\sf topological staircase curve} is a piecewise smooth curve, such that
 \begin{itemize}
 \item   its smooth segments are labeled by ``horizontal'' or ``vertical'' alternatively along the curve, and  
\item at every singular point, the horizontal and vertical tangent directions are linearly independent in the tangent space. 
\end{itemize}
Then, a topological staircase curve $s$ on $C$ is {\sf circular}, if the following conditions are satisfied:
 Letting $\ti{s}$ be a lift of $s$ to $\ti{S}$, 
\begin{itemize}
\item every horizontal segment $h$ of $\ti{s}$ is immersed into a round circle in $\CP^1$ by $f$, and
\item for every vertical segment $v$ of $\ti{s}$, letting $h_1, h_2$ be the horizontal edges starting from the endpoints of $v$, 
\begin{itemize}
\item  the round circles $c_1, c_2$ containing $f(h_1)$ and $f(h_2)$ are  disjoint, and
\item $f | v$ is contained in the round cylinder bounded by $c_1, c_2$ and, it is transverse to the horizontal foliation of the round cylinder.
\end{itemize} 
\end{itemize}

\subsection{Circular polygons}
Let $P$ be a marked polygon with even number of edges.
Then, let  $e_1, e_2, \dots, e_{2n}$ denote its edges in the cyclic order so that the edges with odd indices are vertical edges and with even indices horizontal edges. 
Suppose that $c_2, c_4\dots c_{2n}$ are round circles in $\CP^1$ such that, for every $i \in \Z / n \Z$,
\begin{itemize}
\item  $c_{2i}$ and $c_{2(i + 1)}$ are disjoint,  and
\item  $c_{2(i-1)}$ and $c_{2(i + 1)}$ are contained in the same component of $\CP^1 \minus c_{2i}$.
\end{itemize}

Let $\AAA_i$ denote the round cylinder bounded by $c_{2i}$ and $c_{2(i + 1)}$.
A circular $\CP^1$-structure $\PPP$ on $P$ is  {\sf supported} on $\{c_{2i}\}_{i=1}^n$
if
\begin{itemize}
\item $e_{2i}$ is  immersed into the round circles of $c_{2i}$  by $\dev \PPP$ for every $i = 1, \dots n$, and
\item $e_{2i +1}$ is immersed into $\AAA_i$ and its image is transversal to the horizontal foliation of $\AAA_i$ (Figure \ref{fCircularPolygon}) for every $i = 0, 1, \dots, n-1$.
\end{itemize}

\begin{figure}
\begin{overpic}[scale=.07,
] {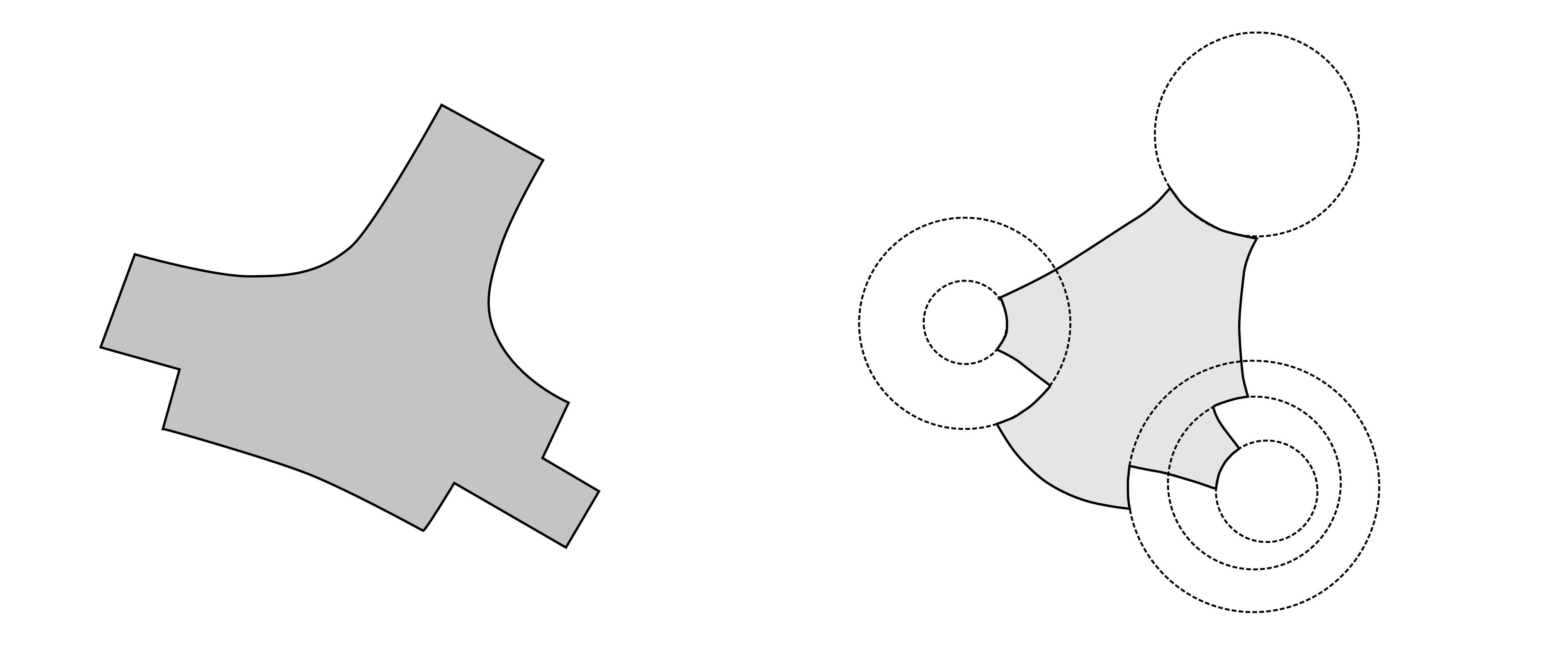} 
\linethickness{1pt}
\put(40, 20){\color{black}\vector(1,0){15}}
   \put(42 , 22){\textcolor{black}{$\dev \PPP$}}  
    \put(65 , 33){\textcolor{black}{$\CP^1$}}  
      \put(18 ,18 ){$P$}  
      \end{overpic}
\caption{A development of a projective polygon supported on round circles (when the developing map is injective).}\Label{fCircularPolygon}
\end{figure}

Let $\PPP$ be a circular $\CP^1$-structure on a polygon $P$ supported on a circle system $\{ c_{2i}\}_i^n$.
For $\ep > 0$,  $\PPP$ is {\sf $\ep$-circular}, if 
\begin{itemize}
\item for every vertical edge $v_i$ is $\ep$-parallel to the vertical foliation $\VVV$ of the support cylinder $\AAA_i$, and
\item the total transversal measure of $v$ given by the vertical foliation $\VVV$ is less than $\ep$. 
\end{itemize} 
 (Here, by the ``total'' transversal measure,  we mean that if $v$ intersects a leaf of $\VVV$ more than once, and the measure is counted with multiplicity.)
 
Let $\PPP_1, \PPP_2$ be  circular $\CP^1$-structures on a $2n$-gon $P$.
Then  $\PPP_1$ and $\PPP_2$ are {\sf compatible} if,  for each $i = 1, \dots, n$,  $\dev \PPP_1$ and $\dev \PPP_2$ take $e_{2 i}$ to the same round circle  and  the arcs  $f_1( v_{2i - 1} )$  and $f_2 ( v_{2 i - 1})$ are the same.

Let $A$ be a flat cylinder with geodesic boundary; then its universal cover $\ti{A}$ is an infinite Euclidean strip. 
A projective structure $(f, \rho)$ on $A$ is {\sf circular}, if the developing map $f\col \ti{A} \to \CP^1$ is a covering map onto a round cylinder in $\CP^1$. 

Next, let $A$ be a spiral cylinder. 
Then each boundary component $b$ of $A$ is a monotone staircase loop.  
Let $\ti{b}$ be the lift of $b$ to the universal cover $\ti{A}$.
Let $\{e_i \}_{i \in \Z}$ be the segments of $\ti{b}$ linearly indexed  so that $e_i$ with an odd index is a vertical edge and with an even index is a horizontal edge; clearly
  $\ti{b} = \cup_{i \in \Z} e_i$. 
Then, a $\CP^1$-structure $(f, \rho)$ on $A$ is {\sf circular}, if, for each boundary staircase loop $b$ of $A$  and each $i  \in \Z$, 
\begin{itemize}
\item the horizontal edge $e_{2i}$ is immersed into a round circle $c_i$ on $\CP^1$; 
\item $c_{i -1}$, $c_i$ and $c_{i + 1}$ are disjoint, and the round annulus bounded by $c_{i-1}$ and $c_{i +1}$ contains $c_i$ in its interior;
\item $f$ embeds  $v_i$  in the round cylinder $\AAA_i$ bounded  by $c_i$ and $c_{i +1}$, and $f(v_i)$ is transverse to the circular foliation of $\AAA_i$. 
\end{itemize}

Two circular $\CP^1$-structures $\AAA_1 = (f_1, \rho_1), \AAA_2 = (f_2, \rho_2)$ on  a spiral cylinder $A$ are {\sf compatible}  
\begin{itemize}
\item $\rho_1$ is equal to $\rho_2$ up to conjugation by an element of $\PSL(2, \C)$ (thus we can assume $\rho_1 = \rho_2$);
\item for each boundary component $h$  of $\ti{A}$,  $f_1$ and $f_2$ take $h$ to the same round circle;
\item  for each vertical edge $v$ of $\ti{A}$, $f_1 |v = f_2 | v$. 
\end{itemize}

More generally, let $\FFF = (f_1, \rho_1)$ and $\FFF' = (f_2, \rho_2)$ be two circular  $\CP^1$-structures on stair-case surfaces  $F$ and $F'$.
First suppose that there is a diffeomorphism  $\phi \col F \to F'$, which takes the vertices of $F$ bijectively to those of $F'$.
Then $\FFF$ is {\sf compatible} with $\FFF'$ if
\begin{itemize}
\item $\rho_1$ is conjugate to $\rho_2$ (thus we can assume that $\rho_1 = \rho_2$);
\item for every vertex $p_1$ of $\FFF_1$, the development of $p_1$ coincides with the development of $\phi (p_2)$;
\item  for every a horizontal edge $h$ of $\FFF_1$, letting $h'$ be its corresponding horizontal edge of $\FFF'$,  then the developments of $h$ and $h'$ are contained in the same round circle;
\item for every  vertical edge $v$ of $\FFF$, letting $v'$ be its corresponding edge $v'$ of $\FFF'$, then the developments of $v'$ and $v$  coincide. 
\end{itemize}

Next, instead of a diffeomorphism, we suppose that there is a semi-diffeomorphism $\phi\col F \to F'$.
Then $\FFF$ is {\sf semi-compatible} with $\FFF'$ if
\begin{itemize}
\item $\rho_1$ is conjugate to $\rho_2$ (thus we can assume that $\rho_1 = \rho_2$);
\item for every vertex $p_1$ of $\FFF_1$, the development of $p_1$ coincides with the development of $\phi (p_2)$;
\item  if a horizontal edge $h$ of $\FFF$ corresponds to a horizontal edge  $h'$ of $\FFF'$, then $h$ and $h'$ are supported on the same round circle on $\CP^1$;
\item for every  vertical edge $v$ of $\FFF$, letting $v'$ be its corresponding vertical edge (segment) of $\FFF'$, then the developments of $v'$ and $v$  coincide. 
\end{itemize}
\subsection{Construction of circular train tracks $\TTT_{Y,\rho}$}\Label{sProjectiveTraintrackForX}
~\\

In this section, if $\rho$ is in $\rchi_X \cap \rchi_Y$ minus a large compact subset, we construct a circular train-track structure of $C_{Y, \rho}$ related to the polygonal train-track decomposition $\tT'_{Y, \rho}$. 

Two train-track structures $T_1, T_2$  on a flat surface $E$ is $(p, q)$-{\sf quasi-isometric} for $p > 1$ and $q > 0$ if there is a continuous $(p, q)$-quasi-isometry $\phi\col E \to E$ homotopic to the identity such that $\phi(T_1) = T_2$ and  the restriction of $\phi$ to $T_1$ is a homotopy equivalence between $T_1$ and  $T_2$.
   
\begin{theorem}\Label{ProjectiveTraintrackForY}
For every $\ep > 0$, there is a bounded subset $K = K_\ep$ in $\rchi_X \cap \rchi_Y$, such that, 
for every $\rho \in \rchi_X \cap \rchi_Y \minus K$,
there is an $\ep$-circular surface train track decomposition $\TTT_{Y, \rho}$ of $C_{Y, \rho}$ with the following properties: 
\begin{enumerate}
\item $\TTT_{Y, \rho}$ is diffeomorphic to $\tT_{Y, \rho}'$, and it is  $(1 + \ep, \ep)$-quasi-isometric to both $\tT_{Y, \rho}$ and $\tT_{Y, \rho}'$ in the normalized metric $E^1_{Y, \rho}$.  \Label{iQIperturbationOfTraintrack}
\item
For every vertical edge $v$ of $\tT_{Y, \rho}'$, its corresponding edge of $\TTT_{Y, \rho}$ is contained in the leaf of the vertical foliation $V_{Y, \rho}$.
 \Label{iVerticalEdgesInLeaves}
\item 
For a branch $B_X$ of $\tT_{X, \rho}$, letting  $B_Y$ be its corresponding branch of  $\tT_{Y, \rho}'$ and letting $\BBB_Y$ be the branch of $\TTT_{Y, \rho}$ corresponding to $B_Y$, 
the restriction of $\hat\beta_{X, \rho}$ to $ \bdr \ti{B}_X$ is $\ep$-close to the restriction of $\hat\beta_{Y, \rho}$ to $\bdr \ti\BBB_Y$ pointwise; 
moreover, if $B_X$ is a  transversal branch, then $\hat\beta_{X, \rho} | \ti{B}_X$ is $\ep$-close to $\hat\beta_{Y, \rho} | \ti\BBB_Y$ pointwise.
\Label{iHomotopyMaksBranchesAreCloseInH}
\end{enumerate}
\end{theorem}
We fix a metric on the unit tangent bundle of $\H^3$ which is left-invariant under $\PSL(2, \C)$.
\begin{proposition}\Label{VerticalTangentsAlognHorizontalEdges}
For every $\epsilon > 0$, if a bounded subset $K_\ep$ of $\rchi_X$ is sufficiently large, then, for every $\rho \in \rchi_X \minus K_\ep$ and every horizontal edge $h$ of $\tT_{X, \rho}$, the $\Ep^\ast_{X, \rho}$-images of the vertical unit tangent vectors of  along $h$ are $\ep$-close. 
\end{proposition}
\begin{proof}
The assertion immediately follows from Theorem \ref{HorizontalAndVerticalEstimate} (\ref{iCloseVerticalTangentVectors}).
\end{proof}

Recall that we have constructed a system of equivariant circles for horizontal edges of $\ti\tT_{X, \rho}$ in Lemma \ref{EquivariantCircles}. 
Let $h = [u, w]$ denote the horizontal edge of $\tT_{X, \rho}$ where $u, w$ are the endpoints.
We shall perturb the endpoints of each horizontal edge of $\tT_{Y, \rho}$ so that the endpoints map to the corresponding round circle.  
\begin{proposition}\Label{CircularEdges}
For every $\ep > 0$, there are  sufficiently small $\del > 0$  and a (large) bounded subset  $K_\ep$ of $\rchi_X \cap \rchi_Y$ satisfying the following: 
For every $\rho \in \rchi_X \cap \rchi_Y \minus K_\ep$, if  $\cc = \{c_h\}$ is a circle system for horizontal edges $h$ of $\tT_{X, \rho}$ given by Lemma \ref{EquivariantCircles} for $\del$, then,
for every horizontal edge $h= [u, w]$ of $\ti\tT_{Y, \rho}$, there are, with respect to the normalized metric $E_{Y, \rho}^1$,  $\ep$-small perturbations $u'$ and $w'$ of  $u$ and $w$ along $V_{Y, \rho}$, respectively, such that $f_{Y, \rho}(u')$ and $f_{Y, \rho}(w')$ are contained in the round circle $c_h$. \Label{iVerticesOnCircles}
\end{proposition}

\begin{proof}
This follows from Theorem \ref{PleatedSurfacesAreClose} and Lemma \ref{EquivariantCircles}  (\ref{iAlmostOrthogonal}).
\end{proof}
  
\proof[Proof of Theorem \ref{ProjectiveTraintrackForY}]

By Proposition \ref{CircularEdges}, for each horizontal edge $h = [u, w]$ of $\tT_{Y, \rho}'$, there is an $\ep$-homotopy of $h$ to the circular segment $h'$ the perturbations $u', w'$ such that, letting $\ti{h}$ be a lift of $h$ to $\ti{E}_{Y, \rho}$, 
the corresponding lift $\ti{h}'$ of $h'$ is immersed into the round circle $c_{\ti{h}}$. 
For each vertical edge $v$ of $\tT_{Y, \rho}'$,  at each endpoint of $v$, there is a horizontal edge of $\tT_{Y, \rho}'$ starting from the point; then the round circles corresponding to the horizontal edges bound a round cylinder.

Note that a vertex $u$ of $\tT_{Y, \rho}'$ is often an endpoint of different horizontal edges $h_1$ and $h_2$. 
Thus, if the perturbations $u_1'$ and $u_2'$ of $u$ are different for $h_1$ and $h_2$, then $\TTT_{Y, \rho}$ has a new short vertical edge connecting  $u_1'$ and $u_2'$, and $\TTT_{Y, \rho}$ is non-diffeomorphic to $\tT_{Y, \rho}'$. 

Recall that the $\del /4$-neighborhood of the singular points of $E_{Y, \rho}^1$ is disjoint from the one-skeleton of $\tT_{Y, \rho}$ by \Cref{ModifiedTTForY}. 
Thus, every vertical edge $v$ of $\tT_{Y, \rho}'$ is $\ep$-circular with respect to the round cylinder by \Cref{VerticalSegmentAlmostZeroMeasure}.
Thus we have (\ref{iVerticalEdgesInLeaves}).
Thus we obtained an $\ep$-circular train-track decomposition $\TTT_{Y, \rho}$ of $E_{Y, \rho}$.

As the applies homotopies are $\ep$-small, $\TTT_{Y, \rho}$ are $\ep$-close to $\tT_{Y, \rho}'$ (\ref{iQIperturbationOfTraintrack}).
Thus we may, in addition, assume that $\TTT_{Y, \rho}$ is $\ep$-close to $\tT_{Y, \rho}'$ by \Cref{ModifiedTTForY}. 
Moreover,
Theorem 
\ref{PleatedSurfacesAreClose} give (\ref{iHomotopyMaksBranchesAreCloseInH}).
 \Qed{ProjectiveTraintrackForY}

\subsection{Construction of $\TTT_{X, \rho}$}\Label{sProjectiveTraintrackX}
 Given a train-track structure on a surface, the union of the edges of its branches is a locally finite graph embedded on the surface.  
 An {\sf edge} of a train-track decomposition is an edge of the graph, which contains no vertex in its interior  (whereas an edge interior of a branch may contain a vertex of the train track).  
\begin{definition}
Let $C$, $C'$ be $\CP^1$-structures on $S$ with the same holonomy $\rho\col \pi_1(S) \to \PSL(2, \C)$, so that $\dev C$ and $\dev C'$ are $\rho$-equivariant.  
A circular train-track decomposition $\TTT = \cup_i \BBB_i$ of $C$  is {\sf semi-compatible} with a circular train-track decomposition $\TTT' = \cup \BBB_j'$ of $C'$ if there is a marking-preserving continuous map $\Theta \col C \to C'$ such that, 
for each branch $\BBB$ of $\TTT$, $\Theta$ takes $\BBB$ to a branch of $\BBB'$ of $\TTT'$, and that $\BBB$ and $\BBB'$ are compatible by $\Theta$.
\end{definition}

\begin{boxedlaw}{13cm}
\begin{theorem}\Label{CircularTraintrackX}
For every $\ep > 0$, if a bounded subset $K_\ep$ in $\rchi_X \cap \rchi_Y$ is sufficiently large, then,  for every $\rho \in  \rchi_X \cap \rchi_Y \minus K_\ep$, there is an $\ep$-circular train track decomposition $\TTT_{X, \rho}$ of $C_{X, \rho}$, such that 
\begin{enumerate}
\item $\TTT_{X, \rho}$ is semi-compatible with $\TTT_{Y, \rho}$, and 
\item $\TTT_{X,\rho}$ additively $2\pi$-Hausdorff-close to $\tT_{X,\rho}$ with respect to the (unnormalized) Euclidean metric $E_{X, \rho}$: More precisely,  in the vertical direction, $\TTT_{X, \rho}$ is $\ep$-close to $\tT_{X, \rho}$, and in the horizontal direction, $2\pi$-close in the Euclidean metric of $E_{X, \rho}$ for all $\rho \in \rchi_X \minus K_\ep$. \Label{iClosenessInHorizontalAndVerticalDirections}
\end{enumerate}
\end{theorem}
\end{boxedlaw}
\proof
First, we transform $\tT_{X, \rho}$ by perturbing horizontal edges so that horizontal edges are circular. 
Recall that, the branches of $\TTT_{Y, \rho}$ are circular with respect to a fixed system $\cc$ of equivariant circles given by Lemma \ref{EquivariantCircles}. 
Thus, the $\beta_{X, \rho}$-images of vertical tangent vectors along $h$ are $\ep$-close to a single vector orthogonal to the hyperbolic plane bounded by $c_h$.
Therefore, similarly to Theorem \ref{ProjectiveTraintrackForY}, we can modify the train-track structure $\tT_{X, \rho}$ so that horizontal edges are circular and $\ep$-Hausdorff close to the original train-track structure in the Euclidean metric of $E_{X, \rho}$ (this process may create new short vertical edges). 
Thus we obtained an $\ep$-circular train track $\tT_{X, \rho}'$ whose horizontal edges map to their corresponding round circles of $\cc$.

Next, we make the vertical edges compatible with $\TTT_{Y, \rho}$.
Recall that $\tT_{X, H_X}$ has no rectangles with short vertical edges (Lemma \ref{PolygonalTraintrackX}).
Therefore, we have the following. 
\begin{lemma}\Label{EnoughHorizontalSpace}
For every $R > 0$, if the bounded subset $K$ of $\rchi$ is sufficiently large, then, for each vertical edge of $\tT_{X, \rho}'$,   the horizontal distance to adjacent vertical edges is at least $R$.
\end{lemma}

Thus, by Lemma \ref{EnoughHorizontalSpace}, there is enough room to move vertical edges, less than $2\pi$, so that the train-track is compatible with $\TTT_{Y, \rho}$ along vertical edges as well. 

Since $\tT_{X, \rho}$  is semi-diffeomorphic to $\tT_{Y, \rho}$ (Proposition \ref{PolygonalTaintrackY} (\ref{iSemiDiffeomoprhictT})), $\TTT_{X, \rho}$ is semi-compatible with $\TTT_{Y, \rho}$.
\Qed{CircularTraintrackX}

\section{Grafting cocycles and intersection of holonomy varieties}\Label{sGraftingCocycle}
   In this section, given a pair of  $\CP^1$-structures on $S$ with the same holonomy, we shall construct a $\Z$-{\sf valued cocycle} under the assumption that the holonomy is outside of an appropriately large compact subset of the character variety $\chi$. 
Namely, we will construct a train-track graph with compatible $\Z$-valued weights on the branches and its immersion into $S$ (see \cite{Penner-Harer-92} for train-track graphs). 
This embedding captures, in a way,  the ``difference'' of the $\CP^1$-structures sharing holonomy. 
If a smooth arc on $S$ is transversal to the immersed train-track graph, then the sum of the $\Z$-weights at the transversal intersection points is an integer--- this functional defined on transversal arcs is called a {\sf transversal cocycle}. 
Note that this cocycle value does not change under the regular homotopy of the arc if it retains the transversality. 
In particular, given a simple closed curve on a surface, we first homotopy the loop so that it has a minimal geometric intersection with the immersed train-track graph, and then consider its transversal cycle with the train-track graph. 
In this manner, we obtain a functional on the set of homotopy classes on the simple closed curves, which we call a {\sf grafting cocycle}. 

Goldman showed that every $\CP^1$-structure with Fuchsian holonomy $\pi_1(S) \to \PSL_2\C$ is obtained by grafting the hyperbolic structure with the Fuchsian holonomy along a $\Z$-weighted multi-loop on $S$ (\cite{Goldman-87}). 
The grafting cocycles that we construct in this paper can be regarded as a generalization of such weighted multiloops.
\subsection{Relative degree of rectangular $\CP^1$-structures}  \Label{sGrafting}
Let $a < b$ be real numbers. 
Then let $f, g\col [a,b] \to \s^1$ be orientation preserving immersions or constant maps, such that $f(a) = g(a)$ and $f(b) = g(b)$.  
\begin{deflemma}\Label{dRelativeDegree}
The integer
$\sharp f^{-1}(x) - \sharp g^{-1}(x)$  is independent on $x \in \s^1 \minus \{f(a), f(b)\}$, where $\sharp$ denotes the cardinality.
 We call this integer the {\sf degree of $f$ relative to $g$},  or simply, the {\sf relative degree}, and denote it by $\deg (f, g)$. 
\end{deflemma}
Clearly, it is not important that $f$ and $g$ are defined on the same interval as long as corresponding endpoints map to the same point on $\s^1$.
Moreover, the degree is additive in the following sense.
\begin{lemma}[Subdivision of relative degree]\Label{RelativeDegreeSubdivision}
Suppose in addition that  $f(c) = g(c)$ for some $c \in (a,b)$.
Then $$\deg(f, g) = \deg (f |_{[a,c]}, g|_{[a,c]} ) + \deg(f|_{[a,c]}, g|_{[a,c]} ).$$
\end{lemma}
The proofs of the lemmas above are elementary.
Let $R, Q$ be circular projective structures  on a marked rectangle, and suppose that $R$ and $Q$ are compatible:
By their developing maps,  corresponding horizontal edges of $R$ and $Q$ are immersed into the same round circle on $\CP^1$, and the corresponding vertices map to the same point. 

Then,  the {\sf degree} of $R$ relative to $Q$ is the degree of a horizontal edge of $R$ relative to its corresponding horizontal edge of $Q$\, --- 
we similarly denote the degree by $\deg(R, Q) \in \Z$. 
Although $R$ has two horizontal edges, this degree is well-defined:
\begin{lemma}[c.f. Lemma 6.2 in \cite{Baba-15gt}]%
The degree $\deg(R, Q)$ is independent of the choice of the horizontal edge. 
\end{lemma}
\begin{proof}
Let $A$ be the round cylinder on ${\CP}^1$ supporting both $R$ and $Q$.
Then, the horizontal foliation $\FFF_A$ of $A$ by round circles $c$ induces foliations $\FFF_R$ and $\FFF_Q$ on $R$ and $Q$, respectively.
Then, for each leaf $c$ of $\FFF_A$, the corresponding leaves $\ell_R$ and $\ell_Q$ of $\FFF_R$ and $\FFF_Q$, respectively, are immersed into $c$, and the endpoints of $\ell_R$ and $\ell_Q$ on the corresponding vertical edges of $R$ and $Q$ map to the same point on $c$. 
The degree of $\ell_R$ relative to $\ell_Q$ is an integer, and it changes continuously in the leaves $c$ of $\FFF_A$. 
Thus, the assertion follows immediately.  
\end{proof}

\begin{lemma}[cf. Lemma 6.2 in \cite{Baba-15gt}]\Label{GraftingRectangles}
Let $R, Q$ be $\CP^1$-structures on a marked rectangle with $\Supp R = \Supp Q$. 
Then
\begin{itemize}
\item 
  if $\deg(R, Q) > 0$, then $R$ is obtained by grafting $Q$ along an admissible arc $\deg(R, Q)$ times; 
  \item 
    if $\deg(R, Q) < 0$, then $Q$ is obtained by grafting $R$ along an admissible arc $ - \deg(R, Q)$ times; 
    \item 
    if  $\deg(R, Q) = 0$, then $R$ is isomorphic to $Q$ (as $\CP^1$-structures).  
\end{itemize}
\end{lemma}
By \Cref{GraftingRectangles}, the ``difference'' of $\CP^1$-rectangles $R$ and $Q$ can be represented by an arc $\alpha$ with weight $\deg(R, Q)$ such that $\alpha$ sits on the base rectangle connecting the horizontal edges.
\begin{lemma}\Label{DegreeAndDecomposiiton}
Let $R$ and $Q$ be  circular projective structures on a marked rectangle such that $\Supp R = \Supp Q$.
Let $A$ be the round cylinder on $\CP^1$ supporting $R$ and $Q$. 
Suppose that there are admissible arcs $\alpha_R$ on $R$ and $\alpha_Q$ on $Q$ which develop onto the same arc on $A$ (transversal to the horizontal foliation), so that
 the arcs decompose  $R$ and $Q$  into two circular rectangles $R_1, R_2$ and  $Q_1, Q_2$, respectively and  $\Supp R_1 = \Supp Q_1$ and $\Supp R_2 = \Supp Q_2$.
Then $$\deg (R,  Q) = \deg (R_1, Q_1)  + \deg (R_2, Q_2).$$
\end{lemma}

\begin{proof}
This follows from Lemma \ref{RelativeDegreeSubdivision}.
\end{proof}

\subsection{Train-track graphs for planar polygons}\Label{sGraphsOnPlanarPolygons}
Let $P$ be a  $L^\infty$-convex  staircase polygon in $\mathbb{E}^2$, which contains no singular points. 
We can decompose $P$ into finitely many rectangles $P_1, P_2, \dots, P_n$ by cutting $P$ along $n-1$ horizontal arcs each connecting a vertex and a point on a vertical edge.
Let $\PPP, \QQQ$ be compatible circular projective structures on $P$  such that the round circles supporting horizontal edges are all disjoint.
The decomposition $P$ into $P_1, P_2, \dots P_n$ gives  decompositions of $\PPP$  into $\PPP_1, \PPP_2, \dots, \PPP_n$ and $\QQQ$ into $\QQQ_1, \QQQ_2, \dots, \QQQ_n$ such that $\Supp \PPP_i = \Supp \QQQ_i$ for $i = 1, 2,\dots, n$.
As in \S \ref{sGrafting}, for each $i$, we obtain an arc $\alpha_i$ connecting horizontal edges of  $P_i$ with weight $\deg(\PPP_i, \QQQ_i)$.
Then, by splitting and combining $\alpha_1, \alpha_2, \dots, \alpha_n$ appropriately, we obtain a $\Z$-valued train-track graph   $\Gamma(\PPP, \QQQ)$  on $P$ transversal to the decomposition (\Cref{fTraintrackGraphOnPolygon}).
\begin{figure}
\begin{overpic}[scale=.15
] {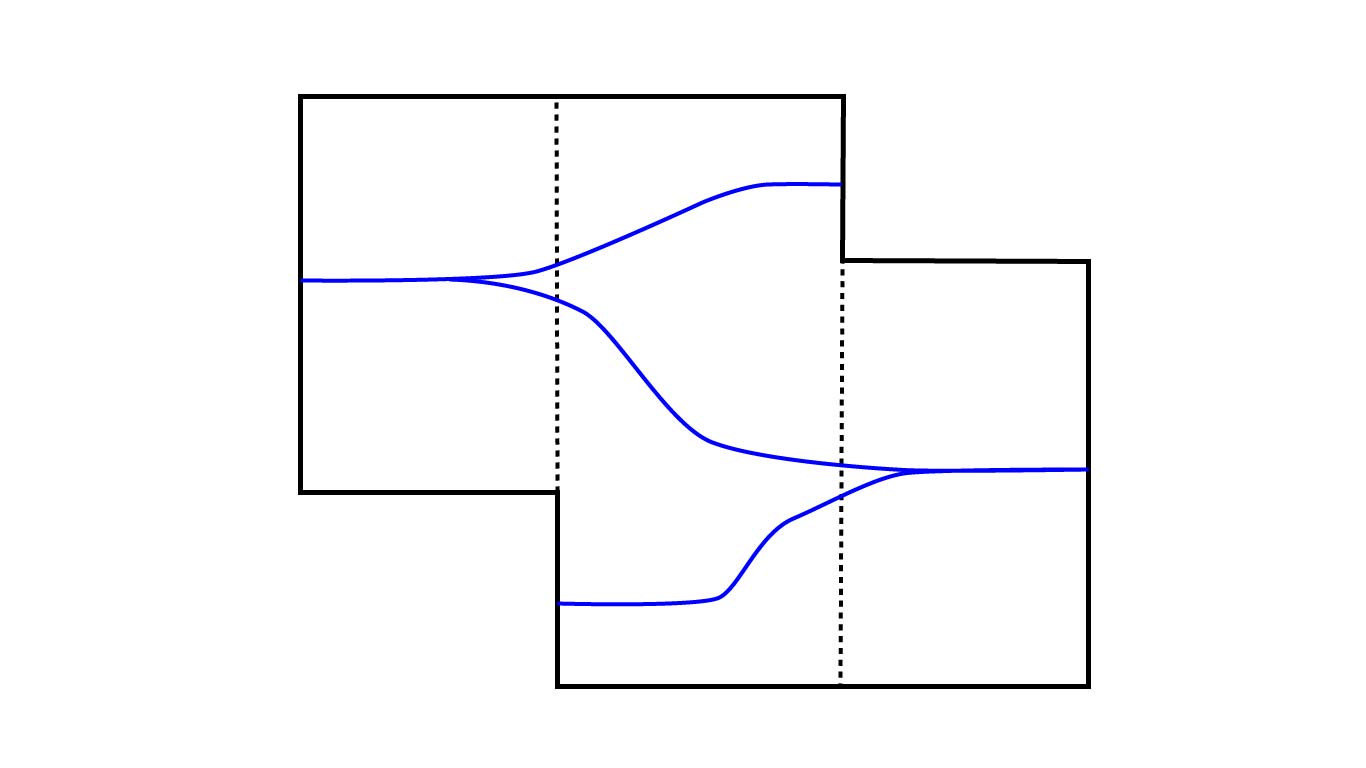} 
   \put( 29, 25){\small $P_1$}  
      \put( 50, 33){\small $P_2$}  
  \put( 67, 11){\small $P_3$}  
   \put(30 , 38 ){\textcolor{blue}{\small $1$}} 
   \put(44 , 42 ){\textcolor{blue}{\small $2$}} 
      \put(54 , 24 ){\textcolor{blue}{\small $-1$}} 
         \put(44 ,14 ){\textcolor{blue}{\small $-2$}} 
    \put(68 , 24 ){\textcolor{blue}{\small $-3$}} 
      \end{overpic}
\caption{An example of a weighted train track on an $L^\infty$-convex polygon in $\mathbb{E}^2$.}\label{fTraintrackGraphOnPolygon}
\end{figure}

\subsection{Train-track graphs for cylinders}\Label{sTrainTracksOnCylinders}
Let $A_X$ be a cylindrical branch of $\tT_{X, \rho}$, and let $A_Y$ be the corresponding cylindrical branch of $\tT_{Y, \rho}$. 

Pick a monotone staircase curve $\alpha$ on $A_X$, such that 
\begin{enumerate}
\item $\alpha$ connects different boundary components of $A_X$, and its endpoints are on horizontal edges (\Cref{fStairCaseCurveOnSpiralCylinder}), 
\item the restriction of $[V_{Y, \rho}]_X$  to $A_X$ has a leaf disjoint from $A_X$.
\end{enumerate}
\begin{figure}
\begin{overpic}[scale=.15
] {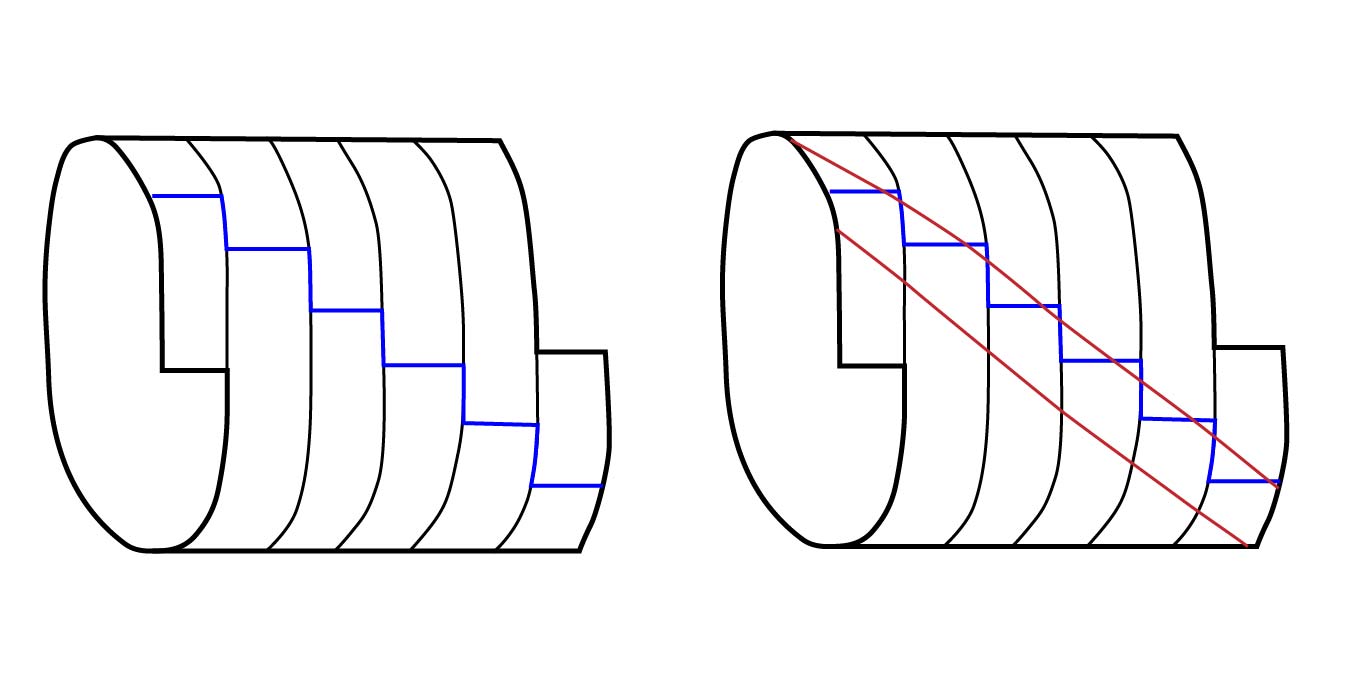} 
 \put(23.5 ,23 ){\textcolor{blue}{$\alpha$}}  
 \put(95 ,  12 ){\textcolor{red}{\small $[V_Y]_X$}}  
      \end{overpic}
\caption{ }\label{fStairCaseCurveOnSpiralCylinder}
\end{figure}
Then $[V_{Y, \rho}]_X$ is essentially carried by $A_X$.
Then, one can easily show that the choice of $\alpha$ is unique through an isotopy preserving the properties. 
\begin{lemma}\Label{IsotopicStaircaseCurves}
Suppose there are two staircase curves $\alpha_1, \alpha_2$ on $A_X$ satisfying Conditions (1) and (2).
Then, $\alpha_1$ and $\alpha_2$ are isotopic through staircase curves $\alpha_t$ satisfying Conditions (1) and (2).
\end{lemma}
In Proposition \ref{HomotopyToBeCarried}, 
pick a realization of $[V_Y]_X$ on the decomposition $(A_X, \alpha_X)$ by a homotopy of $[V_Y]_X$ sweeping out triangles.
This induces an $\ep$-almost staircase curve $\alpha_Y$.  
Similarly to Lemma \ref{EquivariantCircles}, pick a system of round circles $\cc = \{c_h\}$ corresponding to horizontal edges $h$ of $\alpha_X$ so that the $\Ep_{X, \rho}$-images of vertical tangent vectors along $h$ are $\ep$-close to a single vector orthogonal to the hyperbolic plane bounded by $c_h$. 

Then, as in \S\ref{sProjectiveTraintrackForX} we can accordingly isotope the curve $\alpha_Y$ so that the horizontal edges are supported on their corresponding circle of $\cc$ and vertical edges remain vertical---
let $\alpha_Y^\cc$ denote the curve after this isotopy. 
Then $\AAA_Y \minus \alpha_Y^\cc$ is a circular projective structure on a staircase polygon in $\mathbb{E}^2$. 

Then, (similarly to Theorem \ref{CircularTraintrackX}), we can isotope $\alpha_X$ to an $\ep$-almost circular staircase curve $\alpha^\cc_X$ so that 
\begin{itemize}
\item $\alpha_X$ is $2\pi$-Hausdorff close to $\alpha_X^\cc$; 
\item the horizontal edge $h$ of $\alpha_X^\cc$ is  supported on $c_h$; 
\item $\AAA_X \minus \alpha_X^\alpha$ is an $\ep$-almost circular staircase polygon compatible with $\AAA_Y \minus \alpha_Y^\cc$. 
\end{itemize}

As in \S \ref{sGraphsOnPlanarPolygons}, $\AAA_X \minus \alpha_X^\cc$ and $\AAA_Y \minus \alpha_Y^\cc$ yield a $\Z$-valued weighted train track $\Gamma_{A \minus \alpha_X}$ on the polygon $A \minus \alpha_X$ such that $\Gamma_{A_X \minus \alpha_X}$ is transversal to the horizontal foliation. 
Up to a homotopy preserving endpoints on the horizontal edges,  the endpoints of $\Gamma_{A_X \minus \alpha_X}$ match up along $\alpha_X$ as $\Z$-weighted arcs. 
Thus, we obtain a weighted train-track graph $\Gamma_{A_X}$ on $A_X$. 

Consider the subset of the boundary of the circular cylinder $\AAA_X$ which is the union of the vertical boundary edges and the vertices of $\TTT_{X, \rho}$ contained in $\bdr \AAA_X$.
Let $\AA$ be the homotopy class of arcs in $\AAA_X$ connecting different points in this subset. 
Then $[\Gamma_{A_X}] \col \AA \to \Z$ be the map which takes an arc to its total signed intersection number with $\Gamma_{A_X}$.

Then Lemma \ref{IsotopicStaircaseCurves} gives a uniqueness of $[\Gamma_{A_X}]$: 
\begin{proposition}
$[\Gamma_{A_X}] \col \AA \to \Z$ is independent on the choice of the staircase curve $\alpha$ and the realization of $[V_Y]_X$ on $(\AAA_X, \alpha_X)$.
\end{proposition}

\subsection{Weighted train tracks  and $\CP^1$-structures with the same holonomy}\Label{sWeightedTraintarck}
In this section, we suppose that Riemann surfaces $X, Y$ have the same orientation. 
Let $\CC$ be the set of the homotopy classes of closed curves on $S$ (which are not necessarily simple). 
Given a weighted train-track graph immersed on $S$, it gives a cocycle taking  $\gamma \in \CC$  to its weighted intersection number with the graph.
\begin{theorem}\Label{WeightedTraintrack}
For all distinct $X, Y \in \TT$,
there is a bounded subset $K$ in $\rchi_X \cap \rchi_Y$, such that 
\begin{enumerate}
\item for each $\rho \in \rchi_X \cap \rchi_Y \minus K$, the semi-compatible train-track decompositions $\TTT_{X, \rho}$ of $C_{X, \rho}$ and $\TTT_{Y, \rho}$ of $C_{Y, \rho}$  in Theorem \ref{CircularTraintrackX} yield a $\Z$-weighted train track graph $\Gamma_\rho$ carried by $\TTT_{X, \rho}$ (immersed in $S$); \Label{iWeightedTraintrack}
\item the grafting cocycle $[\Gamma_\rho] \col \CC \to \Z$ is independent on the choices for the construction of $\TTT_{X, \rho}$ and $\TTT_{Y, \rho}$;  \Label{iWellDefinedCocycle}
\item $[\Gamma_\rho] \col\CC \to \Z $ is continuous in $\rho \in \rchi_X \cap \rchi_Y \minus K$.\Label{iContinuousTraintrack}
\end{enumerate}

\end{theorem}

Since $[\Gamma_\rho]$ takes values in $\Z$, the continuity immediately implies the following.
\begin{corollary}\Label{ConstantIntersectionNumber}
For sufficiently large boundary subset $K$ of $\rchi_X \cap \rchi_Y$,
 $[\Gamma_\rho]$ is well-defined and constant on each connected component of $\rchi_X \cap \rchi_Y \minus K$. 
\end{corollary}

We first construct a weighted train-track in (\ref{iWeightedTraintrack}). 
Let $h_{X, 1} \dots h_{X, n}$ be the  horizontal edges of branches of  $\TTT_{X, \rho}$. 

Since  $\TTT_{X, \rho}$ and $\TTT_{Y, \rho}$ are semi-compatible (Proposition \ref{CircularTraintrackX}), for each $i = 1, 2, \dots, n$, letting $h_{Y, i}$ be its corresponding edge of a branch of $\TTT_{Y, \rho}$ or a vertex of $\TTT_{Y, \rho}$. 
Then, $h_{X, i}$ and $h_{Y, i}$ develop into the same round circle on $\CP^1$, and their corresponding endpoints map to the same point by the semi-compatibility. 
Thus we have the degree of $h_{X, i}$ relative to $h_{Y, i}$ taking a value in $\Z$ (Definition \ref{dRelativeDegree}) for each horizontal edge:
$$\gamma_\rho \col \{h_{X, 1}, \dots, h_{X, n}\} \to \Z.$$

We will construct a $\Z$-weighted train track $\Gamma_\rho$  carried by $\TTT_{X, \rho}$ so that the intersection number with $h_{X, i}$ is $\gamma_\rho(h_{X, i})$.
The  train-track graph $\Gamma_\rho$ will be constructed on each branch of $\TTT_{X, \rho}$:
\begin{itemize}
\item For each rectangular branch $R$ of $\TTT_{X, \rho}$, we will construct  a $\Z$-train track graph embedded in $R$ (Proposition \ref{IntersectionNumberAndFoliationOnRectangles}). 
\item For each cylinder  $A$ of $\TTT_{X, \rho}$, we have obtained a $\Z$-weighted train track graph embedded in $A$  (\S \ref{sTrainTracksOnCylinders}).
\item For each transversal branch,  we will construct a $\Z$-weighted train track graph embedded in the branch (Proposition \ref{IntersectionNumberAndFoliationOnPolygons}).
\item For each non-transversal branch of $\TTT_{X, \rho}$,  we will construct a $\Z$-weighted train track immersed in the branch (Lemma \ref{WeaklyCompatibleTraintrack}).
\end{itemize}
\subsubsection{Train tracks for rectangular branches}
Given an ordered pair of compatible $\CP^1$-structures on a rectangle,  Lemma \ref{GraftingRectangles} gives a $\Z$-weighted arc connecting the horizontal edges of the rectangle. 
Since the horizontal edges of a rectangular branch of $\TTT_{X, \rho}$ may contain a vertex, we transform the weighted arc to a  weighted train-track graph so that it matches with $\gamma_\rho$.  
\begin{proposition}\Label{IntersectionNumberAndFoliationOnRectangles}
For every $\ep > 0$, there is a bounded subset $K$ of $\rchi_X \cap \rchi_Y$, such that, for every $\rho \in \rchi_X \cap \rchi_Y \minus K$, 
for each rectangular branch $\RRR_X$ of $\TTT_{X, \rho}$, there is a $\Z$-weighted train track graph $\Gamma_{\rho, \RRR_X}$ embedded in $\RRR_X$ satisfying the following:
\begin{itemize}
\item $\Gamma_{\RRR_X}$ induces $\gamma_\rho$; 
\item $\Gamma_{\RRR_X}$ is transversal to the horizontal foliation on ${\RRR_X}$;
\item each horizontal edge of $\TTT_{X, \rho}$ in $\bdr \RRR_X$ contains, at most, one endpoint of $\Gamma_{\rho, \RRR_X}$;
\item As cocycles,  $\Gamma_R$ is $(1 + \ep, \ep)$-quasi-isometric to $(V_{X, \rho} | R_X) - (V_{Y, \rho} | R_Y)$, where $R_X$ is a branch of $\tT_{X, \rho}$ corresponding to $\RRR_X$ and $R_Y$ is the branch of $\tT_{Y, \rho}'$ corresponding to $R_X$ .
\end{itemize}
\begin{figure}
\begin{overpic}[scale=.05,
] {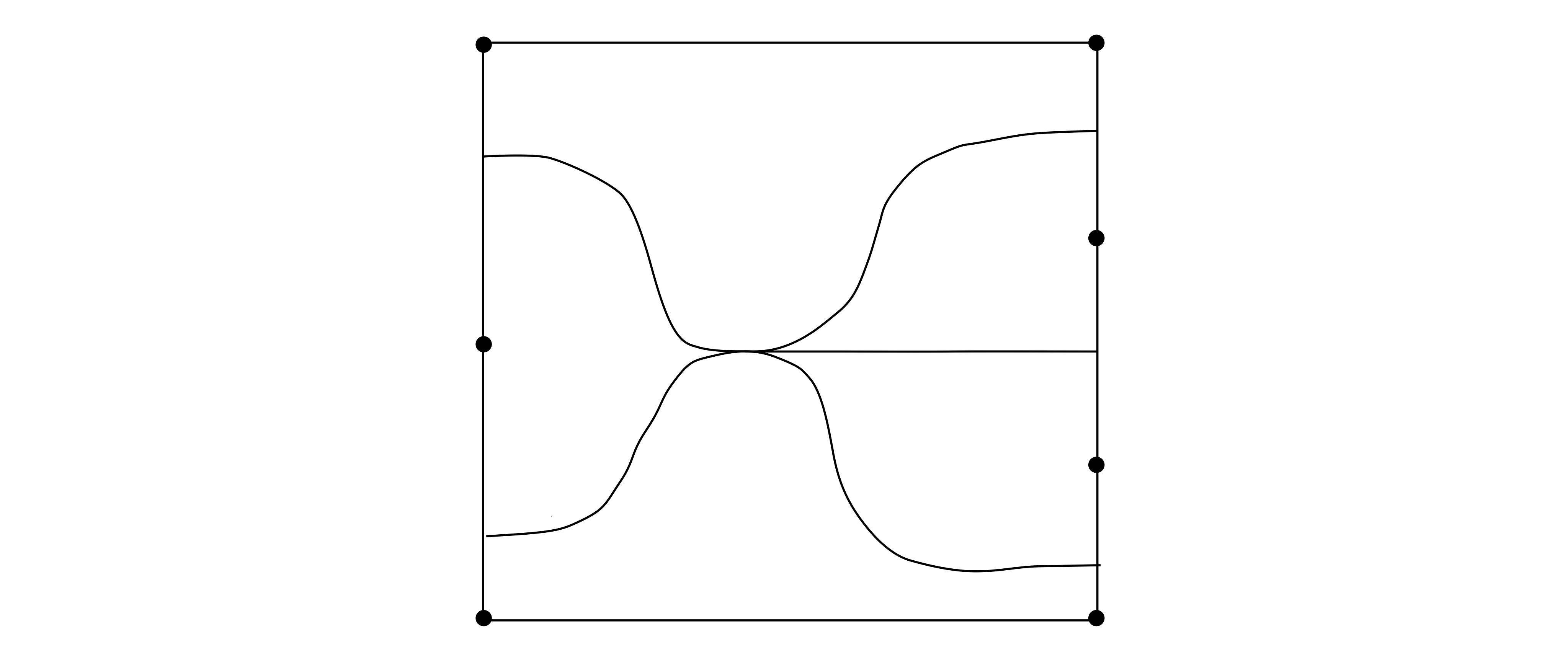} 
  \put(33 ,12 ){$n_4$}
  \put(33 ,33 ){$n_5$}
  \put(62 ,35 ){$n_3$}  
  \put( 62, 21){$n_2$}  
    \put( 62, 8 ){$n_1$}  
   \put( 45, 21 ){$n$}  
      \end{overpic}
\caption{A $\Z$-weight train-track graph for a rectangle. Here $n = n_1 + n_2 + n _3 = n_4 + n_5.$ The black dots are vertices of $\TTT_{X, \rho}$.}\label{fTraintrackOnRectangle}
\end{figure}
\end{proposition}
\begin{proof}
Since $\RRR_X$ and $\RRR_Y$ share their support, let $n = \deg (\RRR_X, \RRR_Y)$ as seem  in Lemma \ref{GraftingRectangles}. 
Let $h_X$ and $h_Y$ be corresponding horizontal edges of $\RRR_X$ and $\RRR_Y$.
Then $h_X = h_{X, 1} \cup \dots \cup h_{X, m}$ be  the decomposition of $h_X$ into  horizontal edges of $\TTT_X$, and let $h_Y = h_{Y, 1} \cup \dots \cup h_{Y, m}$ be the corresponding decomposition into horizontal edges and vertices of $\TTT_X$ compatible with the semi-diffeomorphism $\TTT_{X, \rho} \to \TTT_{Y, \rho}$.
Let $n_i \in \Z$ be $\deg (h_{X,i }, h_{Y, i} )$.
Then, by Lemma \ref{RelativeDegreeSubdivision}, $n = n_1 + \dots + n_m$.
Then it is easy to construct a desired $\Z$-weighted train track realizing such decomposition for both pairs of corresponding horizontal edges (see Figure \ref{fTraintrackOnRectangle}).

The last assertion follows from Lemma \ref{ThurstonLaminationAndVerticalFoliationOnDisks} and Theorem \ref{PleatedSurfacesAreClose}.
\end{proof}

\subsubsection{Train tracks for cylinders}
For each cylindrical branch $\AAA_X$ of $\TTT_{X, \rho}$, in \S\ref{sTrainTracksOnCylinders}, we have constructed a train-track graph $\Gamma_{\rho, \AAA_X}$ on $\AAA_X$, representing the difference between $\AAA_X$ and its corresponding cylindrical branch $\AAA_Y$ of $\TTT_{Y, \rho}$.
\begin{proposition}\Label{FoalitionAndLaminationOnCylinders} 
For every $\ep > 0$, if a bounded subset $K$ of $\rchi_X \cap \rchi_Y$ is sufficiently large, then, 
for each cylindrical branch $\AAA_X$ of $\TTT_{X, \rho}$, 
 the induced cocycle $[\Gamma_{\AAA_X}]\col \AA \to \Z$ times $2\pi$ is $(1+ \ep, \ep)$-quasi-isometric to $V_{Y,\rho} | A_X - V_{X, \rho} | A_Y$, where $A_X$ and $A_Y$ are the corresponding cylindrical branches of $\tT_{X, \rho}$ and $\tT_{Y, \rho}'$.
\end{proposition}
\begin{proof}
Recall that $\Gamma_{\AAA_X}$ is obtained from $\Z$-weighted train-track graphs on the rectangles. 
There is a uniform upper bound,  which depends only on $S$, for the number of the rectangles used to define $[\Gamma_{\AAA_X}]$, since the decomposition was along horizontal arcs starting from singular points.
Then, on each rectangle,  the weighted graph is $(1 + \ep, \ep)$-quasi-isometric to the difference of $V_{X, \rho}$ and $V_{Y, \rho}$ by Proposition \ref{IntersectionNumberAndFoliationOnRectangles}.
Thus $[\Gamma_{\AAA_X}]$ is also $(1 + \ep, \ep)$-quasi-isometric to  the difference of $V_{X, \rho}$ and $V_{Y, \rho}$ if $K$ is sufficiently large. 
\end{proof}

\subsubsection{Train-track graphs for transversal polygonal branches}\Label{sTrainTracksForNontransversalBranches}
Recall that all transversal branches are polygonal or cylindrical (\Cref{TransversalBranches}), i.e. their Euler characteristics are non-negative.
\begin{proposition}\Label{IntersectionNumberAndFoliationOnPolygons}
For every $\ep > 0$, there is a bounded subset $K$ in $\rchi_X \cap \rchi_Y$ such that, 
for each transversal polygonal branch $\PPP_X$ of $\TTT_{X, \rho}$, there is a $\Z$-weighted train track $\Gamma_{\rho, \PPP_X}$ embedded in $\PPP_X$, letting  $P_X$ and $P_Y$ be the branches of $\tT_{X, \rho}$ and $\tT_{Y, \rho}$, respectively, corresponding to $\PPP_X$, respectively, such that 
\begin{enumerate}
\item each horizontal edge $h$ of $\PPP_X$ contains, at most,  one endpoint of $\Gamma_{\rho, P}$; \Label{iOneEndpointOnOneEdge}
\item $[\Gamma_{\PPP_X}]$ agrees with  $\gamma_\rho$ on the horizontal edges of $\TTT_{X, \rho}$ contained in $\bdr \PPP_X$; \Label{iExtensionOfPolygonBoundary}
\item $\Gamma_P$ is transversal to the horizontal foliation $H_{X, \rho}$ on $P_X$; \Label{iTansversalGraph} 
\item $2\pi[\Gamma_P]$ is $(1 + \ep, \ep)$-quasiisometric to  $(V_{X, \rho} | P_X - V_{Y, \rho} | P_Y)$. \Label{iApproximateWeightsOnPolygons}
\end{enumerate}
\end{proposition}
For every $\ep > 0$, if $K$ is sufficiently large, then, by Theorem \ref{EuclideanAndProjectivePolygons} (\ref{CloseToModelCircularPolygon}),
let $\hat\QQQ_X$ be an ideal circular polygon whose truncation ${\QQQ}_X$  is $\ep$-close to $\PPP_X$ in $C_{X, \rho}$.
Similarly,  $\hat\QQQ_Y$ be an ideal circular polygon whose truncation $\QQQ_Y$ is $\ep$-close to $\PPP_Y$.
Since $\Supp \PPP_X = \Supp \PPP_Y$ as circular polygons, we may in addition assume that $\Supp \QQQ_X = \Supp \QQQ_Y$ as truncated idea polygons. 

Let $\ol{\QQQ}_X$ be the canonical polynomial $\CP^1$-structure on $\C$ which contains $\hat\QQQ_X$ (\S \ref{sModelCircularPolygons}). 
Let $\hat\LLL_{X}$ be the restriction of the Thurston lamination of $\ol\QQQ_X$ to $\hat\QQQ_X$.
Similarly, 
let $\ol{\QQQ}_Y$ be the canonical polynomial $\CP^1$-structure on $\C$ which contains $\QQQ_Y$. 
Let $\hat\LLL_Y$ be the restriction of the Thurston lamination of $\ol\QQQ_Y$ to $\hat\QQQ_Y$.
As $\Supp \QQQ_X = \Supp \QQQ_Y$, thus $\ol{\QQQ}_X$ and $\ol{\QQQ}_Y$  share their ideal vertices.
Then $\QQQ_X$ and $\QQQ_Y$ are $\ep$-close to $\PPP_X$ and $\PPP_Y$, respectively. 
Thus, since $\gamma_\rho$ takes values in $\Z$,  $\hat\LLL_{X} - \hat\LLL_Y$ satisfies (\ref{iExtensionOfPolygonBoundary}).

Theorem \ref{EuclideanAndProjectivePolygons} (\ref{iThurstonLaminationAndVerticalFoliation}) implies that $\hat\LLL_{X}$ is $(1 + \ep, \ep)$-quasi-isometric to $V_X | P_X$ and $\hat\LLL_{Y}$ is $(1 + \ep, \ep)$-quasi-isometric to $V_Y | P_Y$,  and therefore $\hat\LLL_{X} - \hat\LLL_Y$ satisfies  (\ref{iApproximateWeightsOnPolygons}). 

(\ref{iOneEndpointOnOneEdge}) is easy to be realized by homotopy combining the edges of the train-track graph ending on the same horizontal edge.
We show that 
there is a $\Z$-weighted train track $\Gamma$ which is $\ep$-close to $\hat\LLL_X - \hat\LLL_Y$, satisfying (\ref{iTansversalGraph}).

Let $\Gamma_X^P$ be a weighted train track graph on $P$ which represents $\hat\LLL_X$a
Let $\Gamma_Y^P$ be the weighted train-track graph which represents  $\hat\LLL_Y$.
By Theorem \ref{PleatedSurfacesAreClose}, the pleated surface of $\PPP_X$ is $\ep$-close to the pleated surface of $\PPP_Y$,
because of the quasi-parallelism in Proposition \ref{AlmostParalellLaiminationAndFoliation}.
Let $\ch\Gamma_X^P$ be the subgraph of $|\Gamma_X^P|$ obtained by eliminating the edges of weights less than a sufficiently small $\ep$. 
Similarly, let $\ch\Gamma_Y^P$ be the subgraph of $\Gamma_Y^P$ obtained by eliminating the edges of weight less than $\ep$. 

Then, there is a minimal train-track graph $\Gamma^P$ containing both $\ch\Gamma_Y^P$ and $\ch\Gamma_X^P$ and satisfying (\ref{iOneEndpointOnOneEdge}). 
Since the pleated surfaces are sufficiently close,  by approximating the weights of $\Gamma_X^P -  \Gamma_Y^P$  by integers, we obtain a desired $\Z$-weighted train-track graph supported on $\Gamma^P$.

Since $\PPP_X$ is a transversal branch,  $\Gamma_X^P$ and $\Gamma_Y^P$ are both transversal to $H_{X, \rho} | P$,  thus $\Gamma^P$ is transversal to  $H_{X, \rho} | P$ (\ref{iTansversalGraph}).
\Qed{IntersectionNumberAndFoliationOnPolygons}

\subsubsection{Train-track graphs for non-transversal branches} 
Let $P_X$ and $P_Y$ be the corresponding  branches of $\tT_{X, \rho}$ and $\tT_{Y, \rho}$, respectively, which are non-transversal. 
Let $\PPP_X$ and $\PPP_Y$ be the branches of $\TTT_{X, \rho}$ and $\TTT_{Y, \rho}$ corresponding to $P_X$ and $P_Y$, respectively.
Then, by  Theorem \ref{ProjectiveTraintrackForY}, $\hat\beta_{X, \rho} | \bdr \PPP_X$ is $\ep$-close to  $\hat\beta_{Y, \rho} | \bdr \PPP_Y$ in the $C^0$-metric and $C^1$-close along the vertical edges. 
Let $\sigma_X$ be a pleated surface with crown-shaped boundary whose truncation approximates $\hat\beta_X |  \PPP_X$, and  let $\sigma_Y$ be the pleated surfaced boundary with crown-shaped boundary whose truncation approximates $\hat\beta_Y |  \PPP_Y$, such that $\sigma_X$ and $\sigma_Y$ share their boundary (in $\H^3$).
In particular, the base hyperbolic surfaces for $\sigma_X$ and $\sigma_Y$ are diffeomorphic preserving markings and spikes. 

Let $\nu_X$ and $\nu_Y$ be the bending measured laminations for  $\sigma_X$ and $\sigma_Y$, respectively; then $\nu_X$ and $\nu_Y$ contain only finitely many leaves whose connect ideal points. 

In Thurston coordinates, the developing map and the pleated surface of a $\CP^1$-structure are related by the nearest point projection to the supporting planes of the pleated surface (\cite{Kullkani-Pinkall-94, Baba20ThurstonParameter}).
For small $\upsilon > 0$, 
let $F_X$, $F_Y$ be the surfaces in $\H^3$ which are at  distance $\upsilon$ from   $\sigma_X$ and $\sigma_Y$ in the direction of the nearest point projections of  $P_X$ and $P_Y$ (\cite[Chapter II.2]{Epstein-Marden-87} ). 

Thus, if necessary, refining $\nu_X$ to $\nu_Y$ to ideal triangulations of the hyperbolic surfaces appropriately,  pick an (irreducible) sequence of flips $w_i$  which connects $\nu_X$ to $\nu_Y$. 
Clearly the sequence $w_i$ corresponds to a sequence of triangulations.

\begin{lemma}
If a bounded subset $K \sub \rchi$ is sufficiently large, for every $\rho \in \rchi_X \cap \rchi_Y \minus K$ and all non-transversal branches $P_X$ and $P_Y$,   there is a uniform upper bound on the length of the flip sequence 
which depends only on $X, Y \in \TT$, or appropriate refinements into triangulations. 
\end{lemma}
\begin{proof}
This follows from the length bound in Proposition \ref{DependentSegmentLengthBound}.
\end{proof}
A triangulation in the sequence given by $w_i$ is {\sf realizable}, if there is an equivariant pleated surface homotopic to $\sigma_X$ (and $\sigma_Y$) relative to the boundary such that the pleating locus agrees with the triangulation.  
In general,  a triangulation in the sequence is not realizable when the endpoints of edges develop to the same point on $\CP^1$. 
However a generic perturbation makes the triangulation realizable:

\begin{lemma}\Label{RealizingFlipSequence}
For almost every perturbation of the holonomy of $P_X$ and holonomy equivariant perturbation of the (ideal) vertices of $\ti{\sigma}_X$ (and $\ti{\sigma}_Y$) in $\CP^1$,  all triangulations in the flip sequence $w_i$  are realizable. 
Moreover, the set of realizable perturbations is connected.
\end{lemma}
\begin{proof}
If $\sigma_X$ is an ideal polygon, the holonomy is trivial. 
Then, since there are only finitely many vertices and $\CP^1$ has real dimension two,  almost every perturbation is realizable. 

If $\sigma_X$ is not a polygon, 
an edge of a triangulation forms a loop if the endpoints are at the same spike of  $\tau_X$. 
For each loop $\ell$ of $\tau_X$,  the condition that the holonomy of $\ell$ is the identity is a complex codimension, at least, one in the character variety  (and also in the representation variety).
Since the flip sequence is finite, for almost all perturbations of the holonomy, if an edge of a triangulation in the sequence forms a loop, then its holonomy is non-trivial. 
Clearly, such a perturbation is connected. 
Then, for every such perturbation of the holonomy, it is easy to see that, for almost all equivariant perturbations of the ideal points, the triangulations in the sequence are realizable.
 \end{proof}
For every perturbation of the holonomy and the ideal vertices given by Lemma \ref{RealizingFlipSequence}, the flip sequence $w_i$ gives
 the sequence of pleated surfaces $\sigma_X = \sigma_1 \xrightarrow{w_1} \sigma_2  \xrightarrow{w_2} \dots  \xrightarrow{w_{n-1}} \sigma_n = \sigma_Y$ in  $\H^3$ connecting $\sigma_Y$ to $\sigma_X$, such that $\sigma_i$'s share their boundary geodesics and ideal vertices. 

\begin{figure}
\begin{overpic}[scale=.1
] {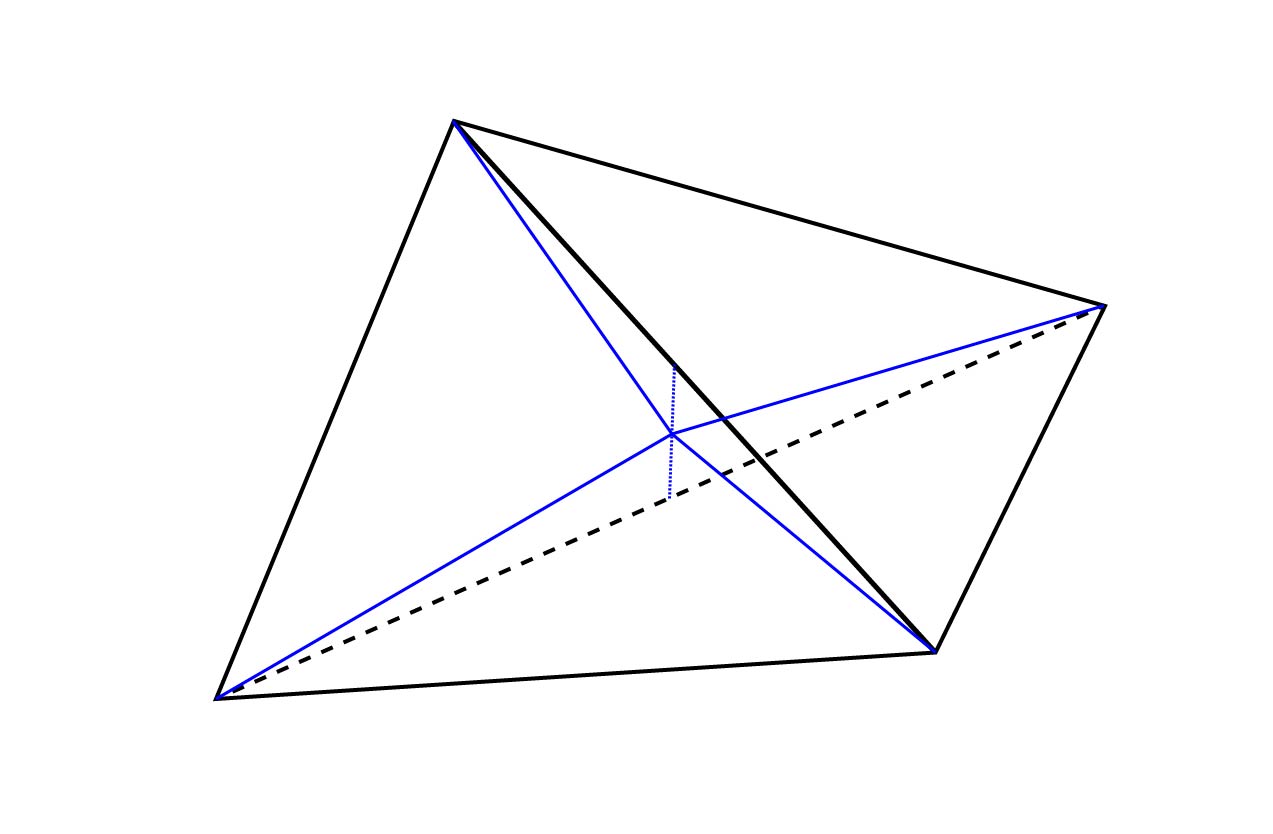} 
      \end{overpic}
\caption{The blue lines are the pleating lamination interpolating the pleated surfaces which differ by a flip.}\label{fHomotopyRealizingFlip}
\end{figure}

For each flip $w_i$, the pairs of triangles of the adjacent pleated surfaces $\sigma_i$ and $\sigma_{i + 1}$ bound a tetrahedron in $\H^3$.
To be precise, if the four vertices are contained in a plane, the tetrahedron is collapsed into a quadrangle, but it does not affect the following argument.
The edges exchanged by the flip correspond to the opposite edges of the tetrahedron. 
 Then pick a geodesic segment connecting those opposite edges.  
Then there is a path $\sigma_t ~( i \leq t \leq i + 1)$ of pleated surfaces  with a single cone point of angle more than $2\pi$ such that
\begin{itemize}
\item $\sigma_t$ connects  $\sigma_i$ to $\sigma_{i + 1}$;
\item the pleated surfaces $\sigma$ share their quadrangular boundary, which corresponds to the ideal quadrangle supporting the flip $w_i$;
\item by the homotopy, $\sigma_t$ sweeps out the tetrahedron; 
\item the cone point on the geodesic segment  (see Figure \ref{fHomotopyRealizingFlip}).
\end{itemize}

In this manner, this sequence of pleated surfaces $\sigma_i$ continuously extends to a homotopy of the pleated surfaces with, at most, one singular point of cone angle greater than $2\pi$. 
This interpolation also connects a bending cocycle on $\sigma_i$ to a bending cocycle on $\sigma_{i + 1}$ continuous, although the induced cocycle on $\sigma_{i + 1}$ may correspond to a measured lamination only immersed on the surface, since the edges of the triangulations transversally intersect. 
Thus, $\nu_X$ induces a sequence of the bending (immersed) measured laminations $\nu_i$ of $\sigma_i$ supported a union of the pleating loci of $\sigma_1, \dots, \sigma_i$.
 
 For each $i$, the difference $ \nu_{i + 1 } - \nu_{i}$ of the transverse cocycles is supported on the geodesics corresponding to the edges of the tetrahedron,  so that, on the surface, the edges form an ideal rectangle with both diagonals. 
Let $\mu_i$ be the difference cocycle $ \nu_{i + 1 } - \nu_{i}$. 
From each vertex of the ideal rectangle, there are three leaves of  $ \nu_{i + 1 } - \nu_{i}$ starting, and the sum of their weights is zero by  Euclidean geometry (Figure \ref{fFlipAndBendingAngles}).
\begin{figure}
\begin{overpic}[scale=.05,
] {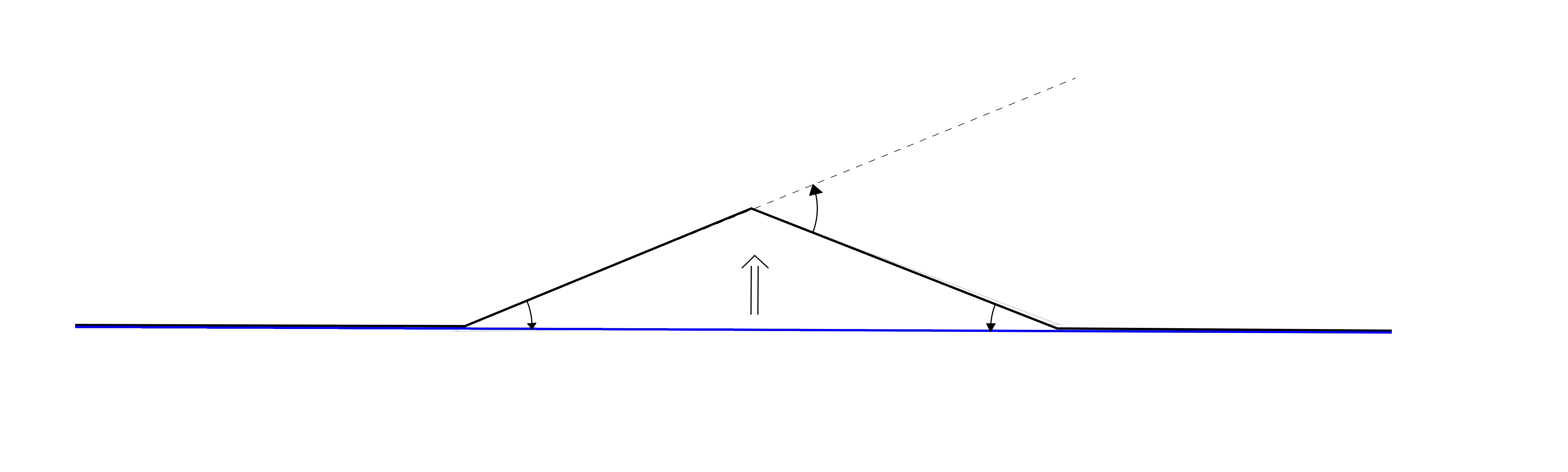} 
  \put(45, 6){\textcolor{blue}{$\sigma_i$}}  
  \put(35, 16){\textcolor{black}{$\sigma_{i + 1}$}}  
      \end{overpic}
\caption{The link of a vertex of the ideal rectangle ----- the sum of the indicated singed angles is zero.}\label{fFlipAndBendingAngles}
\end{figure}
Note that $P_X$ can be identified with $\sigma_X$ by collapsing each horizontal edge of $P_X$ to a point.
Hence, for every $i$, if $\alpha$ is a closed curve on $P_X$ or an arc connecting  vertical edges of $P_X$, then $\mu_i (\alpha) = 0$.
By regarding $\nu_j$ is a geodesic lamination on $\sigma_X$, their union 
 $\cup_{j = 1}^i \nu_j$ is a graph on $\sigma_X$ whose vertices are the transversal intersection points of the triangulations. 
 A small regular neighborhood $N$ of $\cup_{j = 1}^i \nu_j$ is decomposed into a small regular neighborhood $N_0$ of the vertices and a small regular neighborhood of the edges minus $N_0$ in $N \minus N_0$.  

Since, after the Whitehead moves, the pleated surface $\sigma_X$ is transformed into a pleated surface $\sigma_Y$. 
Thus $\nu_n - \nu_Y$ gives a $\Z$-valued transversal cocycle.

By the construction of the regular homotopy, we have the following.
\begin{proposition}[Train tracks for non-transversal branches]\Label{WeaklyCompatibleTraintrack}
For every non-transversally compatible branches $\PPP_X$ of $\TTT_{X, \rho}$ and $\PPP_Y$ of $\TTT_{Y, \rho}$,
there is a $\Z$-weighted immersed train-track graph $\Gamma_{P_X}$ representing the transversal cocycle supported on $\cup_{j = 1}^i \nu_j$. 
Moreover, the train-track cocycle is independent of the choice of the flip sequence $w_i$.
\end{proposition}
\begin{proof}
Given two flip sequences $(w_i), (w_j')$ connecting the triangulations of $\sigma_X$ to $\sigma_Y$, there are connected by a sequence of sequences $(v_i^k)$ of triangulations connecting $\sigma_X$ to $\sigma_Y$, such that $(v_i^k)$ and $(v_i^{k+1})$ differ by either an involutivity,  a commutativity or a pentagon relation (\cite[Chapter 5, Corollary 1.2]{Penner12TeichmullerTheory}). 
Clearly, the difference by an involutivity and a  commutativity do not affect the resulting cocycle.
Also by the pentagon relation, the pleated surface does not change including the bending measure since each flip preserves the total bending along the vertices. 
Therefore  $(v_i^k)$ and $(v_i^{k+1})$ give the same train-track cocycle. 
\end{proof}
Therefore we obtain $\AA  \to \Z$.
By continuity and the connectedness in Lemma \ref{RealizingFlipSequence}.  we have the following. 

\begin{corollary}
$[\Gamma_P]$ is independent on the choice of the perturbation in Lemma \ref{RealizingFlipSequence}. 
\end{corollary}
There are only finitely many combinatorial types of the train-tracks $\tT_{X, \rho}$. 
We say that a branch $B_X$ of $\tT_{X,\rho}$ and a branch $B_X'$ of $\tT_{X,\rho'}$  are {\sf isotopic} if they are diffeomorphic and isotopic on $S$. 
Then there are only finitely many isotopy classes of branches of $\tT_{X, \rho}$ for all $\rho \in \rchi_X \minus K$.
Let $\alpha$ be an arc $\alpha$ on a branch $B_X$, such that each endpoint of $\alpha$ is at either on a vertical edge or a vertex of $\tT_{X, \rho}$.
\begin{proposition}\Label{TraintrackEstimateNontransversalBranch}
Let $X, Y \in \TT$. 
 For every $\ep > 0$,  there  a compact subset $K$ in $\rchi$ with the following property: 
 For every pair $(B_X, \alpha)$ of an isotopy class of a branch $B_X$ and an arc  $\alpha$ as above, there is a constant $k_\alpha > 0$ such that,  if $B_X$ is a (non-transversal) branch of $\tT_{X, \rho}$ for some $\rho \in \rchi_X \cap \rchi_Y \minus K$, then,   $2\pi[\Gam_{P_X}] (\alpha)$ is $(1 + \ep, k_\alpha)$-quasi-isometric to  $\sqrt{2}(V_X | P_X - V_Y | P_Y) (\alpha)$.
\end{proposition}
\begin{proof}
Since the length of the flip sequence is bounded from above,  the difference between $\nu_X$ and $\nu_n$ is uniformly bounded in the space of transverse cocycles on $S$. 
Then the assertion follows.
\end{proof}

\subsubsection{Independency of the transverse cocycle}
From the train-track decompositions $\TTT_{X, \rho}$ and $\TTT_{Y, \rho}$, we have constructed a weighted train-track graph $\Gamma_\rho$ (\Cref{WeightedTraintrack} (\ref{iWeightedTraintrack})). 
Next we show its cocycle is independent on  the train-track decompositions $\TTT_{X, \rho}$ and $\TTT_{Y, \rho}$ ( \Cref{WeightedTraintrack}(\ref{iWellDefinedCocycle})). 

Recall that the train-track decompositions $\TTT_{X, \rho}$ and $\TTT_{Y, \rho}$ are determined by
\begin{enumerate}
\item the holonomy equivariant circle system $\cc = \{c_h\}$ indexed by horizontal edges $h$ of $\ti\tT_{X, \rho}$ (given by Lemma \ref{EquivariantCircles}), \Label{iChoiceOfCircleSystmes}
\item the realization $W_Y$ of  $[V_{Y,\rho}]_{X, \rho}$ on $\tT_{X, \rho}$ (Lemma \ref{HomotopyToBeCarried}), and \Label{iChoiceOfRealization}
\item  the choice of vertical edges of $\TTT_{X, \rho}$ (Theorem \ref{CircularTraintrackX} (\ref{iClosenessInHorizontalAndVerticalDirections})).  \Label{iChoiceOfVerticalEdges} 
\end{enumerate}

\begin{proposition}
The cocycle $[\Gamma_\rho]\col\CC \to \Z$ constructed above is independent of the construction for  $\TTT_{X, \rho}$ and $\TTT_{Y, \rho}$,  i.e. (\ref{iChoiceOfCircleSystmes}), (\ref{iChoiceOfRealization}), (\ref{iChoiceOfVerticalEdges}).
\end{proposition}

\begin{proof}
(\ref{iChoiceOfCircleSystmes})
 By Proposition \ref{IsotopeCircleSystem}, given two appropriate circle systems $\{c_h\}$ and $\{c'_h\}$, there is an equivariant isotopy of circles systems $\{c_{t, h}\}$ connecting $\{c_h\}$ to $\{c'_h\}$.
 Then accordingly, we obtain a continuous family of cocycles $[\Gamma_{t, \rho}]\col \CC \to \Z$. 
 As it takes discrete values,  $[\Gamma_{t, \rho}]$  must remain the same. 
 
By the different choices for (\ref{iChoiceOfRealization}) and (\ref{iChoiceOfVerticalEdges}), the $\Z$-weights on $\Gamma$ shift across bigon regions corresponding vertical edges of $T_{Y, \rho}$ by integer values. 
These weight shifts clearly preserve the cocycle $[\Gamma_\rho]\col \CC \to \Z$. 
\end{proof}

\subsubsection{Continuity of the transverse cocycle}
Next we prove the continuity of $[\Gamma_\rho]$ in $\rho$ claimed in (\ref{iContinuousTraintrack}).

\begin{definition}[Convergence and semi-convergence of train tracks]
Suppose that $C_i \in \PP$ converges to $C \in \PP$. 
In addition, suppose, for each $i$, there are a train-track structure   $\TTT_i$  of $C_i$ and a train-track structure $\TTT$ for $C$. 
Then,

\begin{itemize}
\item 
 $\TTT_i$ {\sf converges} to $\TTT$ if 
\begin{itemize}
\item for every branch $\PPP$ of $\TTT$, there is a sequence of branches $\PPP_i$ of $\TTT_i$ converging to $\PPP$, and
\item for every sequence of branches $\PPP_i$ of $\TTT_i$, up to a subsequence, converges to either a branch of $\TTT$ or an edge of a branch of $\TTT$.
\end{itemize}
\item  $\TTT_i$ {\sf semi-converges} to $\TTT$ if there is a subdivision of $\TTT$ into another circular train-track structure $\TTT'$ so that $\TTT_i$ converges to $\TTT'$.
\end{itemize}
\end{definition}

\begin{lemma}\Label{SemiconvergenceProjectiveTraintracks}
Let $\rho_i$ be a sequence in $\rchi_X \cap \rchi_Y$ converging to $\rho \in \rchi_X \cap \rchi_Y \minus K$, where $K$ is a sufficiently large compact (as in \Cref{CircularTraintrackX}). 
Pick an equivariant circle system $\cc_i$ for $\tT_{X, \rho_i}$ by Lemma \ref{EquivariantCircles} which converges to a circle system $\cc$ for $\tT_{X, \rho}$.
Then, up to a subsequence, 
\begin{itemize}
\item  the circular train track $\TTT_{X, \rho_i}$ of $C_{X, \rho_i}$ semi-converges to a circular train track  $\TTT_{X, \rho}$ of $C_{X, \rho}$; 
\item the circular train track  $\TTT_{Y, \rho_i}$ of $C_{Y, \rho_i}$ semi-converges to a circular train track  $\TTT_{Y, \rho}$ of $C_{Y, \rho}$;
\item $\TTT_{X, \rho}$ is semi-compatible with $\TTT_{Y, \rho}$.
\end{itemize}
\end{lemma}

\proof
By \Cref{TforY}, $\tT_{X, \rho_i}$ semi-converges to $\tT_{X, \rho}$.
Therefore $\tT_{X, \rho_i}$ converges to  a subdivision $\tT_{X, \rho}'$ of $\tT_{X, \rho}$ as $i \to \infi$.
Then, if $\tT_{X, \rho} \neq \tT_{X, \rho}'$, then $\tT_{X, \rho}$ is obtained from $\tT_{X, \rho}'$ by gluing non-rectangular branches with rectangular branches of small width or replacing long rectangles into spiral cylinders (as in \S \ref{sPolygonalTrainTrack} and \S\ref{sBoundedEuclideanTraintrackForX}).

Recall that7 the realization of $[V_{Y, \rho_i}]_{X,\rho_i}$ in the train track $\tT_{X, \rho_i}$ is unique up to shifting across vertical edges of non-rectangular branches (Proposition \ref{EuclideanTraintarck} (\ref{iRealzationUpToShifting})). 
Therefore, up to a subsequence,  the realization of $[V_{Y, \rho_i}]_{X,\rho_i}$ on  $\tT_{X, \rho_i}$  converges to a realization of  $[V_{Y, \rho}]_X$ on  $\tT_{X, \rho}'$,
Since $\tT_{X, \rho}'$ is a subdivision of $\tT_{X, \rho}$,   the limit can be regarded as also a realization on $\tT_{X, \rho}$. 
Since the realization determines the train-track structure of $E_{Y, \rho}$, up to a subsequence,  $\tT_{Y, \rho_i}$ converges to a bounded train-track $\tT_{Y, \rho}'$.
Then $\tT_{Y,\rho}$ is transformed to $\tT_{Y, \rho}'$ by possibly sliding vertical edges and subdividing spiral cylinders to wide rectangles.
Moreover, by Theorem \ref{ProjectiveTraintrackForY} (\ref{iQIperturbationOfTraintrack}), $\TTT_{Y, \rho}$ is $(1 + \ep, \ep)$-quasiisometric to $\tT_{Y, \rho}$.
Therefore,  up to a subsequence, $\TTT_{Y, \rho_i}$ converges to a circular train-track structure $\TTT_Y'$ of $C_{Y, \rho}$.
If $\TTT_{Y, \rho}$ is different from $\TTT_{Y, \rho}'$, then $\TTT_{Y, \rho}$ can be transformed to $\TTT_{Y, \rho}'$ by sliding vertical edges and subdividing spiral cylinders into rectangles. 

By Theorem \ref{CircularTraintrackX},  $\TTT_{X, \rho_i}$ is additively $2\pi$-close to $\tT_{X, \rho_i}$ in the Hausdorff metric of $E^1_{X, \rho}$.
Therefore, up to a subsequence  $\TTT_{X, \rho_i}$ converges to a circular train track decomposition $\TTT_{X, \rho}'$ semi-diffeomorphic to $\TTT_{Y, \rho}$.
Moreover  $\TTT_{X, \rho}$ can be transformed to $\TTT_{X, \rho}'$  possibly by subdividing and sliding by $2\pi$ or $4\pi$.  

We have already shown that  $\TTT_{X, \rho}$ is semi-diffeomorphic to $\TTT_{Y, \rho}$ (Theorem \ref{CircularTraintrackX}).
\Qed{SemiconvergenceProjectiveTraintracks}

Finally we have the continuity (\ref{iContinuousTraintrack}).
\begin{corollary}
$[\Gamma_{\rho_i}]\col \CC \to \Z$ converges to $[\Gamma_\rho]\col \CC \to \Z$ as $i \to \infty$.
\end{corollary}
\begin{proof}     
Since $\TTT_{X, \rho_i}$ semi-converges to $\TTT_{X, \rho}$,   
up to taking a subsequence,  $\TTT_{X,i}$ converges to a subdivision $\TTT'_{X, \rho}$ of $\TTT_{X, \rho}$.
Accordingly, there is a subdivision  $\TTT'_{Y, \rho}$  of $\TTT_{Y, \rho}$\,, such that,  to up a subsequence,  $\TTT_{Y,i}$ converges to $\TTT'_{Y, \rho}$ and that  $\TTT'_{X, \rho}$  is semi-diffeomorphic  to $\TTT'_{Y, \rho}$.

Let $\Gamma_{\rho_i}$ be the $\Z$-weighted train-track  given by $\TTT_{X, \rho_i}$ and $\TTT_{Y, \rho_i}$. 
Let $\Gamma_{\rho}'$ be the $\Z$-weighted train track given $\TTT'_{X, \rho}$ and $\TTT'_{Y, \rho}$.
Then, by the convergence of the train tracks, $\Gamma_{\rho_i}$  converges to $\Gamma_{\rho}'$ as $i \to \infty$.
Since  $\TTT'_{X, \rho}$ and  $\TTT'_{Y, \rho}$  are obtained  by  sliding and subdividing $\TTT_{X, \rho}$ and   $\TTT_{Y,  \rho}$ respectively, thus $\Gamma_{\rho}'$ and $\Gamma_{\rho}$ yield the same cocycle $\CC \to \Z$. 
\end{proof}

\subsection{Approximation of the grafting cocycle $[\Gamma_\rho]$ by vertical foliations}
Suppose that $X, Y$ be distinct marked Riemann surfaces homeomorphic to $S$ such that $X$ and $Y$ have the same orientation.
For a branch $B_X$ of $\tT_{X, \rho}$, let $\AA(B_X)$ be the homotopy class of arcs $\alpha$ on $\RRR_X$ such that every endpoint of $\alpha$ is either on a vertical edge or a vertex of $\tT_{X, \rho}.$

\begin{theorem}\Label{VerticalFoliationsAndGraftingFunction}
Let $c_1, \dots, c_n$ be essential closed curves on $S$. 
Then, for every $\ep > 0$, there is a  bounded subset $K_\ep$ of $\rchi_X \cap \rchi_Y$  such that, for every $\rho \in \rchi_X \cap \rchi_Y \minus K$,
the grafting cocycle $[\Gamma_\rho]$ times $2\pi$ is $(1 + \ep, q)$-quasi-isometric to $\sqrt{2} (V_{X, \rho} - V_{Y, \rho})$  along $c_1, \dots, c_n$. That is,
\begin{equation}
(1 - \ep) 2\pi \Gamma_\rho (c_i)  - q < \sqrt{2}(V_{X, \rho}(c_i) - V_{Y, \rho} (c_i) ) <  (1 + \ep) 2\pi \Gamma_\rho (c_i)  + q \Label{QIintersection}
\end{equation}
 for all $i = 1,2, \dots, n$.
\end{theorem}
\begin{proof}

Let $H \in \PML$. 
Recall that $E^1_{X, H}$ is the flat surface conformal to $X$, such that the horizontal foliation is $H$ and $\Area E^1(X, H) = 1$.
Let $\tT_{X, H}$ be the bounded train-track structure of $E_{X, H}$.

Then, every closed curve $c$ can be isotoped to a closed curve $c'$ so that, for each branch of $B$ of $\tT_{X, \rho}$, $c' | B$ is an arc connecting different vertices. 
Let $c'_1, \dots, c'_m$ be the decomposition into sub-arcs. 
By the finiteness of possible train-tracks,  the number $m$ of the subarcs is bounded from above for all $\rho$.
Then  $2\pi\Gamma_\rho | c'_j $ is, if $B$ is a transversal branch,  $(1 + \ep, \ep)$-quasi-isometric to  $\sqrt{2}(V_{X, \rho} | B  - V_{Y, \rho}  | B) c_k'$ by Proposition \ref{IntersectionNumberAndFoliationOnPolygons}(\ref{iApproximateWeightsOnPolygons}), \Cref{FoalitionAndLaminationOnCylinders}, Proposition \ref{IntersectionNumberAndFoliationOnRectangles}, and, if non-transversal,  $(1 + \ep, q)$-quasi-isometric  by Proposition \ref{TraintrackEstimateNontransversalBranch}.
As the number of subarcs is bounded, the assertion follows. 
 \end{proof}

\section{The discreteness}\Label{sDiscreteness}
\subsection{The discreteness of the intersection of holonomy varieties}

\begin{theorem}\Label{Bounded-Intersection}
Suppose that $X, Y$ are marked Riemann surface structures on $S$ with the same orientation.
Then 
every (connected) component  of $\rchi_X \cap  \rchi_Y$ is bounded. 
\end{theorem}
\proof
Let $K$ be a component of $\rchi_X \cap \rchi_Y$. 
Suppose, to the contrary, that $V$ is unbounded in $\rchi$. 
Then, there is a path $\rho_t$ in  $\rchi_X \cap \rchi_Y$  which leaves every compact subset. 
Then, by \Cref{HorizontalMeasuredLaminations}, by taking a diverging sequence $t_1 < t_2 < \dots$,  there are $k_i, k_i' \in \R_{> 0}$ such that $\frac{k_i}{k_i'} \to 1$ as $i \to \infi$ and   $$\lim_{i \to \infi} k_i H_{X, \rho_{t_i}} = \lim_{i \to \infi }k'_i H_{Y, \rho_{t_i}} \in \ML.$$
By taking a subsequence, we may,  in addition, assume that their vertical foliations $[V_{X, \rho_{t_i}}]$ and $[V_{Y, \rho_{t_i}}]$ converge in $\PML$.
Thus let $[V_{X, \infi}]$ and $[V_{Y, \infi}]$ be their respective limits in $\PML$. 
Since $X \neq Y$,  $V_{X, \infi}$  and $V_{Y, \infi}$ can not be asymptotically the same, in comparison to their horizontal foliations.
Then $V_{X, \rho_{t_i}} - V_{Y, \rho_{t_i}}$  ``diverges to $\infty$''. 
That is, there is a closed curve $\alpha$ on $S$, such that  $$ | V_{X, \rho_{t_i}}(\alpha) - V_{Y, \rho_{t_i}} (\alpha) | \to \infi$$
as $i \to \infi$.

Let $[\Gamma_{\rho_t}]\col\CC  \to \Z$ be the function given by Theorem \ref{WeightedTraintrack}.
As $[\Gamma_{\rho_t}]$ is continuous (Theorem \ref{WeightedTraintrack} (\ref{iContinuousTraintrack})),  $[\Gamma_{\rho_t}]\col \CC \to \Z$ is a constant function (for $t \gg 0$).
On the other hand, by Theorem \ref{VerticalFoliationsAndGraftingFunction}, there is $q > 0$  such that 
$\sqrt{2} (V_{X, \rho_{t_i}} - V_{Y, \rho_{t_i}})(\alpha)$  is $(1 + \ep_i, q)$-quasi-isometrically close to  $2\pi [\Gamma_{\rho_{t_i}}] (\alpha)$, and $\ep_i \to 0$ as $i \to \infi$.
 We thus obtain a contradiction. 
\Qed{Bounded-Intersection}

Since $\rchi_X$ and $\rchi_X$ are complex analytic, thus their intersection is also a complex analytic set  (Theorem 5.4 in \cite{Fritzsche_Hans_02_HolomorphicFunactionsComplexManifolds}). 
Therefore, since every bounded connected analytic set is a singleton  (Proposition \ref{BoundedAnalyticSet}), Theorem \ref{Bounded-Intersection} implies the following.
\begin{corollary}\Label{DicreteIntersection}
$\rchi_X \cap \rchi_Y$ is a discrete set. 
\end{corollary}
We will, moreover, show that this intersection is non-empty in \S \ref{sCompleteness}.
\subsection{A weak simultaneous uniformization theorem}
In this section, using Corollary \ref{DicreteIntersection}, we prove a weak version of a simultaneous uniformization theorem for general representations. 
Let $\rho \col \pi_1(S) \to \PSL(2, \C)$ be any non-elementary representation which lifts to ${\rm SL}(2, \C)$.
Let  $C,  D$ be $\CP^1$-structures on  $S^+$ with the  holonomy $\rho$. 
Then, if a neighborhood $U_\rho$ of $\rho$ in $\rchi$ is sufficiently small, then there are (unique) neighborhoods $V_C$ and $W_D$ of $C$ and $D$ in $\PP$, respectively, which are biholomorphic to $U_\rho$  by $\Hol \col \PP \to \rchi$. 
Then, for every $\eta \in U_\rho$, there are unique $\CP^1$-structures $C_\eta$ in $V_C$ and $D_\eta$ in $W_D$ with holonomy $\eta$.  
Let $\Phi_{\rho, U} = \Phi\col U_\rho \to \TT \times \TT$ be the map which takes $\eta  \in U_\rho$ to the pair of the marked Riemann surface structures of $C_\eta$ and $D_\eta$.
\begin{boxedlaw}{14cm}
\begin{theorem}\Label{LocalUniformization}
$\Phi_{\rho, U}$ is a finite-to-one open mapping.  
\end{theorem}

\end{boxedlaw}

\begin{proof}
By Corollary \ref{DicreteIntersection}, the fiber of $\Phi$ is discrete.
In addition, $\Phi$ is holomorphic and $\dim U_\rho = 2 \dim \TT$.
Therefore, by Theorem \ref{DiscreteFiberThenOpen}, $\Phi$ is an open map. 
\end{proof}

\section{Opposite orientations}\Label{sOppositeOrientation} 
In this section, when the orientations of the Riemann surfaces are opposite, we show the discreteness of $\rchi_X \cap \rchi_Y$  analogous to Theorem \ref{NonemptyDiscreteInteresectionOppositelyOriented} and the local uniformization theorem analogous to Theorem \ref{LocalUniformization}. 
\begin{boxedlaw}{13cm}
\begin{theorem}\Label{NonemptyDiscreteInteresectionOppositelyOriented}
Fix $X \in \TT$ and $Y \in \TT^\ast$ .
Then, $\rchi_X \cap \rchi_Y$ is a non-empty discrete set. 
\end{theorem}
\end{boxedlaw}
Since the proof is similar to the case when the orientations coincide,  we simply outline the proof, yet explain how some parts are modified. 
We leave the details to the reader.

Recall that we have constructed compatible train track decomposition regardless of the orientation of $X Y$ (\S \ref{sProjectiveTraintrackForX}, \S \ref{sProjectiveTraintrackX}). 
In summary, we have the following (in the case of opposite orientataions): 
\begin{proposition}\Label{CircuarTraintrackDecompositionsOppositeOrientataion}
Fix $X \in \TT$ and $Y \in \TT^\ast$. 
For every $\ep > 0$, there is a bounded subset $K_\ep$ in $\rchi_X \cap \rchi_Y$, such that, 
if $\rho \in \rchi_X \cap \rchi_Y \minus K_\ep$, then 
there are circular polygonal train-track decompositions $\TTT_{X, \rho}$  of $C_{X, \rho}$ and $\TTT_{Y, \rho}$ of $C_{Y, \rho}$,  such that
\begin{itemize}
\item $\TTT_{X, \rho}$ and $\TTT_{Y, \rho}$ are semi-diffeomorphic, and 
\item $\TTT_{X, \rho}$ and $\TTT_{Y, \rho}$  are  $(1 + \ep, \ep)$-quasi-isometric to the train track decompositions $\tT_{X, \rho}$ of the flat surface $E_{X, \rho}^1$  and $\tT_{Y, \rho}$ of the flat surface $E_{Y, \rho}^1$, respectively, with respect to the normalized metrics. 
\end{itemize}
\end{proposition}
In the case when the orientation of $X$ and $Y$ are the same, in Theorem \ref{WeightedTraintrack}, we constructed a $\Z$-weighted train-track graph representing the  ``difference'' of projective structures on $X$ and $Y$ with the same holonomy. 
As the orientations of $X$ and $Y$ are different,  we shall construct a $\Z$-weighted train track graph representing, in this case, the ``sum'' of the $\CP^1$-structures on $X$ and $Y$ with the same holonomy.

Let $\TTT _{X, \rho}$ and $\TTT_{Y, \rho}$ be circular train track decompositions of $C_{X, \rho}$ and $C_{Y, \rho}$ given by Proposition.
\ref{CircuarTraintrackDecompositionsOppositeOrientataion}
Let $\{h_{X, 1},  h_{X_2}, \dots, h_{X_n}\}$ be the horizontal edges of $\TTT_{X, \rho}$. 
Similarly to \S\ref{sWeightedTraintarck}, we first define the $\Z$-valued function on the set of horizontal edges.
For each $i = 1, \dots, n$,  let $c_i$ be the round circle on $\CP^1$ supporting the development of $h_{X, i}$. 
First, suppose that $h_{X, i}$ corresponds to an edge $h_{Y, i}$ of $\TTT_{Y, \rho}$ by the collapsing map $\TTT_{X, \rho} \to \TTT_{Y, \rho}$. 
Since $\TTT_{X, \rho}$ and $\TTT_{Y, \rho}$ are compatible, the corresponding endpoints of $h_{X, i}$ and $h_{Y, i}$ map to the same point on $c_i$  by their developing maps.
Thus, by identifying the endpoints, we obtain a covering map from a circle $h_{X, i} \cup h_{Y, i}$ onto  $c_i$. 
Then, define  $\gamma_{X, i}(h_{X, i}) \in \Z_{> 0}$ to be the covering degree.

Next, suppose that $h_{X, i}$ corresponds to a vertex of $\TTT_{Y, i}$. 
Then the endpoints of $h_{X, i}$ develop to the same point on $c_i$.
 The  circle obtained by identifying endpoints of $h_{X, i}$ covers $c_i$. 
Thus, let  $\gamma_{X, i}(h_{X, i}) \in \Z_{> 0}$ be the covering degree.

Similarly to \S \ref{sWeightedTraintarck}, we shall construct a $\Z$-weighted train-track graph $\Gamma_\rho$ immersed in $\TTT_{X, \rho}$. 
On each branch $\PPP_X$ of $\TTT_{X, \rho}$,
we construct a $\Z$-weighted train-track graph  $\Gamma_{\PPP_X}$ on  $\PPP_X$ such that, for every $\ep > 0$, there is a compact subset $K$ of $\rchi$ satisfying the following: 
\begin{itemize}
\item The endpoints of $\Gamma_{\PPP_X}$ are on horizontal edges of $\PPP_X$. 
\item  They agree with $\gamma_{X, \rho}$ along the horizontal edges.
\item If $\PPP_X$ is a transversal branch, then, for $\rho \in \rchi_X \cap \rchi_Y \minus K$, then $2\pi \Gamma_{\PPP_X} (\alpha)$ is $( 1+ \ep, \ep)$-quasi-isometric to $\sqrt{2}(V_{X, \rho} | P_X) (\alpha) +\sqrt{2} (V_{Y, \rho} | P_Y)(\alpha)$ for all $\alpha \in \AAA(\PPP_X)$. 
\item For every smooth isotopy class of a staircase surface $B$ on a flat surface homeomorphic to $S$ and every arc $\alpha$ on $B$ connecting points on horizontal edges or vertices, there is a positive constant $q(B, \alpha)$ such that, if $\TTT_{X, \rho}$ contains a non-transversal branch $\BBB_X$ smoothly isotopic to $B$ on $S$, then 
  $2\pi\Gamma_{\BBB_X} (\alpha)$ is $( 1+ \ep, q(B, \alpha))$-quasi-isometric to $\sqrt{2}\, V_{X, \rho} | B_X  (\alpha)+ \sqrt{2}\, V_{Y, \rho} | B_Y (\alpha)$.
  \end{itemize}
   
\begin{theorem}\Label{PropoeriesOfIntersectionFunctionOpppositeOrientation}
Let $X \in \TT$ and $Y \in \TT^\ast$.
For every $\ep > 0$, there is a bounded subset $K_\ep \sub \rchi_X \cap \rchi_Y$ such that,
for every $\rho \in \rchi_X \cap \rchi_Y \minus K_\ep$, there is a $\Z$-weighted graph $\Gamma_\rho$ carried in $\TTT_{X, \rho}$ such that 

\begin{enumerate}
\item the induced cocycle $[\Gamma_\rho]\col \CC \to \Z$ changes continuously in $\rho \in \rchi_X \cap \rchi_Y \minus K$, \Label{iContinuityCocycle}
\item for every loop $\alpha$ on $S$, there is  $q_\alpha > 0$, such that  $2\pi \Gamma_\rho (\alpha)$ is $(1 + \ep, q_\alpha)$-quasi-isometric to $\sqrt{2}( V_{X, \rho} (\alpha)+ V_{Y, \rho} (\alpha)) $ for all $\rho \in \rchi_X \cap \rchi_Y \minus K$. \Label{iTraintrackWeight}
\end{enumerate}
\end{theorem}
\begin{proof}
The proof is similar to 
Theorem 
\ref{WeightedTraintrack} (\ref{iContinuousTraintrack}) and Theorem \ref{VerticalFoliationsAndGraftingFunction}.%
\end{proof}

Then, Theorem \ref{PropoeriesOfIntersectionFunctionOpppositeOrientation} implies, similarly to Theorem \ref{Bounded-Intersection},  the following:
\begin{theorem}\Label{BoundedComponentsOppositelyOriented}
Each connected component of 
$\rchi_X \cap \rchi_Y$ is bounded. 
\end{theorem}

\section{The completeness}\Label{sCompleteness}
In this section, we prove the completeness in \Cref{GeneralizedQF}.
Let $Q$ be a connected component of the Bers' space $\BB$; then $Q$ is a complex submanifold of $(\PP \sqcup \PP^\ast) \times (\PP \sqcup \PP^\ast)$, and   $\dim_\C Q = 6g-6$. 
We call that $\psi\col \PP \sqcup \PP^\ast \to \TT \sqcup \TT^\ast$ is the uniformization map and  $\Psi\col Q \to (\TT \sqcup \TT^\ast)^2$ is defined by $\Psi(C, D) = (\psi(C), \psi(D))$.
Then, by Theorem \ref{LocalUniformization}, $\Psi$ is an open holomorphic map. 
In this section, we prove the completeness of $\Psi$.

\begin{lemma}\Label{LocalPathLifting}
The open map $\Psi\col Q \to (\TT \sqcup \TT^\ast)^2$ has a {\sf local path lifting property}.
That is, for every $z \in Q$, there is a neighborhood $W$ of $\Psi(z)$ such that if  path $\alpha_t, \, 0\leq t \leq 1$ in $W$ satisfies $\zeta(z) = \alpha_0$, then there is a lift $\ti\alpha_t$ of $\alpha_t $ to $Q$ with $\ti\alpha_ 0 = z$.
\end{lemma}

\begin{proof}
Since $\Psi$ is an open map and $\dim Q = \dim (\TT \sqcup \TT^\ast)^2$, $\Psi$ is a locally branched covering map. 
Then, for every $z \in Q$, there is an open neighborhood $V$ of $z$ in $Q$ and a finite group $G_z$ biholomorphically acting $V$, such that $\Psi$ is $G_z$-invariant, and $\Psi$ induces the biholomorphic map $V/ G_z \to  \Psi(V)$.

For $g \in G_z \minus \{id\}$, let $F_g \sub W$ be the (pointwise) fixed point set of $g$. 
Clearly $F_g$ is a proper analytic subset, and thus $F \coloneqq \cup_{g \in G_z \minus \{id\}} F_g$ is an analytic subset strictly contained in $W$.

For every path $\alpha\col[0,1] \to V$ with $\alpha(0) = \Psi(z)$,
we can take a one-parameter family  of paths $\alpha_t \,(t \in [0,1])$ in $W$ with $\alpha_t(0) = \Psi(z)$
such that  $\alpha_1= \alpha$ and, for $t < 1$, $\alpha_t$ is disjoint from $\Psi(
F)$ (since $\Psi (F)$ has complex codimension, at least,  one.)

Then, for $t < 1$,  $\alpha_t$ continuously lifts a path $\ti\alpha_t\col [0,1] \to Q$, and $\alpha_t$ converges to a desired lift of $\alpha_1$ as $t \to 1$. 
\end{proof}

Now we are ready to prove the completeness. 
\begin{theorem}\Label{Surjectivity}
$\Psi\col Q \to (\TT \sqcup \TT^\ast)^2 \minus \Delta$ is a complete map, where $\Delta$ is the diagonal set.
\end{theorem}

The completeness of Theorem \ref{Surjectivity} immediately implies the following:
\begin{corollary}\Label{NontrivialINtersection}
 $\Psi\col Q \to (\TT \sqcup \TT^\ast)^2 \minus \Delta$ is surjective onto a connected component of $(\TT \sqcup \TT^\ast)^2 \minus \Delta$.
\end{corollary}
\proof[Proof of Theorem \ref{Surjectivity}]
By Lemma \ref{LocalPathLifting},  $\Psi$ has a local path lifting property.
Thus, suppose that $(X_t, Y_t) \col [0,1] \to (\TT \sqcup \TT^\ast)^2 \minus \Delta ~(0 \leq t \leq 1)$ be a path and there is a (partial) lift $(C_t, D_t) \col [0,1) \to Q$  of the path $(X_t, Y_t)$.
For each $t \in [0,1]$, let $\rho_t \in \rchi$ denote the common holonomy of $C_t$ and $D_t$. 
By the continuity, the orientations of Riemann surfaces in the pairs obviously remain the same along $\rho_t$ for all $t > 0$. 

First we, in addition, suppose that there is an increasing sequence $0 \leq t_1 < t_2 < \dots$ converging to $1$, such that  $\rho_{t_i}$  converges to a representation in $\rchi$.  
By Corollary \ref{DicreteIntersection} or, in the case of opposite orientations, Theorem \ref{NonemptyDiscreteInteresectionOppositelyOriented},  $\rchi_{X_1} \cap \rchi_{Y_1}$ is a discrete subset of $\rchi$. 
Then, since $(X_t, Y_t)$ converges to a point $(X_1, Y_1)$ in $\TT \times \TT$, and $\rchi_{X_t}$ and $\rchi_{Y_t}$ change continuously in  $t \in [0,1]$,  every neighborhood $\rchi_{X_1} \cap \rchi_{Y_1}$ contains $\rho_{t_i}$ for all sufficiently large $i$. 
Thus the sequence $\rho_{t_i}$ converges to a point in $\rchi_{X_1} \cap \rchi_{Y_1}$.
Since $\rchi_{X_t} \cap \rchi_{Y_t}$ is a discrete set in $\rchi$ which continuously changes in $t \in [0,1]$, 
we indeed have a genuine convergence. 
\begin{lemma}
$\rho_t$ converges to $\rho_1$ as $t \to 1$. 
\end{lemma}
By this lemma, $(C_t, D_t)$ converges in $(\PP \sqcup \PP^\ast) \times (\PP \sqcup \PP^\ast)$ as $t \to 1$, so that the partial lift $(C_t, D_t)$ extends to $t = 1$. 

Thus it is suffices  to show the addition assumption always holds: 
\begin{proposition}
There is a compact subset $K$ in $\rchi$, such that, for every $t > 0$, there is $t' > 0$ such that $\rho_{t'} \in K$.
\end{proposition}

\begin{proof}
The proof is essentially the same as the proof of 
 Theorem \ref{Bounded-Intersection} or Theorem \ref{NonemptyDiscreteInteresectionOppositelyOriented}, which states that each component of $\rchi_X \cap \rchi_Y$ is a bounded subset of $\rchi$. 

For $0 \leq t < 1$,
let $V_{C_t}$ and $V_{D_t} $ be the vertical measured foliations of $C_t$ and $D_t$, respectively.
Suppose, to the contrary, that $\rho_t$ leaves every compact subset of $\rchi$. 
As $(X_t, Y_t)$ converges to $(X, Y) \in (\TT \sqcup \TT^\ast)^2 \minus \Delta$ and $\Hol (C_t) = \Hol(D_t)$,  similarly to Theorem \ref{WeightedTraintrack} or \Cref{PropoeriesOfIntersectionFunctionOpppositeOrientation}, for $t$ close to $1$, we can construct a $\Z$-weighed train track $\Gamma_t$ on $S$, such that 
\begin{itemize}
\item the intersection function $[\Gamma_t]\col \CC \to \Z$  is continuous in $t$, and
\item for every closed curve $\alpha$ on $S$, there is a constant $q_\alpha > 0$ and a function $\ep_t > 0$ converging to $0$, such that, for all sufficiently large $t > 0$,  $[\Gamma _t] (\alpha)$ is $(1 + \ep_t, q_\alpha)$-quasi-isometric to $V_{C_t} (\alpha) - V_{D_t} (\alpha)$ if the orientation of $X$ and $Y$ are the same and to $V_{C_t} (\alpha) + V_{D_t} (\alpha)$ if the orientation of $X$ and $Y$ are different. 
\end{itemize}
The first condition implies that the intersection number is constant in $t$, whereas the second condition implies that the intersection number with some closed curve $\alpha$ diverges to infinity as $t \to 1$. 
This is a contradiction.
\end{proof}

\Qed{Surjectivity}

  Last we remark the behavior of $\Hol$ near the diagonal $\Delta$. 
\begin{proposition}\Label{NearDiagonal}
Let $(C_t, D_t),  \,0 \leq  t < 1$ be a path in $\BB$, such that $\Psi (C_t, D_t)$ converges to a diagonal point $(X, X)$  of  $(\TT \sqcup \TT^\ast)^2$. 
Then $\Hol C_t = \Hol D_t$ leaves every compact set in $\rchi$ as $t \to 1$. 
\end{proposition}
\begin{proof}
Fix arbitrary $X \in \TT \sqcup \TT^\ast$ and a bounded open subset $U$ in  $\chi$.
Recall that $\Hol \col \PP \sqcup \PP^\ast \to \rchi$ is a locally biholomorphic map, 
and that,  for all $Y \in \TT \sqcup \TT^\ast$,  the set $\PP_Y$ of $\CP^1$-structures on $Y$ is properly embedded in $\chi$ by $\Hol$.
Therefore, if an open neighborhood $V$ of $X$ in $\TT \sqcup \TT^\ast$  is sufficiently small, then, letting  $\PP_V$ denotes all $\CP^1$-structure on Riemann surfaces in $V$,  $\PP_V \cap \Hol^{-1} (U)$ embeds into $U$ by the holonomy map $\Hol$.  
In particular, for every $Y, Z \in V$, $\rchi_Y \cap \rchi_Z$ is disjoint from $U$. 
Then the assertion is immediate. 
\end{proof}
\subsection{Cardinalities of the intersections}
By the surjectivity of \Cref{NontrivialINtersection}  and the existence of non-quasi-Fuchsian components of $\BB$ in $\PP \sqcup \PP^\ast$ in \Cref{ProjectiveQuasiFuchsianManifold}, we immediately have the following: 
\begin{corollary}\Label{Cardinality}
Let $X, Y$ be distinct marked Riemann surface structures on $S$ with any orientations. 
If the orientations of $X$ and $Y$ are opposite, the intersection  $\rchi_X \cap \rchi_Y$ contains, at least, two points,  if the orientations of $X$ and $Y$ are the same, $\rchi_X \cap \rchi_Y$ contains, at least, one point.
\end{corollary}

\section{A proof of the simultaneous uniformization theorem}\Label{SimultaneousUniformization}
In this section, using Theorem \ref{GeneralizedQF}, we give a new proof of the simultaneous uniformization theorem without using the measurable Riemann mapping theorem. 
Recall $\QF$ is the quasi-Fuchsian space, and it is embedded in $\BB/\Z_2$.  

  Given $(C, D)$ in $\QF$, the universal covers $\ti{C}, \ti{D}$ are the connected components of $\CP^1$ minus its equivariant Jordan curve equivariant via $\Hol C = \Hol D$.  
 
    \begin{lemma}\Label{OpenAndClosed}
$\QF$ is a union of connected components of $\BB/\Z_2$.
\end{lemma}
\begin{proof}
 As being a quasi-isometric embedding is an open condition, $\QF$ is an open subset of $\BB$. 
 Thus, it suffices to show that $\QF$ is closed.

Let $(C_i, D_i) \in \PP \times \PP^\ast$ be a sequence in $\QF$ which converges to $(C, D) \in \PP \times \PP^\ast$.
Let $\rho_i\col \pi_1(S) \to \PSL(2, \C)$ be the   quasi-Fuchsian representation of $C_i$ and $D_i$. 
We show that the  holonomy $\rho\col \pi_1(S) \to \PSL(2, \C)$ of the limits $C$ and $D$ is also quasi-Fuchsian.  
Let  $\ti{C}_i$ and $\ti{D}_i$ be the universal covers of $C_i$ and $D_i$, respectively. 
Then $\ti{C}_i$ and $\ti{D}_i$ are the components of  $\CP^1$ minus the $\rho_i$-equivariant Jordan curve.
Let $f_i \col \ti{C_i} \cup \s^1 \to \CP^1$ and $g_i \col  \ti{D_i} \cup \s^1 \to \CP^1$ be the extensions of the embeddings to their boundary circles by a theorem of  Carath\'eodory. 
Let $h_i$ be the homeomorphism $ \ti{C}_i \cup \s^1 \cup \ti{D}_i \to \CP^1$.

Since embeddings $\dev C_i$ converge to $\dev C$ uniformly on compact as $i \to \infty$,  the limit  $\dev C$ is also an embedding. 
(Suppose, to the contrary, that  $\dev C\col \ti{C} \to \CP^1$ is not embedding. 
 Then there are distinct open subsets in $\ti{C}$ which homeomorphically map onto the same open subset in $\CP^1$. Then  $\dev C_i$ is also not embedding for all sufficiently large $i$ against the assumption.)
Thus, by the convergence of corresponding convex pleated surfaces in $\H^3$,   the equivariant property implies that $\dev C$  extends to the boundary circle continuously and equivariantly. 
Similarly, since the embedding $\dev D_i$ converges to $\dev D$, then $\dev D$ is also an embedding,  and $\dev D$  extends to the boundary circle continuously and equivariantly. 
 Therefore $h_i$ converges to a continuous map $$h\col \s^2 \cong \ti{C} \cup \s^1 \cup \ti{D} \to \CP^1$$ such that the restriction of $h$ to $\ti{C} \sqcup \ti{D}$ is an embedding. 

The domain and the target of $h$ are both homeomorphic to $\s^2$. 
Therefore, if $h | \s^1$ is not a Jordan curve on $\CP^1$, then there is a point $z \in \CP^1$, such that $h^{-1}(z)$ is a single segment of $\s^1$. 
By the equivariant property, $h^{-1} (z) = \s^1$, and $\Im\, h$ is a  wedge of two copies of $\s^2$, which is a contradiction. 
\end{proof}
The following asserts that the diagonal of $\TT \times \TT^\ast$ corresponds to the Fuchsian representations. 
\begin{lemma}\Label{DegreeOneOverFuchsian}
Let $X \in \TT$. 
Let $\eta\col \pi_1(S) \to \PSL(2, \C)$ be a quasi-Fuchsian representation, such that the ideal boundary of $\H^3 / \Im \rho$ realizes the marked Riemann surface $X$ and its complex conjugate $X^\ast \in \TT^\ast$.
Then $\eta$ is the Fuchsian representation $\pi_1(S) \to \PSL(2, \R)$ such that $X = \H^2/ \Im \eta$. 
\end{lemma}
\begin{proof}
By the Riemann uniformization theorem, the universal covers  $\ti{X}$ and $\ti{X}^\ast$  are the upper and the lower half planes. 
Then, by identifying their ideal boundaries equivariantly, we obtain $\CP^1$ so that the universal covers   $\ti{X}$ and  $\ti{X}^\ast$ are round open disks. 
Let $(C, D) \in \PP \times \PP^\ast$ be the pair corresponding to $\eta$, such that the complex structure of $C$ is  $X$ and the complex structure of  $D$ is on $X^\ast$.    

On the other hand, the universal covers $\ti{C}$ and $\ti{D}$ are connected components of $\CP^1 \minus \Lambda(\eta)$, where $\Lambda(\eta)$ is the $\eta$-equivariant Jordan curve in $\CP^1$.  
Thus, there is a $\eta$-equivariant homeomorphism $\phi \col \CP^1 \to \CP^1$, such that
$\phi$ restricts to a  biholomorphism from $\CP^1 \minus {\R \cup \{\infi\}} \to \CP^1 \minus \Lambda(\eta)$.

Then, by Morera's theorem for the line integral along triangles (see \cite{Stein03} for example), $\phi$ is a genuine  biholomorphic map $\CP^1 \to \CP^1$. 
Therefore $\phi$ is a Möbius transformation, and therefore $\eta$ is conformally conjugate to the Fuchsian representation uniformizing $X$.
\end{proof}
\begin{proposition}\Label{ConnectedQF}
$\QF$ is a single connected component of $\BB/\Z_2$.
\end{proposition}
\begin{proof} 
By Lemma \ref{OpenAndClosed}, $\QF$ is the union of some connected components of $\BB$. 
 By Theorem \ref{GeneralizedQF}, for every component  $Q$ of $\QF$, the image $\Psi (Q)$ contains the diagonal $\{ (X, X^\ast)\}$ of $\TT \times \TT^\ast$.
Then, by Lemma \ref{DegreeOneOverFuchsian}, every diagonal pair $(X, X^\ast) \in \TT \times \TT^\ast$ corresponds to a unique point in $\QF$. 
Therefore $\QF$ is connected. 
   \end{proof}

Last we reprove the simultaneous uniformization theorem.
 \begin{theorem}\Label{QuasifuchsianComponent}
  $\QF$ is biholomorphic to $ \TT \times \TT^\ast$ by $\Psi$. 
\end{theorem}
\begin{proof}
By Theorem \ref{GeneralizedQF}, $\Psi$ is a complete local branched covering map. 
Since $\Psi$ is surjective, 
by Lemma \ref{DegreeOneOverFuchsian}, $\Psi$ is a degree-one  over the diagonal $\{(X, \ol{X} ) \mid  X \in \TT \}$, and the diagonal corresponds to the Fuchsian space. 

The set of ramification points of $\Psi$ is an analytic set. 
The Fuchsian space is a totally real subspace of dimension $6g -6$. 
Therefore, if the ramification locus contains the Fuchsian space, then the locus must have the complex dimension $6g -6$, the full dimension. 
This is a contradiction as $\Psi$ is a locally branched covering map. 
Therefore, $\QF \to \TT \times \TT^\ast$ has degree one, and thus it is biholomorphic.  
\end{proof} 

\bibliography{SimultaneousUniformizationAndHolonomyVarieties.bbl}
\bibliographystyle{alpha}

\end{document}